\documentclass[11pt,reqno]{amsart}
\usepackage{geometry}               
\usepackage{graphicx}
\usepackage{amssymb}
\usepackage{mathabx}
\usepackage{pdfsync}
\usepackage{bbm} 
\usepackage{hyperref}
\usepackage{verbatim} 
\usepackage{epstopdf}
\usepackage{amsrefs}

\def\R{\mathbb{R}}

\def\N{\mathbb{N}}
\def\Z{\mathbb{Z}}
\def\Co{\mathbb{C}}

\newcommand{\I}{\mathcal{I}}

\def\f{\widehat{f}}

\newcommand{\ds}{\displaystyle}

\newcommand{\pvector}[1]{
  \begin{pmatrix}
    #1
  \end{pmatrix}} %

\newcommand{\ddirac}[1]{
  \,\boldsymbol{\delta}\!\pvector{#1}\!} %

\renewcommand{\d}{\,{\rm d}} 
\newcommand{\one}{\mathbbm 1}
\newcommand{\weak}{\rightharpoonup}

\newcommand{\la}{\lambda}
\newcommand{\vphi}{\varphi}
\newcommand{\vrho}{\varrho}
\newcommand{\eps}{\varepsilon}

\providecommand{\ab}[1]{\vert #1\vert}
\providecommand{\abs}[1]{\Bigl\vert #1 \Bigr\vert}

\providecommand{\norma}[1]{\Vert #1 \Vert} 
\providecommand{\Norma}[1]{\Bigl\Vert #1 \Bigr\Vert}
\providecommand{\NOrma}[1]{\biggl\Vert #1 \biggr\Vert}

\newcommand{\supp}{\operatorname{supp}}

\theoremstyle{plain}
\newtheorem{theorem}{Theorem}[section]
\newtheorem{corollary}[theorem]{Corollary}
\newtheorem{lemma}[theorem]{Lemma}
\newtheorem{proposition}[theorem]{Proposition}
\newtheorem{conjecture}[theorem]{Conjecture}
\theoremstyle{definition}
\newtheorem{definition}[theorem]{Definition}
\newtheorem{remark}[theorem]{Remark}

\numberwithin{equation}{section}

\title[Sharp Strichartz inequalities]{Sharp Strichartz inequalities for fractional and higher order Schr\"odinger equations}

\author[Brocchi]{Gianmarco Brocchi}

\address{
        Gianmarco Brocchi:
        School of Mathematics\\
        University of Birmingham\\
        Edgbaston\\
        Birmingham\\
        B15 2TT\\
        England}
\email{g.brocchi@pgr.bham.ac.uk}

\author[Oliveira e Silva]{Diogo Oliveira e Silva}
\address{
        Diogo Oliveira e Silva:
        School of Mathematics\\
        University of Birmingham\\
        Edgbaston\\
        Birmingham\\
        B15 2TT\\
        England,
        and
Hausdorff Center for Mathematics, Endenicher Allee 62, 53115 Bonn, Germany}
\email{d.oliveiraesilva@bham.ac.uk}
\thanks{\it D.O.S. was  supported by the Hausdorff Center for Mathematics and DFG grant CRC 1060.}

 \author[Quilodr\'an]{Ren\'e Quilodr\'an}
\address{Ren\'e Quilodr\'an}
\email{rquilodr@dim.uchile.cl}

\begin{document}

\begin{abstract}
We investigate a class of sharp Fourier extension inequalities on the planar curves $s=|y|^p$,  $p>1$.
We identify the mechanism responsible for the possible loss of compactness of nonnegative extremizing sequences, and prove that extremizers exist if $1<p<p_0$, for some $p_0>4$.
In particular, this resolves the dichotomy of Jiang, Pausader \& Shao concerning the existence of extremizers for the Strichartz inequality for the fourth order Schr\"odinger equation in one spatial dimension.
One of our tools is a geometric comparison principle for $n$-fold convolutions of certain singular measures in $\R^d$, developed in the companion paper \cite{OSQ18}.
We further show that any extremizer exhibits fast $L^2$-decay in physical space, and so its Fourier transform can be extended to an entire function on the whole complex plane. 
Finally, we investigate the extent to which our methods apply to the case of the planar curves $s=y|y|^{p-1}$,  $p>1$.
\end{abstract}
\maketitle
\tableofcontents

\section{Introduction}

Gaussians are known to extremize certain Strichartz estimates in low dimensions.
Consider, for instance, the Strichartz inequality for the homogenous Schr\"odinger equation in $d$ spatial dimensions,
\begin{equation}\label{eq:StrichartzSchrodinger}
\|e^{-it\Delta} f\|_{L_{x,t}^{2+\frac4d}(\R^{d+1})}\leq {\bf S}(d) \|f\|_{L^2(\R^d)},
\end{equation}
with optimal constant given by
\begin{equation}\label{eq:BestConstantStrSchr}
{\bf S}(d):=\sup_{0\neq f\in L^2}\frac{\|e^{-it\Delta} f\|_{L_{x,t}^{2+\frac4d}(\R^{d+1})}}{\|f\|_{L^2(\R^d)}}.
\end{equation}
That ${\bf S}(d)<\infty$ is  of course due to the original work of Strichartz \cite{St77}, which in turn had precursors in \cite{To75, Se76}.
If $d\in\{1,2\}$, then  Gaussians extremize  \eqref{eq:StrichartzSchrodinger}, and therefore ${\bf S}(1)=12^{-1/{12}}$ and ${\bf S}(2)=2^{-1/2}$.
This was originally established by Foschi \cite{Fo07} and Hundertmark \& Zharnitski \cite{HZ06}, and alternative proofs were subsequently given by Bennett et al. \cite{BBCH09, BBI15} and Gon\c{c}alves \cite{Go17}.
All of these approaches ultimately rely on the fact that the Strichartz exponent $2+\frac4d$ is an even integer if $d\in\{1,2\}$, which in turn allows to recast inequality \eqref{eq:StrichartzSchrodinger} in convolution form. 
This is a powerful technique that has proved very successful in tackling a number of problems in sharp Fourier restriction theory, see the recent survey \cite{FOS17} and the references therein.

In recent work of the second and third authors \cite{OSQ16}, we explored the convolution structure of a family of Strichartz inequalities for higher order Schr\"odinger equations in two spatial dimensions in order to answer a question concerning the existence of extremizers that had appeared in the previous literature.
Our purpose with the present work is three-fold.
Firstly, we resolve the dichotomy from \cite{JPS10} concerning the existence of extremizers for the Strichartz inequality for the fourth order Schr\"odinger equation in one spatial dimension.
This is related to the Fourier extension problem on the planar curve $s=y^4$.
Secondly, we study similar questions in the more general setting of the Fourier extension problem on the curve $s=|y|^p$, for arbitrary $p>1$.
We also consider {\it odd} curves $s=y\ab{y}^{p-1}$, $p>1$, the case $p=3$ relating to the Airy--Strichartz inequality \cite{FV18, FS17, Sh09}.
 Lastly, we study super-exponential decay and analyticity of the corresponding extremizers and their Fourier transform via a bootstrapping procedure.\\

In \cite{JPS10}, Jiang, Pausader \& Shao considered the fourth order Schr\"odinger equation with $L^2$ initial datum in one spatial dimension,
\begin{equation}\label{eq:4thOrderSchroedinger}
   \begin{cases}
   i \partial_t u-\mu\partial_x^2 u+\partial_x^4 u=0,\quad (x,t)\in\R\times\R,\\
     u(\cdot,0)=f\in L_x^2(\R),
  \end{cases} 
\end{equation}
where $u:\R\times\R\to\Co$, and $\mu\geq 0$.
By scaling, one may restrict attention to $\mu\in\{0,1\}$.
The solution of the Cauchy problem \eqref{eq:4thOrderSchroedinger} can be written in terms of the propagator 
$$u(x,t)=e^{it(\partial_x^4-\mu\partial_x^2)}f(x)=\frac{1}{2\pi}\int_\R e^{ix\xi} e^{it(\xi^4+\mu\xi^2)}\widehat{f}(\xi)\d\xi,$$
where the spatial Fourier transform is defined as\footnote{The Fourier transform will occasionally be denoted by $\mathcal{F}(f)=\widehat{f}$.}
$\widehat{f}(\xi):=\int_\R e^{-ix\xi} f(x) \d x$.
The solution disperses as $|t|\to\infty$, and consequently the following Strichartz inequality due to Kenig, Ponce \& Vega \cite[Theorem 2.1]{KPV91} holds:\footnote{Given $\mu\in\{0,1\}$ and $\alpha\in\R$, we follow the notation from  \cite{JPS10} and denote by $D_\mu^\alpha$  the differentiation operator
$$D_\mu^\alpha f(x):=\frac1{2\pi}\int_\R e^{ix\xi} (\mu+6\xi^2)^{\frac{\alpha}2} \widehat{f}(\xi)\d\xi.$$}
\begin{equation}\label{eq:Strichartz4thOSchr}
\|D_\mu^{\frac13} e^{it(\partial_x^4-\mu\partial_x^2)} f\|_{L_{x,t}^6(\R^{1+1})}\lesssim \|f\|_{L^2(\R)}.
\end{equation}
The main result of \cite{JPS10} is a linear profile decomposition for equation \eqref{eq:4thOrderSchroedinger}, which uses a refinement of the Strichartz inequality \eqref{eq:Strichartz4thOSchr} in the scale of Besov spaces, together with  improved localized Fourier restriction estimates.
As a consequence, the authors of \cite{JPS10} establish a dichotomy result for the existence of extremizers for  \eqref{eq:Strichartz4thOSchr} when $\mu=0$, which can  be summarized as follows: 
Consider the sharp inequality in multiplier form
\begin{equation}\label{eq:SharpStr4thOSchrMu0}
\|D_0^{\frac13} e^{it\partial_x^4} f\|_{L_{x,t}^6(\R^{1+1})}\leq {\bf M} \|f\|_{L^2(\R)},
\end{equation}
with optimal constant given by
\begin{equation}\label{eq:BestConstantStr4thOSchr}
{\bf M}:=\sup_{0\neq f\in L^2}\frac{\|D_0^{\frac13}e^{it\partial_x^4} f\|_{L_{x,t}^6(\R^{1+1})}}{\|f\|_{L^2(\R)}}.
\end{equation}
Then \cite[Theorem 1.8]{JPS10} states that either an extremizer for \eqref{eq:SharpStr4thOSchrMu0} exists, or there exist a sequence $\{a_n\}\subset\R$ satisfying $|a_n|\to\infty$, as $n\to\infty$, and a function $f\in L^2$, such that
$${\bf M}=\lim_{n\to\infty}\frac{\|D_0^{\frac13}e^{it\partial_x^4} (e^{ia_nx} f)\|_{L_{x,t}^6(\R^{1+1})}}{\|f\|_{L^2(\R)}}.$$
In the latter case, one necessarily has ${\bf M}={\bf S}(1)$, where ${\bf S}(1)$ denotes the optimal constant  defined in \eqref{eq:BestConstantStrSchr}. 
Our first main result resolves this dichotomy.
\begin{theorem}\label{thm:Thm1}
There exists an extremizer for \eqref{eq:SharpStr4thOSchrMu0}.
\end{theorem}
\noindent Theorem \ref{thm:Thm1} will follow from a more general result which we now introduce.
As noted in \cite[\S 2]{KPV91}, the operator $D_0^{1/3}e^{it\partial_x^4}$ is nothing but a constant multiple of the Fourier transform at the point {$(-x,-t)\in\R^2$} of the singular measure 
\begin{equation}\label{eq:DefSigma}
\d\sigma_4(y,s)=\ddirac{s-y^4} |y|^{\frac 13}\d y\d s
\end{equation}
defined on the curve $s=y^4$.
 As in \cite[\S 6.4]{OSQ16}, one is naturally led to consider generic power curves $s=|y|^p$.
The corresponding inequality is 
\begin{equation}\label{eq:SharpMultiplierFormGenP}
\|{\mathcal M}_p(f)\|_{L_{x,t}^6(\R^{1+1})}\leq {\bf M}_p \|f\|_{L^2(\R)},
\end{equation}
where the multiplier operator ${\mathcal M}_p$ is defined as 
$${\mathcal M}_p(f)(x,t)=D_0^{\frac{p-2}6} e^{it|\partial_x|^p} f(x).$$
Inequality \eqref{eq:SharpMultiplierFormGenP} can be equivalently restated as a Fourier extension inequality,
\begin{equation}\label{eq:SharpExtensionFormGenP}
\|\mathcal{E}_p(f)\|_{L^6(\R^2)}\leq {\bf E}_p \|f\|_{L^2(\R)},
\end{equation}
or in convolution form as
\begin{equation}\label{eq:SharpConvolutionFormGenP}
\|f\sigma_p\ast f\sigma_p\ast f\sigma_p\|_{L^2(\R^2)}\leq {\bf C}_p^3\|f\|^3_{L^2(\R)}.
\end{equation}
Here, the singular measure $\sigma_p$ is defined in accordance with \eqref{eq:DefSigma} by
\begin{equation}\label{eq:DefSigmaMeasure}
\d\sigma_p(y,s)=\ddirac{s-|y|^p} |y|^{\frac {p-2}6}\d y\d s,
\end{equation}
and the Fourier extension operator $\mathcal{E}_p(f)=\mathcal{F}(f\sigma_p)(-\cdot)$ is given by 
\begin{equation}\label{eq:DefFSigmaHat}
\mathcal{E}_p(f)(x,t)=\int_\R e^{ix y}e^{i t |y|^p} |y|^{\frac{p-2}6} f(y)\d y,
\end{equation}
so that $6^{\frac{p-2}{12}}{\mathcal E}_p(\widehat{f})=2\pi{\mathcal M}_p(f)$. 
If $f$ is an extremizer for \eqref{eq:SharpExtensionFormGenP}, then $f$ is likewise an extremizer for \eqref{eq:SharpConvolutionFormGenP}, and $\mathcal{F}^{-1}(f)$ is an extremizer for \eqref{eq:SharpMultiplierFormGenP}.
Thus these three existence problems are essentially equivalent.
The convolution form \eqref{eq:SharpConvolutionFormGenP} also shows that the search for extremizers can be restricted to the class of {\it nonnegative} functions.
An application of Plancherel's Theorem further reveals that the corresponding optimal constants satisfy 
$${\bf E}_p^6=(2\pi)^2{\bf C}_p^6=(2\pi)^3 6^{1-\frac{p}{2}}{\bf M}_p^6.$$

Our next result extends the dichotomy proved in \cite[Theorem 1.8]{JPS10} to the case of arbitrary exponents $p>1$.
It states that one of  two possible scenarios occurs,  {\it compactness} or {\it concentration} at a point.
We make the latter notion precise.
 \begin{definition}\label{def:concentration-point}
	A sequence of functions $\{f_n\}\subset L^2(\R)$ {\it concentrates at a point} $y_0\in\R$ if,
	for every $\eps,\rho>0$, there exists $N\in\N$ such that, for every $n\geq N$, 
	\[ \int_{\ab{y-y_0}\geq \rho}\ab{f_n(y)}^2\d y<\eps\norma{f_n}_{L^2(\R)}^2.\]
\end{definition}

\noindent We choose to phrase our second main result in terms of the convolution inequality \eqref{eq:SharpConvolutionFormGenP} because, as we shall see, condition \eqref{eq:IneqCriticalValueCp} has a very simple geometric meaning in terms of the boundary value of the relevant 3-fold convolution measure.

\begin{theorem}\label{thm:Thm2}
Let $p>1$. If
\begin{equation}\label{eq:IneqCriticalValueCp}
{\bf C}_p^6>\frac{2\pi}{\sqrt{3} p(p-1)},
\end{equation}
then any extremizing sequence of nonnegative functions in $L^2(\R)$ for \eqref{eq:SharpConvolutionFormGenP} is precompact, after normalization and scaling.
In this case, extremizers for \eqref{eq:SharpConvolutionFormGenP} exist.
If instead equality holds in \eqref{eq:IneqCriticalValueCp} then, given any $y_0\in\R$, there exists an extremizing sequence for \eqref{eq:SharpConvolutionFormGenP} which concentrates at $y_0$.
\end{theorem}

\noindent
A few remarks may help to further orient the reader. 
Firstly,  if $p=1$, then the curve $s=|y|$  has no curvature, and no non-trivial Fourier extension estimate can hold. 
 Secondly, if equality holds in \eqref{eq:IneqCriticalValueCp}, then Theorem \ref{thm:Thm2} does {\it not} guarantee the non-existence of extremizers. 
Indeed, ${\bf C}_2^6=\pi/\sqrt{3}$, and Gaussians are known to extremize \eqref{eq:SharpConvolutionFormGenP} when $p=2$.
Various results of a similar flavour to that of Theorem \ref{thm:Thm2} have appeared in the recent literature. 
They are typically derived from a sophisticated application of concentration-compactness techniques \cite{CS12a, Sh16}, a full profile decomposition \cite{JPS10, JSS14, Sh09}, or the missing mass method as in \cite{FLS16, FS17}.
We introduce a new variant which follows the spirit of the celebrated works of Lieb \cite{BL83, Li83} and Lions \cite{Li84,Li84+}. 
It seems more elementary, and may be easier to adapt to other manifolds. 
The proof of Theorem \ref{thm:Thm2} involves a variant of Lions' concentration-compactness lemma \cite{Li84}, a variant of the corollary of the Br\'ezis--Lieb lemma from \cite{FVV11},  bilinear extension estimates, and a refinement of inequality \eqref{eq:SharpExtensionFormGenP} over a suitable {\it cap space}.

In a range of exponents that includes the case $p=4$,
we are able to resolve the dichotomy posed by Theorem \ref{thm:Thm2}. 

\begin{theorem}\label{thm:Thm3}
There exists $p_0>4$ such that, for every $p\in(1,p_0)\setminus\{2\}$, the strict  inequality \eqref{eq:IneqCriticalValueCp}
holds. In particular,  if $p\in (1,p_0)$, then there exists an extremizer for \eqref{eq:SharpConvolutionFormGenP}.
\end{theorem}

\noindent Our method yields $p_0\approx 4.803$  with 3 decimal places, and effectively computes arbitrarily good lower bounds for the ratio of $L^2$-norms in \eqref{eq:SharpConvolutionFormGenP} via expansions of suitable trial functions in the orthogonal basis of Legendre polynomials. 
We remark that the value $p_0\approx 4.803$ is suboptimal, in the sense that a natural refinement of our argument allows to increase this value to $\approx 5.485$, see \S \ref{sec:other-powers-v2} below.

Once the existence of extremizers has been established, their properties are typically deduced from the study of the associated Euler--Lagrange equation.
Following this paradigm, we show that any extremizer of \eqref{eq:SharpExtensionFormGenP} decays super-exponentially fast in $L^2$, which reflects the analiticity of its Fourier transform.
This is the content of our next result.

\begin{theorem}\label{thm:Thm4}
Let $p>1$.
If $f$ is an extremizer for \eqref{eq:SharpExtensionFormGenP}, then there exists $\mu_0>0$, such that
$$x\mapsto e^{\mu_0|x|^p} {f}(x)\in L^2(\R).$$
In particular, its Fourier transform $\widehat{f}$ can be extended to an entire function on $\Co$. 
\end{theorem}
\noindent Note that the exponent $\mu_0$ necessarily depends on the extremizer itself, see the discussion in \cite[p. 964]{CS12b}.
The proof relies on a bootstrapping argument that found similar applications in \cite{CS12b, EHL11, HS12, Sh16b}.

To some extent, our methods are able to handle the case of the planar {\it odd} curves $s=y|y|^{p-1}$, $p>1$.
Define the singular measure
\begin{equation}\label{eq:DefMuMeasure}
\d\mu_p(y,s)=\ddirac{s-y|y|^{p-1}} |y|^{\frac {p-2}6}\d y\d s.
\end{equation}
The associated Fourier extension operator
${\mathcal S}_p(f)=\mathcal{F}(f\mu_p)(-\cdot)$, 
defined in \eqref{eq:odd-fourier-extension-operator} below,
satisfies the estimate $\|{\mathcal S}_p(f)\|_{L^6}\lesssim \|f\|_{L^2}$. 
In sharp convolution form, this can be \mbox{rewritten as}
\begin{equation}\label{eq:OddSharpConvolutionFormGenP}
\|f\mu_p\ast f\mu_p\ast f\mu_p\|_{L^2(\R^2)}\leq {\bf Q}_p^3\|f\|_{L^2(\R)}^3,
\end{equation}
where ${\bf Q}_p$ denotes the optimal constant.
Odd curves are of independent interest, in particular because a new phenomenon emerges: caps centered around points with parallel tangents interact strongly, regardless of separation between the points. 
This mechanism was discovered in \cite{CS12a}, and further explored in \cite{CFOST17, Fo15, FLS16, FS17, Sh16}. 
Some of these works include a symmetrization step which relies on the convolution structure of the underlying inequality. 
In the present case, we also show that the search for extremizers can be further restricted to the class of even functions, but interestingly our symmetrization argument does not depend on the convolution structure.
This may be of independent interest since it applies to other Fourier extension inequalities where some additional symmetry is present, as we indicate in \S \ref{sec:symmetry} below.

The following versions of Theorems \ref{thm:Thm2} and \ref{thm:Thm3} hold for odd curves.

\begin{theorem}\label{thm:Thm5}
Let $p>1$. If
\begin{equation}\label{eq:IneqCriticalValueOp}
{\bf Q}_p^6>\frac{5\pi}{\sqrt{3} p(p-1)},
\end{equation}
then any extremizing sequence of nonnegative, even functions in $L^2(\R)$ for \eqref{eq:OddSharpConvolutionFormGenP} is precompact, after normalization and scaling.
In this case, extremizers for \eqref{eq:OddSharpConvolutionFormGenP} exist.
If instead equality holds in \eqref{eq:IneqCriticalValueOp}  then, given any $y_0\in\R$, there exists an extremizing sequence for \eqref{eq:OddSharpConvolutionFormGenP} which concentrates at the  pair $\{-y_0,y_0\}$.
\end{theorem}
\noindent The case $p=3$ of Theorem \ref{thm:Thm5} coincides with a special case of  \cite[Theorem 1]{FS17}, which was obtained by different methods.
\begin{theorem}\label{thm:Thm5.5}
If $p\in(1,2)$, then the strict inequality \eqref{eq:IneqCriticalValueOp}
holds and, in particular, there exists an extremizer for \eqref{eq:OddSharpConvolutionFormGenP}.
\end{theorem}
\noindent
We believe that extremizers do not exist if $p\geq 2$, see Conjecture \ref{cjct:nonexistence} below.\\

{\bf Overview.}
The paper is organized as follows. \S \ref{sec:Prelims} is devoted to the technical preliminaries for the dichotomy statement concerning the existence of extremizers: bilinear estimates and cap bounds. We then prove Theorem \ref{thm:Thm2}  in \S \ref{sec:Dichotomy}. Existence of extremizers is the subject of  \S \ref{sec:Existence}, where we establish Theorem \ref{thm:Thm3}. Theorem \ref{thm:Thm4} addresses the regularity of extremizers and is established in \S \ref{sec:Smoothness}. Odd curves are treated in \S \ref{sec:Odd}, where Theorems \ref{thm:Thm5} and \ref{thm:Thm5.5} are proved. In the Appendix, we establish useful variants of Lions' concentration-compactness lemma  (Proposition \ref{prop:metric-concentration-compactness-lemma}) and of a corollary of the Br\'ezis--Lieb lemma (Proposition \ref{prop:new-fvv}).\\

{\bf Notation.}
If $x,y$ are real numbers, we write $x=O(y)$ or $x\lesssim y$ if there exists a finite absolute constant $C$ such that $|x|\leq C|y|$. 
If we want to make explicit the dependence of the constant $C$  on some parameter $\alpha$, we  write $x=O_\alpha(y)$ or $x\lesssim_\alpha y$. 
We write $x\gtrsim y$ if $y\lesssim x$, and $x\simeq y$ if $x\lesssim y$ and $x\gtrsim y$.
Finally, the indicator function of a set $E\subset\R^d$ will be denoted by $\one_E$, and the complement of $E$ will at times be denoted by $E^\complement$.

\section{Bilinear estimates and cap refinements}\label{sec:Prelims} 
In this section, we prove the bilinear extension estimates and cap refinements which will be needed in the next section.
Bilinear extension estimates are usually deep \cite{Ta03, Wo01}, but in the one-dimensional case one may rely on the  classical Hausdorff--Young inequality. 
Throughout this section, we shall consider the dyadic regions 
$$I_k:=[2^{k},2^{k+1}), \text{ and } I_k^\bullet:=(-2^{k+1},-2^k]\cup [2^k,2^{k+1}),\;\;\;(k\in\Z).$$ 
\subsection{Bilinear estimates}
Recall the definitions \eqref{eq:DefSigmaMeasure} and \eqref{eq:DefFSigmaHat} of the measure {$\sigma_p$} and the Fourier extension operator $\mathcal{E}_p$, respectively.
Our first result quantifies the principle that distant caps interact weakly.
\begin{proposition}\label{prop:bilinear-p-power}
	Let $p>1$ and $k,k'\in\Z$.
	Then
	\begin{equation}\label{bilinear-exponential-decay}
	\norma{\mathcal{E}_p(f)\mathcal{E}_p(g)}_{L^3(\R^2)}\lesssim_p 
	2^{-\ab{k-k'}\frac{p-1}6}\norma{f}_{L^2(\R)}\norma{g}_{L^2(\R)},
		\end{equation}
	for every $f,g\in L^2(\R)$ satisfying 
	$\supp f\subseteq I_k^\bullet$ and 
	$\supp g\subseteq I_{k'}^\bullet$. 
\end{proposition}

\begin{proof}
Setting $\psi=\ab{\cdot}^p$ and $w=\ab{\cdot}^{\frac{p-2}3}$, we have that
\[ \big(\mathcal{E}_p(f)\mathcal{E}_p(g)\big)(x,t)=\int_{\R^2} 
e^{ix(y+y')}e^{it(\psi(y)+\psi(y'))}f(y)g(y')w(y)^{\frac12}w(y')^{\frac12}\d y\d y'. \]
Change variables $(y,y')\mapsto (u,v)=(y+y',\psi(y)+\psi(y'))$. 
Except for null sets, this is a $2$--to--$1$ map from $\R^2$ onto the region $\{(u,v)\colon v\geq 2\psi(u/2)\}$.
Its Jacobian is given by
\begin{equation}\label{eq:jacobian}
\begin{split}
J^{-1}(y,y')=\frac{\partial(u,v)}{\partial(y,y')}&=
\det\begin{pmatrix}
1&\psi'(y)\\
1&\psi'(y')
\end{pmatrix}=\psi'(y')-\psi'(y)\\
&=p(y'\ab{y'}^{p-2}-y\ab{y}^{p-2}),
\end{split}
\end{equation}
and satisfies
$\ab{J^{-1}(y,y')}\geq p \ab{\ab{y}^{p-1}-\ab{y'}^{p-1}}$,
with equality if and only if $yy'\geq 0$. 
Thus 
\begin{equation}\label{eq:PreHausdorffYoung}
\big(\mathcal{E}_p(f)\mathcal{E}_p(g)\big)(x,t)
=2\int e^{ixu}e^{itv}f(y)g(y')w(y)^{\frac12}w(y')^{\frac12}J(u,v)\d u\d v, 
\end{equation}
where the integral is taken over the region $\{(u,v)\colon v\geq 2\psi(u/2)\}$.
Note that this implies
\begin{equation}\label{eq:altFormulaConv}
(f\sigma_p\ast g\sigma_p)(u,v)=2f(y)g(y')w(y)^{\frac12}w(y')^{\frac12}J(u,v),
\end{equation}
for every $(u,v)$ satisfying $v> 2\psi(u/2)$,
where $(y, y')$ is related to $(u,v)$ via the change of variables described above. 

By symmetry, we can and will restrict attention to {$\ab{y'}\leq \ab{y}$}. 
Taking the $L^3$-norm of \eqref{eq:PreHausdorffYoung}, invoking the Hausdorff--Young inequality, and then changing variables back to $(y,y')$,
\begin{align*}
\norma{\mathcal{E}_p(f)\mathcal{E}_p(g)}_{L^3(\R^2)}
&\lesssim\norma{f(y)g(y')w(y)^{\frac12}w(y')^{\frac12}J(u,v) }_{L_{u,v}^{3/2}(\R^{1+1})}\\
&=\norma{f(y)g(y')w(y)^{\frac12}w(y')^{\frac12}\ab{J(y,y')}^{\frac13} }_{L_{y,y'}^{3/2}(\R^{1+1})}.
\end{align*}
If $2^{k}\leq\ab{y}< 2^{k+1}$, $2^{k'}\leq\ab{y'}< 2^{k'+1}$ and $k\geq k'+2$, then 
\begin{equation}\label{dyadic-bound-kernel}
\frac{\ab{yy'}^{\frac{p-2}4}}{p^{\frac12}\ab{\ab{y}^{p-1}-\ab{y'}^{p-1}}^{\frac12}}
\lesssim\frac{2^{(k+k')\frac{p-2}4}}{2^{k\frac{p-1}2}(1-2^{-(k-k'-1)(p-1)})^{\frac12}}\lesssim 
2^{(k'-k)\frac p4-\frac{k'}2}.
\end{equation}
It follows that
\begin{align}
\norma{\mathcal{E}_p(f)\mathcal{E}_p(g)}_{L^3}^{3/2}
&\lesssim\int_{\R^2} |f(y)g(y')|^{\frac32}w(y)^{\frac34}w(y')^{\frac34}\ab{J(y,y')}^{\frac12}\d y\d y'\label{eq:AfterHY}\\
&\leq\int_{\R^2} 
|f(y)g(y')|^{\frac32}
\frac{\ab{yy'}^{\frac{p-2}4}}{p^{\frac12}\ab{\ab{y}^{p-1}-\ab{y'}^{p-1}}^{\frac12}}\d 
y\d y'\notag\\
&\lesssim 2^{(k'-k)\frac p4-\frac{k'}2}2^{\frac{k}4}2^{\frac {k'}4}
\norma{f}_{L^2}^{\frac32}\norma{g}_{L^2}^{\frac32}\notag\\
&=2^{-\ab{k-k'}\frac{p-1}4}\norma{f}_{L^2}^{\frac32}\norma{g}_{L^2}^{\frac32}.\notag
\end{align}
If $k\in\{k',k'+1\}$, then we can simply use the estimate 
$\norma{\mathcal{E}_p(f)\mathcal{E}_p(g)}_{L^3}\lesssim\norma{f}_{L^2}\norma{g}_{L^2}$.
\end{proof}
\begin{corollary}\label{cor:bilinear-separated}
	Let $p>1$ and $k,k'\in\Z$ be such that $k'\leq k$.
		Then
	\begin{equation}\label{eq:bilinear-exponential-decay}
	\norma{\mathcal{E}_p(f)\mathcal{E}_p(g)}_{L^3(\R^2)}\lesssim_p 
	2^{-\ab{k-k'}\frac{p-1}6}\norma{f}_{L^2(\R)}\norma{g}_{L^2(\R)},
	\end{equation}
	for every $f,g\in L^2(\R)$ satisfying $\supp 
	f\subseteq \{\ab{y}\geq 2^k\}$ and $\supp g\subseteq \{\ab{y'}\leq 2^{k'} \}$. 
\end{corollary}
\begin{proof}
	Write $f=\sum_{j\geq k}f_j$ and $g=\sum_{j'< k'}g_{j'}$, 
	where $f_j:=f\one_{I_j^\bullet}$ and $g_{j'}:=g\one_{I_{j'}^\bullet}$. Then:
	\begin{align*}
	\norma{\mathcal{E}_p(f)\mathcal{E}_p(g)}_{L^3(\R^2)}
         &\leq \sum_{j\geq k,\,j'< k'}\norma{\mathcal{E}_p(f_j)\mathcal{E}_p(g_{j'})}_{L^3}
         \lesssim \sum_{j\geq k,\,j'< k'} 2^{-\ab{j-j'}\frac{p-1}6}\norma{f_j}_{L^2}\norma{g_{j'}}_{L^2}\\
	&\leq \biggl(\sum_{j\geq k,\,j'< k'} 2^{-\ab{j-j'}\frac{p-1}3}\biggr)^{\frac12}
	\biggl(\sum_{j\geq k,\,j'< k'} \norma{f_j}_{L^2}^2\norma{g_{j'}}_{L^2}^2\biggr)^{\frac12}\\
	&\simeq\biggl(\sum_{j\geq k} 2^{-\ab{j-k'}\frac{p-1}3}\biggr)^{\frac12}
	\norma{f}_{L^2}\norma{g}_{L^2}
	\simeq 2^{-\ab{k-k'}\frac{p-1}6}\norma{f}_{L^2}\norma{g}_{L^2},
	\end{align*}
	where we used the triangle inequality, Proposition \ref{prop:bilinear-p-power}, the Cauchy--Schwarz inequality, $L^2$-orthogonality, and the fact that a geometric series is comparable to its largest term.
\end{proof}
When studying concentration at points different from the origin, it will be useful to consider dyadic decompositions of the real line with arbitrary centers. 
By reflexion and scaling, it suffices to consider decompositions centered at $1$. 
Define the dyadic regions
$$\mathcal{I}_k:=\{2^k\leq y-1< 2^{k+1}\}, \text{ and } \mathcal{I}_k^\bullet:=\{2^k\leq \ab{y-1}< 2^{k+1}\},\;\;\;(k\in\Z)$$ 
so that $\I_k=1+I_k$ and $\I_k^\bullet=1+I_k^\bullet$. 
The following analogue of Proposition \ref{prop:bilinear-p-power} holds.

\begin{proposition}\label{prop:unc-bilinear-p-power}
	Let $p>1$ and $k,k'\in\Z$.
	Let $\beta=\min\{\frac{1}{6},\frac{p-1}{6}\}$.
	Then
	\begin{equation}\label{unc-bilinear-exponential-decay}
	\norma{\mathcal{E}_p(f)\mathcal{E}_p(g)}_{L^3(\R^2)}\lesssim_p 
	2^{-\beta\ab{k-k'}}\norma{f}_{L^2(\R)}\norma{g}_{L^2(\R)},
	\end{equation}
	for every $f,g\in L^2(\R)$ satisfying
	 $\supp f\subseteq \I_k^\bullet$ and 
	 $\supp g\subseteq \I_{k'}^\bullet$. 
\end{proposition}
Before embarking on the proof, let us take a closer look at the factor $\ab{yy'}^{\frac{p-2}4} \ab{J(y,y')}^{\frac12}$  that appears after applying the Hausdorff--Young inequality in \eqref{eq:AfterHY}. We have already seen that
	\begin{equation}\label{eq:JacobianFactorLook}
	\ab{J^{-1}(y,y')}=p\abs{y\ab{y}^{p-2}-y'\ab{y'}^{p-2}}.
	\end{equation}	
	In \eqref{dyadic-bound-kernel} we observed that, if $y,y'$ are {\it separated} (say, $\ab{y'}\leq \tfrac12 \ab{y}$),
	then 
	\begin{equation}\label{eq:separation-bound-kernel} 
	\frac{\ab{yy'}^{\frac{p-2}4}}{\abs{y\ab{y}^{p-2}-y'\ab{y'}^{p-2}}^{\frac12}}
	\lesssim\frac{\ab{yy'}^{\frac{p-2}4}}{\ab{y}^{\frac{p-1}2}}=\ab{y}^{-\frac p4}\ab{y'}^{\frac{p-2}4}.
	\end{equation}
 In order to obtain a useful bound in the case when both $y,y'$ are close to $1$, invoke the Mean Value Theorem and write
	\[ \ab{y}^{p-1}-\ab{y'}^{p-1}=(p-1)s^{p-2}(\ab{y}-\ab{y'}), \]
	for some $s\in[|y'|,|y|]$.
	Then, for $0\leq y'\leq y$, we have that
	\[ \abs{y\ab{y}^{p-2}-y'\ab{y'}^{p-2}}=\ab{y^{p-1}-y'^{p-1}}\gtrsim\begin{cases}
	\ab{y-y'}y^{p-2},&\text{if }p\in(1,2],\\
	\ab{y-y'}y'^{p-2},&\text{if }p\in [2,\infty).\\
	\end{cases} \]
	It follows that the following estimate holds, for every $\frac{1}{2}\leq y,y'\leq \frac{3}{2}$:
	\begin{equation}\label{close-to-1}
	\frac{\ab{yy'}^{\frac{p-2}4}}{\abs{y\ab{y}^{p-2}-y'\ab{y'}^{p-2}}^{\frac12}}\lesssim
	{\ab{y-y'}^{-\frac12}}.
	\end{equation}
	
\begin{proof}[Proof of Proposition \ref{prop:unc-bilinear-p-power}]	
	Without loss of generality, assume  $\ab{k-k'}\geq 2$.	
	We start by considering the situation when 0 is an endpoint of $\I_{k'}^\bullet$, i.e. $k'\in\{-1,0\}$. 
	Let $k'=-1$, so that $\I_{k'}^\bullet=(0,\frac{1}{2}]\cup[\frac{3}{2},2)$, split $g=g_\ell+g_r$, with 
	$g_\ell:=g\one_{(0,\frac{1}{2}]}$ and  $g_r:=g\one_{[\frac{3}{2},2)}$, and
 dyadically decompose 
	$$g_\ell=\sum_{j\geq 1}g_j,\; \text{ with } g_j:=g\one_{(2^{-(j+1)},2^{-j}]}.$$
	If $k\leq -3$, then \eqref{eq:separation-bound-kernel} implies
	\begin{align*}
	\norma{\mathcal{E}_p(f)\mathcal{E}_p(g_\ell)}_{L^3}
	&\lesssim \sum_{j\geq 1}
	\biggl(\int_{\R^2} |f(y)g_j(y')|^{\frac32} 
	\frac{\ab{yy'}^{\frac{p-2}4}}{\ab{\ab{y}^{p-1}-\ab{y'}^{p-1}}^{\frac12}}\d 
	y\d y'\biggr)^{\frac23}\\
	&\lesssim \sum_{j\geq 1}
	\biggl(2^{-j\frac{p-2}4}\int_{\R^2} |f(y)g_j(y')|^{\frac32} \d y\d y'\biggr)^{\frac23}
	\lesssim \sum_{j\geq 1}
	\biggl(2^{-j\frac{p-2}4}2^{\frac k4}2^{-\frac j4}\norma{f}_{L^2}^{\frac32}\norma{g_j}_{L^2}^{\frac32}\biggr)^{\frac23}\\
	&=2^{\frac k6}\norma{f}_{L^2}\sum_{j\geq 1}
	2^{-j\frac{p-1}6}\norma{g_j}_{L^2}\lesssim 2^{\frac k6}\norma{f}_{L^2}\norma{g_\ell}_{L^2}
	\lesssim 2^{-\frac{\ab{k-k'}}6}\norma{f}_{L^2}\norma{g}_{L^2}.
	\end{align*}
	If $k\geq 1$, then Corollary \ref{cor:bilinear-separated} applies, and directly yields
	\[ \norma{\mathcal{E}_p(f)\mathcal{E}_p(g_\ell)}_{L^3}
	\lesssim 2^{-\ab{k-k'}\frac{p-1}6}\norma{f}_{L^2}\norma{g}_{L^2}. \]
	A similar analysis applies to $g_r$.
	Setting $\beta:=\min\{\frac16,\frac{p-1}6\}$, we conclude that, if $k'=-1$ and $|k-k'|\geq 2$, then
	$$\norma{\mathcal{E}_p(f)\mathcal{E}_p(g)}_{L^3}\lesssim 2^{-\beta\ab{k-k'}}\norma{f}_{L^2}\norma{g}_{L^2}.$$
	The case $k'=0$ admits a similar treatment.
	If $k,k'\leq -2$ and $k-k'\geq 2$, then \eqref{close-to-1} implies
	\[ \norma{\mathcal{E}_p(f)\mathcal{E}_p(g)}_{L^3}\lesssim \frac{2^{\frac k6}2^{\frac{k'}6}}{2^{\frac k3}} \norma{f}_{L^2}\norma{g}_{L^2}=2^{-\frac{\ab{k-k'}}6}\norma{f}_{L^2}\norma{g}_{L^2}. \]
  Finally, the remaining cases can  be handled in a similar way by Corollary \ref{cor:bilinear-separated}.
\end{proof}
\begin{corollary}\label{bilinear-separated-1}
	Let $p>1$ and  $k,k'\in\Z$ be such that $k'\leq k$. 
	Let $\beta=\min\{\frac{1}{6},\frac{p-1}{6}\}$.
	Then
	\begin{equation}\label{bilinear-exponential-decay-1}
	\norma{\mathcal{E}_p(f)\mathcal{E}_p(g)}_{L^3(\R^2)}\lesssim_p 
	2^{-\beta\ab{k-k'}}\norma{f}_{L^2(\R)}\norma{g}_{L^2(\R)},
	\end{equation}
	for every $f,g\in L^2(\R)$ satisfying $\supp 
	f\subseteq \{\ab{y-1}\geq 2^k\}$ and $\supp g\subseteq \{\ab{y'-1}\leq 2^{k'} \}$. 
	\end{corollary}
We finish this subsection by taking yet another look at the Jacobian factor \eqref{eq:JacobianFactorLook}.
This will be useful in \S \ref{sec:CapBounds} below.
Let $p\geq 2$.
If $yy'\leq 0$, then 
$\ab{J^{-1}(y,y')}=p(\ab{y}^{p-1}+\ab{y'}^{p-1})$, in which case 
\begin{align*}
\frac{\ab{yy'}^{\frac{p-2}4}}{(\ab{y}^{p-1}+\ab{y'}^{p-1})^{\frac12}}
\lesssim{(\ab{y}+\ab{y'})^{-\frac12}}={\ab{y-y'}^{-\frac12}},
\end{align*}
uniformly in $y,y'$. 
To handle the complementary case $yy'> 0$, note that,
if $p\geq 2$ and $0\leq a\leq b$, then
\begin{equation}\label{eq:difference-powers}
b^{p-1}-a^{p-1}\simeq(b-a)b^{p-2}.
\end{equation}
 It follows that, if $p\geq 2$ and $yy'> 0$, then 
\[\ab{J^{-1}(y,y')}=p\ab{\ab{y}^{p-1}-\ab{y'}^{p-1}}\simeq\ab{y-y'}\max\{\ab{y},\ab{y'}\}^{p-2},\]
and so  if additionally $\ab{y}\geq \ab{y'}$, then
\begin{equation*}
\frac{\ab{yy'}^{\frac{p-2}4}}{\ab{\ab{y}^{p-1}-\ab{y'}^{p-1}}^{\frac12}}
\lesssim \frac{\ab{yy'}^{\frac{p-2}4}}{\ab{y}^{\frac{p-2}2}\ab{y-y'}^{\frac12}}
\leq {\ab{y-y'}^{-\frac12}}.
\end{equation*}
Therefore the estimate
\begin{equation}\label{eq:upper-bound-p-geq-2}
\norma{\mathcal{E}_p(f)\mathcal{E}_p(g)}_{L^3(\R^2)}^{\frac32}\lesssim \int_{\R^2} 
\frac{|f(y)g(y')|^{\frac32}}{\ab{y-y'}^{\frac12}}
\d y\d y'
\end{equation}
holds as long as $p\geq 2$.
We cannot hope for such a bound if $1<p<2$ since \eqref{eq:difference-powers} fails in that case.
However,  if $\ab{y}\simeq\ab{y'}$,
then one can check in a similar way that  the estimate
\begin{equation}\label{eq:same-dyadic-scale-bound-T}
\norma{\mathcal{E}_p(f_k)\mathcal{E}_p(g_k)}_{L^3(\R^2)}^{\frac32}\lesssim \int_{\R^2} 
\frac{|f_k(y)g_k(y')|^{\frac32}}{\ab{y-y'}^{\frac12}}
\d y\d y',
\end{equation}
holds for any $p>1$ and  functions $f_k,g_k$ which are both supported on $I_k^\bullet$.

\subsection{Cap bounds}\label{sec:CapBounds}
An inspection of the proof of Proposition \ref{prop:bilinear-p-power} reveals that if $\supp f\subseteq 
I_k^\bullet$ and $\supp g\subseteq I_{k'}^\bullet$, 
for some $k,k'\in\Z$ satisfying $k-k'\geq2$, then
\begin{equation}\label{eq:bilinear-refinement}
\begin{split}
\norma{\mathcal{E}_p(f)\mathcal{E}_p(g)}_{L^3(\R^2)}
&\lesssim 2^{-\ab{k-k'}\frac{p-1}{6}}\biggl(\ab{I_{k}}^{-\frac14}\int_{I_k^\bullet}|f|^{\frac32}\biggr)^{\frac23}\biggl(\ab{I_{k'}}^{-\frac14}\int_{I_{k'}^\bullet}|g|^{\frac32}\biggr)^{\frac23}\\
&\lesssim 2^{-\ab{k-k'}\frac{p-1}{6}}
\Lambda(f)^{\frac29}\Lambda(g)^{\frac29}\norma{f}_{L^2(\R)}^{\frac23}\norma{g}_{L^2(\R)}^{\frac23},
\end{split}
\end{equation}
where the quantity $\Lambda(f)$ is defined via
\begin{equation}\label{eq:DefLambdaF}
\Lambda(f):=\sup_{k\in\Z}\ab{I_{k}}^{-\frac14}\int_{I_k^\bullet}|f|^{\frac32}.
\end{equation}
The purpose of this subsection is to develop on this observation.
Given $f\in L^2(\R)$, write $f=\sum_{k\in\Z} f_k$, with $f_k:=f\one_{I_k^\bullet}$.
Our first result is the following.
\begin{proposition}\label{eq:cubic-refinement}
	Let $p>1$. Then the following estimates hold, for every $f\in L^2(\R)$:
	\begin{align}
	\norma{\mathcal{E}_p(f)}^3_{L^6(\R^2)}&\lesssim_p\sum_{k\in\Z}\norma{f_k}_{L^2(\R)}^3,\label{eq:refined}\\
	\norma{\mathcal{E}_p(f)}_{L^6(\R^2)}^3\lesssim_p 
	\sum_{k\in\Z}\norma{\mathcal{E}_p(f_k)}&_{L^6(\R^2)}^3+
\Lambda(f)^{\frac49}	\Bigl(\sum_{k\in\Z}\norma{f_k}_{L^2(\R)}^3\Bigr)^{\frac13}
	\norma{f}_{L^2(\R)}^{\frac43}.\label{eq:super-refined}
	\end{align}
\end{proposition}
\begin{proof}
	By the triangle inequality,
	\begin{align*}
	\norma{\mathcal{E}_p(f)}_{L^6}^3
	\leq	\sum_{(i,j,k)\in\Z^3}\norma{\mathcal{E}_p(f_i)\mathcal{E}_p(f_j)\mathcal{E}_p(f_k)}_{L^2}.
	\end{align*}
	For each triple $(i,j,k)$ in the previous sum, we lose no generality in assuming that 
	\begin{equation}\label{eq:maximality}
	\ab{j-k}=\max\{\ab{i'-j'}:i',j'\in\{i,j,k\}\}.
	\end{equation}
	H\"older's inequality and Proposition \ref{prop:bilinear-p-power} then imply
	\begin{equation*}
	\norma{\mathcal{E}_p(f_i)\mathcal{E}_p(f_j)\mathcal{E}_p(f_k)}_{L^2}
	\lesssim 2^{-\ab{j-k}\frac{p-1}6}\norma{f_i}_{L^2}\norma{f_j}_{L^2}\norma{f_k}_{L^2}.
	\end{equation*}
	By the maximality of $\ab{j-k}$, we have that $\ab{j-k}\geq 
	\frac{1}{3}\ab{i-j}+\frac{1}{3}\ab{j-k}+\frac{1}{3}\ab{k-i}$, and hence 
	\[\norma{\mathcal{E}_p(f)}_{L^6}^3\lesssim 
	\sum_{(i,j,k)\in\Z^3}2^{-\ab{i-j}\frac{p-1}{18}}2^{-\ab{j-k}\frac{p-1}{18}}2^{-\ab{k-i}\frac{p-1}{18}}
	\norma{f_i}_{L^2}\norma{f_j}_{L^2}\norma{f_k}_{L^2}.\]
	A final application of H\"older's inequality  yields \eqref{eq:refined}.
Estimate \eqref{eq:super-refined} follows from similar considerations which we now detail. 
Let $S:=\{(i,j,k)\in\Z^3\colon \max\{\ab{i-j},\ab{j-k},\ab{k-i}\}\leq 1 \}$ and $S^\complement:=\Z^3\setminus S$. 
Split the sum into diagonal and off-diagonal contributions,
	\begin{align*}
	\norma{\mathcal{E}_p(f)}_{L^6}^3
	\leq  \Norma{\sum_{(i,j,k)\in S}\mathcal{E}_p(f_i) \mathcal{E}_p(f_j) \mathcal{E}_p(f_k)}_{L^2}+\Norma{\sum_{(i,j,k)\in 
	S^\complement}\mathcal{E}_p(f_i) \mathcal{E}_p(f_j) \mathcal{E}_p(f_k)}_{L^2},
    \end{align*}
        and analyze the two terms separately. For the diagonal term, note that
        	\begin{align*}
	\Norma{&\sum_{(i,j,k)\in S}\mathcal{E}_p(f_i)\mathcal{E}_p(f_j)\mathcal{E}_p(f_k)}_{L^2}\\
	&\leq \sum_{k\in\Z} \big(3\norma{\mathcal{E}_p(f_k)\mathcal{E}_p(f_k)\mathcal{E}_p(f_{k+1})}_{L^2}
	+3\norma{\mathcal{E}_p(f_{k-1})\mathcal{E}_p(f_k)\mathcal{E}_p(f_{k})}_{L^2}
	+\norma{\mathcal{E}_p(f_k)\mathcal{E}_p(f_k)\mathcal{E}_p(f_k)}_{L^2}\big)\\
	&\leq \sum_{k\in\Z} \big(3\norma{\mathcal{E}_p(f_k)}_{L^6}^2\norma{ 
		\mathcal{E}_p(f_{k+1})}_{L^6}+3\norma{\mathcal{E}_p(f_{k-1})}_{L^6}\norma{\mathcal{E}_p(f_k)}_{L^6}^2+\norma{\mathcal{E}_p(f_k)}_{L^6}^3\big)
		\lesssim\sum_{k\in\Z}\norma{\mathcal{E}_p(f_k)}_{L^6}^3.
	\end{align*}
		To handle the off-diagonal term,  note that estimate \eqref{eq:bilinear-refinement} implies 
	\begin{align*}
	\Norma{\sum_{(i,j,k)\in S^\complement}\mathcal{E}_p(f_i)\mathcal{E}_p(f_j)\mathcal{E}_p(f_k)}_{L^2}
	&\lesssim \sideset{}{'}\sum_{(i,j,k):\ab{j-k}\geq 2}\norma{f_i}_{L^2}\norma{\mathcal{E}_p(f_j)\mathcal{E}_p(f_k)}_{L^3}\\
	&\lesssim \Lambda(f)^{\frac49}
	\sideset{}{'}\sum_{(i,j,k):\ab{j-k}\geq 
		2}2^{-\ab{j-k}\frac{p-1}{6}}\norma{f_i}_{L^2}\norma{f_j}_{L^2}^{\frac23}\norma{f_k}_{L^2}^{\frac23},
	\end{align*}
	where the sum $\Sigma'$ is taken over triples $(i,j,k)\in S^\complement$ for which $(j,k)$ satisfies the maximality assumption \eqref{eq:maximality}.
	It follows that
	\begin{align*}
	\Norma{\sum_{(i,j,k)\in S^\complement}\mathcal{E}_p(f_i)\mathcal{E}_p(f_j)\mathcal{E}_p(f_k)}_{L^2}
	&\lesssim \Lambda(f)^{\frac49} \sum_{i,j,k}2^{-(\ab{i-j}+\ab{j-k}+\ab{k-i})\frac{p-1}{18}}
	\norma{f_i}_{L^2}\norma{f_j}_{L^2}^{\frac23}\norma{f_k}_{L^2}^{\frac23}\\
	&\lesssim \Lambda(f)^{\frac49}
	\Bigl(\sum_{k\in\Z}\norma{f_k}_{L^2}^3\Bigr)^{\frac13}\Bigl(\sum_{k\in\Z}\norma{f_k}_{L^2}^2\Bigr)^{\frac23}.
	\end{align*}
This implies \eqref{eq:super-refined} at once, and concludes the proof of the proposition.
\end{proof}

\noindent The following $L^2$ dyadic cap estimate is a direct consequence of \eqref{eq:refined}.
\begin{corollary}\label{cor:dyadic-improvement}
	Let $p>1$. Then, for every $f\in L^2(\R)$,
	\[ \norma{\mathcal{E}_p(f)}^3_{L^6(\R^2)}\lesssim_p 
	\Big(\sup_{k\in\Z}\norma{f_k}_{L^2(\R)}\Big)\norma{f}_{L^2(\R)}^2.  \]
\end{corollary}

We now derive a cap bound similar to \cite[Lemma 1.2]{JPS10} and \cite[Lemma 1.2]{Sh09}.
\begin{proposition}\label{cap-bound-p-power}
	Let $p>1$. Then the following estimate holds:
	\begin{align}
	\norma{\mathcal{E}_p(f)}_{L^6(\R^2)}^3&\lesssim_{p}
        	\Bigl(\sup_{k\in\Z}\sup_{I\subseteq  
	I_k^\bullet}\ab{I}^{-\frac{1}{6}}\norma{f}_{L^{\frac32}(I)}\Bigr)^{\frac23} \norma{f}_{L^2(\R)}^{\frac73},
	\label{eq:12-cap-statement}
	\end{align}
	for every $f\in L^2(\R)$,
	 where the inner supremum is taken over all subintervals $I\subseteq  I_k^\bullet$.
\end{proposition}

\begin{proof}
We start by considering the case when $f=f_k(=f\one_{I_k^\bullet})$. 
From \eqref{eq:same-dyadic-scale-bound-T}, we have that
\begin{equation}\label{eq:p12-below}
\norma{\mathcal{E}_p(f_k)}_{L^6}^3
\lesssim \int_{\R^2} \frac{|f_k(y)f_k(y')|^{\frac32}}{\ab{y-y'}^{\frac12}}
\d y\d y'.
\end{equation}
Arguing as in as in \cite{JPS10, Sh09} we obtain, for every $q>1$, that
\begin{equation}\label{eq:RefinedJPSstyle}
\norma{\mathcal{E}_p(f_k)}_{L^6}\lesssim 
\bigl(\sup_{I\subseteq 
I_k^\bullet}\ab{I}^{\frac{1}{2}-\frac{1}{q}}\norma{f_k}_{L^q(I)}\bigr)^{\frac13}\norma{f_k}_{L^2(\R)}^{\frac23}.
\end{equation}
For the convenience of the reader, we provide the details. 
In light of \eqref{eq:p12-below}, we may assume $f_k\geq 0$.
Normalizing the supremum in \eqref{eq:RefinedJPSstyle} to equal 1, we may further assume that
\begin{equation}\label{eq:SupNormalization}
\int_I f_k^q\leq |I|^{1-\frac q2},\text{ for every subinterval } I\subseteq I_k^\bullet.
\end{equation}
Denote the collection of dyadic intervals of length $2^j$ by $\mathcal{D}_j:=\{2^j[k,k+1): k\in\Z\}$, and set $\mathcal{D}:=\bigcup_{j\in\Z}\mathcal{D}_j$. 
We perform a Whitney decomposition of $\R^2\setminus\{(y,y):y\in\R\}$ in the following manner, see for instance \cite[Lemma 10]{DMPS17} and \cite[Proof of Theorem 1.2]{BV07}. 
Given distinct $y,y'\in\R$, there exists a unique pair of maximal dyadic intervals $I,I'$ 
satisfying
$$(y,y')\in I\times I',\;\;|I|=|I'|,\;\text{ and dist}(I,I')\geq 4|I|.$$
Let {$\mathfrak{I}$} denote the collection of all such pairs as $y\neq y'$ ranges over $\R\times\R.$
Then 
\[\sum_{(I,I')\in\mathfrak{I}}\one_I(y)\one_{I'}(y')=1,\quad \text{for every }(y,y')\in\R^2 \text{ with } y\neq y',\]
{and therefore}
$$f_k(y)f_k(y')=\sum_{(I,I')\in\mathfrak{I}} f_{k,I}(y)f_{k,I'}(y'), \text{ for a.e. }(y,y')\in\R^2,$$
where $f_{k,I}:=f_k\one_I$. 
Clearly, if $(y,y')\in I\times I'$ and $(I,I')\in\mathfrak{I}$, then $|y-y'|\simeq |I|$.
From this and \eqref{eq:p12-below}, 
we may choose a slightly larger dyadic interval containing $I\cup I'$ but  of length comparable to $|I|$ (still denoted by $I$), and
it suffices to show that
$$\sum_{I\in\mathcal{D}}\frac{1}{|I|^{\frac12}}\Big(\int f_{k,I}^{\frac32}\Big)^2\lesssim \int f_k^2.$$
We further decompose
$$f_{k,I}=\sum_{n\in\Z} f_{k,I,n},\text{ where }
f_{k,I,n}:=f_k\one_{\big\{y\in I:\frac{2^n}{|I|^{1/2}}\leq f_k(y)< \frac{2^{n+1}}{|I|^{1/2}}\big\}},$$
and note that it suffices to establish
\begin{equation}\label{eq:SufficesToShow}
\sum_{I\in\mathcal{D}}\frac{1}{|I|^{\frac12}}\Big(\int f_{k,I,n}^{\frac32}\Big)^2\lesssim 2^{-|n|\eps}\int f_k^2,
\end{equation}
for some $\eps>0$ and every $n\in\Z$.
By the Cauchy--Schwarz inequality,
$$\Big(\int f_{k,I,n}^{\frac32}\Big)^2\leq\Big(\int f_{k,I,n}^2\Big)\Big(\int f_{k,I,n}\Big).$$
By construction of $f_{k,I,n}$, Chebyshev's inequality, and normalization \eqref{eq:SupNormalization},
\begin{equation}\label{eq:EstimatePosN}
\int f_{k,I,n}
\leq\tfrac{2^{n+1}}{|I|^{1/2}}\big|\big\{y\in I: f_k(y)\geq\tfrac{2^n}{|I|^{1/2}}\big\}\big|
\leq\tfrac{2^{n+1}}{|I|^{1/2}}\frac{\int_I f_k^q}{2^{nq}|I|^{-q/2}}
\lesssim 2^{-|n|(q-1)} |I|^{\frac12},
\end{equation}
for  every $q>1$ and $n\geq 0$. If $n<0$, then the following simpler estimate suffices:
\begin{equation}\label{eq:EstimateNegN}
\int f_{k,I,n}\lesssim\tfrac{2^n}{|I|^{1/2}} |I|=2^{-|n|} |I|^{\frac12}.
\end{equation}
Combining \eqref{eq:EstimatePosN} and \eqref{eq:EstimateNegN}, we conclude
$$\sum_{I\in\mathcal{D}}\frac{1}{|I|^{\frac12}}\Big(\int f_{k,I,n}^{\frac32}\Big)^2
\lesssim 2^{-|n|\eps}\sum_{I\in\mathcal{D}}\int f_{k,I,n}^2,$$
for some $\eps>0$, from which we get the desired \eqref{eq:SufficesToShow} by noting that
$$\sum_{I\in\mathcal{D}}\int f_{k,I,n}^2
=\sum_{j\in\Z}\sum_{I\in\mathcal{D}_j}\int f_{k}^2 \one_{\{ f_k\simeq 2^{n-j/2}\}}
=\int_\R \Big(\sum_{j\in\Z: \atop f_k(y)\simeq 2^{n-j/2}} f_k^2(y)\Big) \d y\lesssim \int f_k^2.$$
This concludes the verification of \eqref{eq:RefinedJPSstyle}.
Recalling inequality \eqref{eq:super-refined}, and 
specializing  \eqref{eq:RefinedJPSstyle} to $q=\frac32$, yields
\begin{align*}
\norma{\mathcal{E}_p(f)}_{L^6}^3&\lesssim \Bigl(\sup_{k,I\subseteq 
I_k^\bullet}\ab{I}^{-\frac{1}{6}}\norma{f_k}_{L^{\frac32}(I)}\Bigr)\sum_{k\in\Z}
\norma{f_k}_{L^2}^{2}
+\Bigl(\sup_{k\in\Z}\ab{I_k}^{-\frac{1}{6}}\norma{f_k}_{L^{\frac32}}\Bigr)^{\frac23}\norma{f}_{L^2}^{\frac73}\\
&\lesssim 
\Bigl(\sup_{k\in\Z}\sup_{I\subseteq 
I_k^\bullet}\ab{I}^{-\frac{1}{6}}\norma{f_k}_{L^{\frac32}(I)}\Bigr)^{\frac23}\norma{f}_{L^2}^{\frac73},
\end{align*}
where the last line follows from H\"older's inequality. 
This concludes the proof.
\end{proof}

In the next section, it will be useful to have the $L^1$ version of \eqref{eq:12-cap-statement} at our disposal, and this is the content of the following result.

\begin{proposition}\label{prop:l1-cap-bound-p-power}
	Let $p>1$. Then there exist $\gamma\in(0,1)$  such that	
	\begin{align}
	\norma{\mathcal{E}_p(f)}_{L^6(\R^2)}
	&\lesssim_{p,\gamma}\bigl(\sup_{k\in\Z}\sup_{I\subseteq 
	I_k^\bullet}\ab{I}^{-\frac{1}{2}}\norma{f}_{L^1(I)}\bigr)^{\gamma}\norma{f}_{L^2(\R)}^{1-\gamma},
	\label{eq:dyadic-l1-cap-bound}
	\end{align}
	 for every	$f\in L^2(\R)$, 	 where the inner supremum is taken over all subintervals $I\subseteq  I_k^\bullet$.
\end{proposition}
\noindent The proof below yields $\gamma=\frac2{45}$, and is inspired by \cite[Proposition 2.9]{CS12a}.

\begin{proof}[Proof of Proposition \ref{prop:l1-cap-bound-p-power}]
Set $\delta:={\norma{\mathcal{E}_p(f)}_{L^6}}{\norma{f}^{-1}_{L^2}}$.
From \eqref{eq:12-cap-statement}  we have that
\[ \sup_{k\in\Z}\sup_{I\subseteq I_k^\bullet}\ab{I}^{-\frac16}\norma{f}_{L^{\frac32}(I)}\gtrsim\delta^{\frac92}\norma{f}_{L^2(\R)}. \]
Then there exist $k\in\Z$ and an interval $I\subseteq I_k^\bullet$, such that 
\[ \int_I \ab{f}^{\frac32}\geq 
c_0\delta^{\frac{27}4}\ab{I}^{\frac14}\norma{f}_{L^2(\R)}^{\frac32}, \]
for a universal constant $c_0$ (independent of $f, \delta$).
Given $R\geq 1$, define the set $E:=\{y\in I\colon\ab{f(y)}\leq R \}$. 
Set $g:=f\one_E$ and $h:=f-g$. 
Then $g$ and $h$ have disjoint supports, and $\norma{g}_{L^{\infty}}\leq R$. 
Since $\ab{h(y)}\geq R$ for almost every $y\in I$ for which $h(y)\neq 0$,  we have
\[ \int_I \ab{h}^{\frac32}\leq R^{-\frac12}\int_I \ab{h}^2\leq R^{-\frac12}\norma{f}_{L^2(\R)}^2.\]
Choose $R$ satisfying $R^{-\frac12}=\frac{c_0}{2}\delta^{\frac{27}4}\ab{I}^{\frac14}\norma{f}_{L^2(\R)}^{-\frac12}$. 
Then
\[ \int_I\ab{g}^{\frac32}=\int_I\ab{f}^{\frac32}-\int_I\ab{h}^{\frac32}\geq 
\frac{c_0}{2}\delta^{\frac{27}4}\ab{I}^{\frac14}\norma{f}_{L^2(\R)}^{\frac32}. \]
Since $g$ is supported on $I$, H\"older's inequality implies
\begin{equation}\label{eq:LowerBoundL2}
 \norma{g}_{L^2}\geq \ab{I}^{-\frac{1}{6}}\norma{g}_{L^{\frac32}}\geq c_1\delta^{\frac92}\norma{f}_{L^2}, 
 \end{equation}
where $c_1$ is universal.
Since $\|g\|_{L^\infty}\leq R$, we have (by definition of $R$) that
\[ \ab{g(y)}\leq c_2\delta^{-\frac{27}{2}}\ab{I}^{-\frac12}\norma{f}_{L^2(\R)}\one_I(y), \text{ for almost every } y\in\R, \]
where $c_2$ is universal. 
Together with \eqref{eq:LowerBoundL2}, this implies the lower bound 
\begin{equation*}
\int_I\ab{g}\geq \int_I \ab{g}\frac{\ab{g}}{ c_2\delta^{-\frac{27}{2}}\ab{I}^{-\frac12}\norma{f}_{L^2}}
=c_2^{-1}\delta^{\frac{27}{2}}\ab{I}^{\frac12}\frac{\norma{g}_{L^2}^2}{\norma{f}_{L^2}}
\geq c_3\delta^{\frac{45}{2}}\ab{I}^{\frac12}\norma{f}_{L^2}.
\end{equation*}
where $c_3$ is universal.
Since $\ab{g}\leq \ab{f}$, it follows that
\[c_3\delta^{\frac{45}{2}}\ab{I}^{\frac12}\norma{f}_{L^2(\R)}
\leq \norma{g}_{L^1(I)} \leq\norma{f}_{L^1(I)}. \]
Recalling the definition of $\delta$, we obtain \eqref{eq:dyadic-l1-cap-bound} with $\gamma=\frac2{45}$.
This completes the proof.
\end{proof}

\section{Existence versus concentration}\label{sec:Dichotomy}
This section is devoted to the proof of Theorem \ref{thm:Thm2}.
Start by observing the scale invariance of \eqref{eq:SharpConvolutionFormGenP}, or equivalently that of \eqref{eq:SharpExtensionFormGenP}. 
Indeed, if $f_\la(y):=f(\la y)$, then
$\norma{f_\la}_{L^2(\R)}=\la^{-1/2}\norma{f}_{L^2(\R)}$. 
On the other hand, $\mathcal{E}_p(f_\la)(x,t)=\la^{-(p+4)/6}\mathcal{E}_p(f)(x/\la,t/\la^p)$, and so
$$\norma{\mathcal{E}_p(f_\la)}_{L^6(\R^2)}
=\la^{-\frac{p+4}6+\frac{p+1}6}\norma{\mathcal{E}_p(f)}_{L^6(\R^2)}
=\la^{-\frac12}\norma{\mathcal{E}_p(f)}_{L^6(\R^2)}.$$
In particular, given any sequence $\{a_n\}\subset\R\setminus\{0\}$,
 if $\{f_n\}$ is an $L^2$-normalized extremizing sequence for \eqref{eq:SharpExtensionFormGenP}, then so is $\{\ab{a_n}^{1/2}f_n(a_n \cdot)\}$.
 
 We come to the first main result of this section.

\begin{proposition}\label{prop:existence-vs-concentration}
	Let $\{f_n\}\subset L^2(\R)$ be an $L^2$-normalized extremizing sequence of nonnegative functions for \eqref{eq:SharpExtensionFormGenP}. 
	Then there exists a subsequence $\{f_{n_k}\}$, and a sequence $\{a_k\}\subset \R\setminus\{0\}$,  such that the rescaled sequence $\{g_k\}$, $g_k:=\ab{a_k}^{1/2}f_{n_k}(a_k \cdot)$, satisfies one of the 
	following conditions:
	\begin{enumerate}
		\item[(i)] There exists $g\in L^2(\R)$ such that $g_{k}\to g$ in $L^2(\R)$, as $k\to\infty$, or
		\item[(ii)] $\{g_k\}$  concentrates at $y_0=1$.
	\end{enumerate}
\end{proposition}

Theorem \ref{thm:Thm2} follows at once from Proposition \ref{prop:existence-vs-concentration} and the following result.

\begin{lemma}\label{lem:p-upper-bound-concentration}
	Let $p>1$.
	Given $y_0\in\R\setminus\{0\}$, let $\{f_n\}\subset L^2(\R)$ be a sequence concentrating at $y_0$. Then 
	\begin{equation}
	\label{eq:p-value_boundary}
	\limsup_{n\to\infty}\frac{\norma{f_n\sigma_p\ast f_n\sigma_p \ast f_n\sigma_p}_{L^2(\R^{2})}^2}{\norma{f_n}_{L^2(\R)}^6}\leq 
	\frac{2\pi}{\sqrt{3}\,p(p-1)}.
	\end{equation}
	If we set
	 $f_n(y)=e^{-n(\ab{y}^p-\ab{y_0}^p-py_0\ab{y_0}^{p-2}(y-y_0))}\ab{y}^{\frac{p-2}6}$,
	then the sequence 
	$\{f_n\norma{f_n}_{L^2}^{-1}\}$ concentrates at $y_0$, and equality holds in \eqref{eq:p-value_boundary}.
\end{lemma}

\noindent Convolution of singular measures is treated in much greater generality in the companion paper \cite{OSQ18}.
Lemma \ref{lem:p-upper-bound-concentration} is almost contained in \cite{OSQ16, OSQ18}, and we just indicate the necessary changes.

\begin{proof}[Proof sketch of Lemma \ref{lem:p-upper-bound-concentration}]
Once the boundary value for $\ab{\cdot}^{\frac{p-2}6}\sigma_p\ast \ab{\cdot}^{\frac{p-2}6}\sigma_p\ast \ab{\cdot}^{\frac{p-2}6}\sigma_p$ given in  \eqref{eq:BoundaryValues3fold} below is known to equal the right-hand side of \eqref{eq:p-value_boundary},
the proof for $p\geq 2$ follows the exact same lines as that of \cite[Lemmata 4.1 and 4.2]{OSQ16}. We omit the details.

If $1<p<2$, then the function $|\cdot|^{\frac{p-2}{6}}$ fails to be continuous at the origin, and an additional argument is needed.
We show how to reduce matters to the analysis of projection measure.
Let  $\{f_n\}\subset L^2(\R)$ concentrate at $y_0\neq 0$. Then
\begin{equation}\label{removing-weight}
	\limsup_{n\to\infty}\frac{\norma{f_n\sigma_p\ast 
			f_n\sigma_p \ast f_n\sigma_p}_{L^2}^2}{\norma{f_n}_{L^2}^6}
	=\ab{y_0}^{p-2}\limsup_{n\to\infty}\frac{\norma{f_n\nu_p\ast 
			f_n\nu_p \ast f_n\nu_p}_{L^2}^2}{\norma{f_n}_{L^2}^6},
\end{equation}
where $\nu_p$ denotes the projection measure $\d\nu_p=\ddirac{s-|y|^p}\d y\d s$.
To verify \eqref{removing-weight}, consider the interval $J:=[{y_0}/{2},{3y_0}/{2}]$. Then
\begin{align*}
\limsup_{n\to\infty}\frac{\norma{f_n\sigma_p\ast f_n\sigma_p \ast f_n\sigma_p}_{L^2}^2}{\norma{f_n}_{L^2}^6}&
=\limsup_{n\to\infty}\frac{\norma{f_n\one_{J}\sigma_p\ast f_n\one_{J}\sigma_p \ast f_n\one_{J}\sigma_p}_{L^2}^2}{\norma{f_n\one_{J}}_{L^2}^6}\\
&=\ab{y_0}^{p-2}\limsup_{n\to\infty}\frac{\norma{f_n\nu_p\ast f_n\nu_p \ast f_n\nu_p}_{L^2}^2}{\norma{f_n}_{L^2}^6}.
\end{align*}
Here, to justify the first equality,  invoke the continuity of the operator $\mathcal E_p$, and the fact that the sequence $\{f_n\}$ concentrates at $y_0$. 
For the second equality, additionally note that
\[\frac{\norma{f_n\one_{J}\ab{\cdot}^{\frac{p-2}6}-f_n\one_{J}\ab{y_0}^{\frac{p-2}6}}_{L^2}}{\norma{f_n\one_{J}}_{L^2}}\to0,\quad\text{ as }n\to\infty.\]
From \cite[Proposition 2.1]{OSQ18}, the measure $\nu_p\ast\nu_p\ast\nu_p$ defines a continuous function in the 
interior of its support, with continuous extension to the boundary except at  $(0,0)$. Moreover, for any $y_0\neq 0$,
\[ 
(\nu_p\ast\nu_p\ast\nu_p)(3y_0,3\ab{y_0}^p)=\frac{2\pi}{\sqrt{3}\,p(p-1)\ab{y_0}^{p-2}}.
\]
The result now follows as in \cite[Lemmata 4.1 and 4.2]{OSQ16}.
\end{proof}

The proof of Proposition \ref{prop:existence-vs-concentration} relies on the bilinear extension estimates and cap bounds from \S \ref{sec:Prelims}, together with a suitable variant of Lions' concentration-compactness lemma, which is formulated in the appendix as Proposition \ref{prop:metric-concentration-compactness-lemma}.
This has two important consequences for the present context, the first of which is the following.

\begin{proposition}\label{prop:small-cap-implies-concentration}
	Let $\{f_n\}\subset L^2(\R)$ be an $L^2$--normalized extremizing sequence for \eqref{eq:SharpExtensionFormGenP}. Let $\{r_n\}$ be a sequence of nonnegative numbers, satisfying $r_n\to0$, as $n\to\infty$, and
	\[ \inf_{n\in\N}\int_{1-r_n}^{1+r_n}\ab{f_n(y)}^2\d y> 0. \]
	Then the sequence $\{f_n\}$ concentrates at $y_0=1$.
\end{proposition}
\begin{proof}
Consider the intervals $J_n:=[1-r_n,1+r_n]$, $n\in\N$, and define the pseudometric
	\begin{equation}\label{eq:DefPseudoMetric}
	\vrho\colon\R\setminus\{1\}\times\R\setminus\{1\}\to[0,\infty),\qquad \vrho(x,y):=\ab{k-k'}, 
	\end{equation}
	where $k,k'$ are such that $\ab{x-1}\in[2^{k},2^{k+1})$ and $\ab{y-1}\in[2^{k'},2^{k'+1})$. Let $R$ be an integer.
	Then the ball centered at $x\neq 1$ of radius $R$ defined by $\vrho$ is given by 
	$$B(x,R)=\{y\in\R\setminus\{1\}\colon 2^{k-R}\leq \ab{y-1}< 2^{k+R+1}\}.$$ 
		Let $\{f_n\}$ be as in the statement of the proposition. 
		Apply Proposition \ref{prop:metric-concentration-compactness-lemma} to the 
	sequence $\{\ab{f_n}^2\}$ with $X=\R$ equipped with Lebesgue measure, $\bar{x}=1$, the function $\vrho$ defined as in \eqref{eq:DefPseudoMetric}, and $\lambda=1$. 
	Passing to a subsequence, also denoted by $\{|f_n|^2\}$, one of three cases arises.

	{\bf Case 1.} The sequence $\{\ab{f_n}^2\}$ satisfies {\it compactness}.
	In this case, 
	there exists 
	$\{x_n\}\subset\R\setminus\{1\}$ with the property that for any $\eps>0$, there exists 
	$R<\infty$ such that, for every $n\geq 1$,
	\begin{equation}\label{eq:case-concentration}
	\int_{B(x_n,R)}\ab{f_n}^2\geq 1-\eps. 
	\end{equation}
	Suppose that $\limsup_{n\to\infty}\ab{x_n-1}>0$. Then, possibly after extraction of a subsequence,  $\{x_n\}$ is eventually far 
	from $1$, i.e. there exist $N_0\in\N$, $\ell^\ast\in\Z$
	such that $\ab{x_n-1}>2^{\ell^\ast}$, for every  $n\geq N_0$.
	Let 
	$\eps:=\frac{1}{2}\inf_n\norma{f_n}_{L^2(J_n)}^2>0$, and choose an integer 
	$R$ such that \eqref{eq:case-concentration} holds.
	Now, 
	$$B(x_n,R)=\{y\in\R\setminus\{1\}\colon 2^{k_n-R}\leq\ab{y-1}< 2^{k_n+R+1}\},$$
	where $k_n$ is such that $\ab{x_n-1}\in[2^{k_n},2^{k_n+1})$, and hence
	$B(x_n,R)\subseteq \{y\neq 1\colon 
	\ab{y-1}\geq 2^{\ell^\ast-R}\}$.
	Let $N_1\geq N_0$ be  such that $r_n< 2^{\ell^\ast-R}$, for every $n\geq N_1$. In this case, we have $J_n\cap B(x_n,R)=\emptyset$, 
	which is impossible because our choice of $\eps$ would then force
	\[ 1=\int_{\R}\ab{f_n}^2\geq  \int_{J_n}\ab{f_n}^2+\int_{B(x_n,R)}\ab{f_n}^2>1.\]
	It follows that $x_n\to 1$, as $n\to\infty$, and consequently the sequence 
	$\{f_n\}$ concentrates at $y_0=1$. 
	Indeed, given $\eps>0$, choose an integer $R$ such that
	\eqref{eq:case-concentration} holds. 
	Then $B(x_n,R)\subseteq [1-2^{k_n+R+1},1+2^{k_n+R+1}]\setminus\{1\}$, where $\ab{x_n-1}\in[2^{k_n},2^{k_n+1})$ and 
	$k_n\to-\infty$, as $n\to\infty$, so that $2^{k_n+R+1}\to 0$, as $n\to\infty$. This forces
	\[ \int_{1-2^{k_n+R+1}}^{1+2^{k_n+R+1}}\ab{f_n(y)}^2\d y\geq 1-\eps, \]
	for every $n\geq 1$,
	which implies concentration of the sequence $\{f_n\}$ at $y_0=1$. 
	
		{\bf Case 2.} The sequence $\{\ab{f_n}^2\}$ satisfies {\it dichotomy}.
 Let $\alpha\in(0,1)$ be as in the dichotomy condition. Given $\eps>0$, consider the corresponding data
	$R,\,k_0,\rho_{n,j}=|f_{n,j}|^2,j\in\{1,2\},\{x_n\}\subset\R\setminus\{1\},\{R_n\}\subset[0,\infty)$. 
	In particular,
	\[ \supp(f_{n,1})\subset B(x_n,R),\text{ and } \supp(f_{n,2})\subset B(x_n,R_n)^\complement.\]
	Since $R_n-R\to\infty$, as $n\to\infty$, by Corollary \ref{bilinear-separated-1} we obtain
	\begin{equation}\label{eq:weak-dichotomy}
	\norma{\mathcal{E}_p(f_{n,1})\mathcal{E}_p(f_{n,2})}_{L^3}\leq C_n\norma{f_{n,1}}_{L^2}\norma{f_{n,2}}_{L^2},
	\end{equation}
	where $C_n=C_n(\eps)\lesssim 2^{-\beta(R_n-R)}$, for some $\beta>0$.
	In particular, given $\eps>0$, we have that $C_n\to 0$, as $n\to\infty$. 
	Aiming at a contradiction, consider that
	\begin{equation}\label{eq:IneqConstant1}
	 \norma{\mathcal{E}_p(f_n-f_{n,1}-f_{n,2})}_{L^6}\leq {\bf E}_p\norma{f_n-(f_{n,1}+f_{n,2})}_{L^2}\leq 
	{\bf E}_p\eps^{\frac12}, 
	\end{equation}
	The latter inequality requires a short justification which boils down to the pointwise estimate
	\begin{equation}\label{eq:IneqConstant2}
          (\ab{f_n}-(\ab{f_{n,1}}+\ab{f_{n,2}}))^2\leq \ab{\ab{f_n}^2-(\ab{f_{n,1}}+\ab{f_{n,2}})^2}
          =\ab{\ab{f_n}^2-(\ab{f_{n,1}}^2+\ab{f_{n,2}}^2)}.
        \end{equation}
		This, in turn, follows from the disjointness of the supports of $f_{n,1}$ and $f_{n,2}$, together with the trivial estimate 
		$\ab{\ab{f_n}-(\ab{f_{n,1}}+\ab{f_{n,2}})}\leq {\ab{f_n}+(\ab{f_{n,1}}+\ab{f_{n,2}})}$. 
		In this way, \eqref{eq:IneqConstant2} and Proposition \ref{prop:metric-concentration-compactness-lemma} imply
		\begin{equation*}
		\norma{(\ab{f_n}-(\ab{f_{n,1}}+\ab{f_{n,2}}))^2}_{L^1}\leq \norma{\ab{f_n}^2-(\ab{f_{n,1}}^2+\ab{f_{n,2}}^2)}_{L^1}\leq \eps.
		\end{equation*}
	Coming back to \eqref{eq:IneqConstant1}, we have as an immediate consequence that 
$$\norma{\mathcal{E}_p(f_n)}_{L^6}\leq {\bf E}_p\eps^{\frac12}+\norma{\mathcal{E}_p(f_{n,1}+f_{n,2})}_{L^6}.$$
	Expanding the binomial, using  $\norma{f_{n,1}}_{L^2},\norma{f_{n,2}}_{L^2}\leq 1$, and H\"older's 
	inequality together with \eqref{eq:weak-dichotomy}, we find 
	that there exists $c$ independent of $n$ such that, for sufficiently large $n$,
	\begin{equation}\label{dichotomy-argument}
	\begin{split}
	\norma{\mathcal{E}_p(f_{n,1}+f_{n,2})}_{L^6}^6&\leq 
	\norma{\mathcal{E}_p(f_{n,1})}_{L^6}^6+\norma{\mathcal{E}_p(f_{n,2})}_{L^6}^6+cC_n\\
	&\leq {\bf E}_p^6(\norma{f_{n,1}}_{L^2}^6+\norma{f_{n,2}}_{L^2}^6)
	+cC_n\\
	&\leq {\bf E}_p^6((\alpha+\eps)^3+(1-\alpha+\eps)^3)+cC_n.
	\end{split}
	\end{equation}
	This implies, for every sufficiently large $n$,
	\[ \norma{\mathcal{E}_p(f_n)}_{L^6}\leq 
	{\bf E}_p\eps^{\frac12}+({\bf E}_p^6((\alpha+\eps)^3+(1-\alpha+\eps)^3)+cC_n)^{\frac16}.\]
	Taking $n\to\infty$, and recalling that $\{f_n\}$ is an $L^2$-normalized extremizing sequence for \eqref{eq:SharpExtensionFormGenP}, we find that 
		\[ {\bf E}_p\leq{\bf E}_p\eps^{\frac12}+{\bf E}_p((\alpha+\eps)^3+(1-\alpha+\eps)^3)^{\frac16},  \]
		for every $\eps>0$. Taking $\eps\to 0$ yields $1\leq\alpha^3+(1-\alpha)^3$, 
	which is impossible since $\alpha\in(0,1)$. Hence dichotomy does not arise.
	
	{\bf Case 3.} The sequence $\{|f_n|^2\}$ satisfies {\it vanishing}.
	In this case, 
	\[ \lim_{n\to\infty}\sup_{k\in\Z}\int_{2^{k-R}\leq\ab{y-1}\leq 2^{k+R+1}}\ab{f_n(y)}^2\d y=0, \]
	for every integer $R<\infty$.
	In particular, for  fixed $k\in \N$, we have
	\begin{equation}\label{eq:SmallOutside2k}
	 \lim_{n\to\infty}\int_{2^{-k}\leq\ab{y-1}\leq 2^k}\ab{f_n(y)}^2\d y=0. \
	 \end{equation}	
	 Set $f_{n,1}:=f_n\one_{[1-2^{-k},1+2^{-k}]}$ and $f_{n,2}:=f_n\one_{\{\ab{y-1}\geq  2^k\}}$. Since 
	$\norma{f_n-f_{n,1}-f_{n,2}}_{L^2}\to 0$, as $n\to\infty$, it follows that $\{f_{n,1}+f_{n,2}\}_n$ is also an 
	extremizing sequence for \eqref{eq:SharpExtensionFormGenP}, for each $k\in \N$. This new sequence splits the mass into two separated regions, and so we expect to reach a contradiction if 
	$\limsup_{n\to\infty}\norma{f_{n,2}}_{L^2}>0$, just as in Case 2. 
	Set $\alpha_k:=\limsup_{n\to\infty}\norma{f_{n,2}}_{L^2}^2$ (recall that $f_{n,2}$ depends on $k$), and note that  $\{\alpha_k\}$ is a constant sequence. Indeed, 
	\begin{equation}\label{eq:alpha-constant}
	\int_{\ab{y-1}\geq 2^k}\ab{f_n(y)}^2\d y=\int_{\ab{y-1}\geq 2^{k+1}}\ab{f_n(y)}^2\d 
	y+\int_{2^k\leq\ab{y-1}\leq 2^{k+1}}\ab{f_n(y)}^2\d y
	\end{equation}
	and from \eqref{eq:SmallOutside2k} with $k+1$ instead of $k$ we have
	\[ \lim_{n\to\infty}\int_{2^k\leq\ab{y-1}\leq 2^{k+1}}\ab{f_n(y)}^2\d y=0.\]
	Taking $\limsup_{n\to\infty}$ in \eqref{eq:alpha-constant} yields $\alpha_{k+1}=\alpha_k$, for every $k\in\N$. 
	An argument analogous to that of Case 2 (starting from \eqref{dichotomy-argument}) shows that
	there exist $\beta>0$ and a sequence $\{C_k\}$, $0\leq C_k\lesssim 2^{-\beta k}\to 0$, as $k\to\infty$, such that	
		\[ 1\leq \alpha_k^3+(1-\alpha_k)^3+C_k, \; \text{ for every } k\in\N.\]
	Since $\alpha_k\equiv \alpha$ is constant, we may take $k\to\infty$ 
	in the previous inequality and obtain
	$1\leq\alpha^3+(1-\alpha)^3.$
	Since $\alpha\in [0,1]$,  necessarily $\alpha\in\{0,1\}$. We claim that $\alpha=0$. 
	For any $k\geq 1$, the support of $f_{n,2}$ is disjoint 
	from the interval $J_n$ if $n$ large enough. Thus
	\[ \norma{f_{n,2}}_{L^2}^2\leq 1-\int_{J_n}\ab{f_n}^2\leq 1-\inf_{n\in\N}\int_{J_n}\ab{f_n}^2, \]
	and therefore
	\[ \alpha\leq 1-\inf_{n\in\N}\int_{J_n}\ab{f_n}^2<1. \]
	We conclude that $\alpha=0$, as claimed.
	Finally, we show that vanishing implies concentration at $y=1$. Since
	\[ 
	1=\norma{f_n}_{L^2}^2=\norma{f_{n,1}}_{L^2}^2+\norma{f_{n,2}}_{L^2}^2+o_n(1)=\norma{f_{n,1}}_{L^2}^2+
	o_n(1)=\norma{f_n\one_{[1-2^{-k},1+2^{-k}]}}_{L^2}^2+o_n(1), \]
	we find that, for every $k\in\N$,
	\[ \lim_{n\to\infty}\int_{1-2^{-k}}^{1+2^{-k}}\ab{f_n(y)}^2\d y=1.\]
	This implies that the sequence $\{f_n\}$ concentrates at $y_0=1$.
	
	To sum up, we proved that any sequence $\{f_n\}$ as in the statement of the proposition does not 
	satisfy dichotomy; and that if it satisfies compactness or vanishing, then it concentrates at $y_0=1$. 
Thus the proof is complete. 
 \end{proof}

As a second application of Proposition \ref{prop:metric-concentration-compactness-lemma},
we prove dyadic localization of extremizing sequences, after rescaling. We take $X=\R$, $\bar{x}=0$, and use the dyadic pseudometric 
\begin{equation}\label{eq:pseudometric-center-zero}
 \vrho\colon\R\setminus\{0\}\times\R\setminus\{0\}\to[0,\infty),\quad \vrho(x,y):=\ab{k-k'},
\end{equation}
where this time $\ab{x}\in[2^{k},2^{k+1})$ and $\ab{y}\in [2^{k'},2^{k'+1})$. In this case, if $R$ is an integer, then 
$$B(x,R)=\{y\in\R\setminus\{0\}\colon 2^{k-R}\leq\ab{y}< 2^{k+R+1}\}.$$

\begin{proposition}\label{prop:dyadic-localization}
Let $\{f_n\}\subset L^2(\R)$ be an $L^2$-normalized extremizing sequence for \eqref{eq:SharpExtensionFormGenP}. Then there exist a subsequence $\{f_{n_k}\}$, a sequence $\{a_k\}\subset\R\setminus\{0\}$, and a function $\Theta:[1,\infty)\to(0,\infty)$, $\Theta(R)\to0$, as $R\to\infty$, such that the rescaled sequence $\{g_k\}$, $g_k:={\ab{a_k}^{1/2}}f_{n_k}(a_k\cdot)$, satisfies
\begin{equation}\label{eq:uniform-l2-decay}
\norma{g_k}_{L^2([-R,R]^\complement)}\leq \Theta(R), \text{ for every }  k\geq 1 \text{ and } R\geq 1.
\end{equation}
\end{proposition}

\noindent This proposition will provide the input for the suitable application of the Br\'ezis--Lieb lemma, which is formulated in the appendix as Proposition \ref{prop:new-fvv}.

\begin{proof}[Proof of Proposition \ref{prop:dyadic-localization}]
Let $\{f_n\}$ be as in the statement of the proposition.
In view of Corollary \ref{cor:dyadic-improvement}, there exists $\ell_n\in\Z$ such that
$\norma{f_{n}}_{L^2(I_{\ell_n}^\bullet)}\gtrsim_p 1$, if $n$ is large enough.
Setting  $g_n:=2^{\ell_n/2}f_n(2^{\ell_n} \cdot)$, we then have that
\begin{equation}\label{lower-boung-g}
\norma{g_n}_{L^2(I_0^\bullet)}\gtrsim_p 1, 
\end{equation}
for every sufficiently large $n$.
Using Proposition \ref{prop:metric-concentration-compactness-lemma} with the pseudometric \eqref{eq:pseudometric-center-zero}, we obtain a subsequence $\{|g_{n_k}|^2\}$ that satisfies one of three possibilities. Because of \eqref{lower-boung-g}, vanishing does not occur. The argument given in Case 2 of the proof of Proposition \ref{prop:small-cap-implies-concentration} can be used in conjunction with Corollary \ref{cor:bilinear-separated} to show that the sequence $\{|g_{n_k}|^2\}$ does not satisfy dichotomy either. Therefore it must  satisfy compactness. Thus, there exists a sequence $\{N_k\}\subset\Z$ such that, for every $k\geq 1$ and $\eps>0$, there exists an integer $r=r(\eps)$ for which
\[ \int_{2^{N_k-r}\leq \ab{y}\leq 2^{N_k+r+1}}\ab{g_k(y)}^2\d y\geq 1-\eps. \]
Because of \eqref{lower-boung-g}, the sequence $\{N_k\}$ is bounded, $\sup_{k\geq 1}\ab{N_k}=:r_0<\infty$.
By redefining $r$ as $r+r_0+1$, it follows that
\begin{equation}\label{integral-R}
\int_{2^{-r}\leq \ab{y}\leq 2^{r}}\ab{g_k(y)}^2\d y\geq 1-\eps,\quad\text{ for every }k\geq 1.
\end{equation}
Defining the function 
\[ \theta(R):=\sup_{k\geq 1}\int_{\{R^{-1}\leq \ab{y}\leq R\}^\complement}\ab{g_k(y)}^2\d y, \]
then $R\mapsto \theta(R)$ is a non-increasing function of $R$ which is bounded by 1 and, in view of \eqref{integral-R}, satisfies $\theta(R)\to 0$, as $R\to\infty$. By construction,
\[ \int_{\{R^{-1}\leq \ab{y}\leq R\}^\complement}\ab{g_k(y)}^2\d y\leq \theta(R), \text{ for every }k\geq 1,R\geq 1,\]
which implies \eqref{eq:uniform-l2-decay} at once by taking $\Theta:=\theta^{\frac12}$. 
This concludes the proof.
\end{proof}

 We are finally ready to prove Proposition \ref{prop:existence-vs-concentration}.
\begin{proof}[Proof of Proposition \ref{prop:existence-vs-concentration}]
Let $\{f_n\}$ be as in the statement of the proposition.
Apply Proposition \ref{prop:dyadic-localization} to $\{f_n\}$,
and denote the resulting rescaled subsequence by $\{g_n\}$. 
From the $L^1$ cap estimate \eqref{eq:dyadic-l1-cap-bound} we know that, for each sufficiently large $n$,  
there exists an interval 
$J_n=[s_n-r_n,s_n+r_n]$, contained 
in a dyadic interval\footnote{Or its negative, but in that case we replace $f_n$ by its reflection around the origin.} $[2^{k_n},2^{k_n+1}]$,  such that
\[ \int_{J_n}\ab{g_n}\geq c\ab{J_n}^{\frac12}, \]
for some $c>0$ which is independent of $n$. By the Cauchy--Schwarz inequality,
\begin{equation}\label{lower-bound-l2}
\norma{g_n}_{L^2(J_n)}\geq c,
\end{equation}
and so estimate \eqref{eq:uniform-l2-decay} implies the existence of $C>0$ 
independent of $n$, such 
that $C^{-1}\leq \ab{s_n}\leq C$. Rescaling again, we may assume  $s_n=1$, for every $n$.

If $r^\ast:=\liminf_{n\to\infty}\ab{J_n}>0$, then  passing to the 
relevant subsequence that realizes the limit inferior we have
\[ \int_{1-2r^\ast}^{1+2r^\ast}g_n(y)\d y=\int_{1-2r^\ast}^{1+2r^\ast}\ab{g_n(y)}\d y\geq 
\int_{J_n}\ab{g_n}\gtrsim\sqrt{r^\ast}, \]
provided $n$ is large enough to ensure $J_n\subseteq [1-2r^\ast,1+2r^\ast]$. 
Therefore any $L^2$-weak limit of the sequence $\{g_n\}$ is nonzero. 
Here we used the nonnegativity of the sequence $\{g_n\}$.
By Proposition \ref{prop:new-fvv}, we conclude that there exists $0\neq g\in 
L^2(\R)$, such that possibly after a further extraction, $g_n\to g$ in $L^2(\R)$, as $n\to\infty$.
In other words, (i) holds.

It remains to consider the case when $\ab{J_n}\to 0$, as $n\to\infty$. In view of  \eqref{lower-bound-l2}, Proposition \ref{prop:small-cap-implies-concentration} applies, and the sequence $\{g_n\}$ concentrates at $y_0=1$, i.e. (ii) holds.
 This finishes the proof of Proposition \ref{prop:existence-vs-concentration} (and therefore of Theorem \ref{thm:Thm2}).
 \end{proof}

\section{Existence of extremizers}\label{sec:Existence}

In this section, we prove Theorem \ref{thm:Thm3}. 
The basic strategy is to choose an appropriate trial function $f$ for which  the ratio from \eqref{eq:SharpConvolutionFormGenP},
\begin{equation}\label{eq:RatioToTest}
\Phi_p(f):=\frac{\norma{f\sigma_p\ast f\sigma_p\ast f\sigma_p}_{L^2(\R^2)}^2}{\norma{f}_{L^2(\R)}^6}, 
\end{equation}
can be estimated via a simple lower bound. 
We will give different arguments depending on whether $1<p<2$ or $p>2$, which rely on distinct choices of trial functions. This can be explained by the different qualitative nature of the 3-fold convolutions 
$w\nu_p\ast w\nu_p\ast w\nu_p$ in the two regimes of $p$, see  Figures \ref{fig:graph_triple_convolution_powers} and \ref{fig:graph_triple_convolution_powers_leq2} below.
Here, and throughout this section, $\d\nu_p=\ddirac{s-|y|^p}\d y\d s$ denotes projection measure on the curve $s=|y|^p$, and the weight is given by $w=\ab{\cdot}^{(p-2)/3}$.
 Note that $\d\sigma_p=\sqrt{w}\d\nu_p$.
 
 The following  analogue of \cite[Proposition 6.4]{OSQ16} holds for 3-fold convolutions in $\R^2$.
\begin{proposition}\label{prop:convPowers}
	Given $p>1$,
	 the following assertions hold for $w\nu_p\ast w\nu_p\ast w\nu_p$:
	\begin{itemize}
		\item[(a)] It is absolutely continuous with respect to Lebesgue measure on $\R^2$.
		\item[(b)] Its support, denoted $E_p$, is given by
		\begin{equation}\label{eq:DefEp}
		 E_p=\{(\xi,\tau)\in\R^{2}:\tau\geq {3^{1-p}}{\ab{\xi}^p}\}. 
		 \end{equation}
		\item[(c)] If $p\geq 2$, then its Radon--Nikodym derivative, also denoted by $w\nu_p\ast w\nu_p\ast w\nu_p$, 
		defines a bounded, continuous function in the interior of the set $E_p$. 
		If $1<p<2$, then $w\nu_p\ast w\nu_p\ast w\nu_p$ defines a continuous function on the set 
		$$\widetilde{E}_p:=\{(\xi,\tau)\in\R^{2}:{3^{1-p}}{\ab{\xi}^p}<\tau< {2^{1-p}}{\ab{\xi}^p} \}.$$
		\item[(d)] It is even in $\xi$, 
		\[ (w\nu_p\ast w\nu_p\ast w\nu_p)(-\xi,\tau)=(w\nu_p\ast w\nu_p\ast w\nu_p)(\xi,\tau),\]
		for every $\xi\in\R$, $\tau>0$,
		and homogeneous of degree zero in the sense that 
		$$(w\nu_p\ast w\nu_p\ast w\nu_p)(\la\xi,\la^p\tau)=(w\nu_p\ast w\nu_p\ast 
		w\nu_p)(\xi,\tau), \text{ for every }\la>0.$$
		\item[(e)] It extends continuously  to the boundary of $E_p$, except at the point $(\xi,\tau)=(0,0)$, with values given by
		\begin{equation}\label{eq:BoundaryValues3fold}
		(w\nu_p\ast w\nu_p\ast w\nu_p)(\xi,{{3^{1-p}}\ab{\xi}^p})=\frac{2\pi}{\sqrt{3}p(p-1)}, 
		\text{ if } \xi\neq 0.
		\end{equation}
	\end{itemize}
\end{proposition}
\begin{proof}
For $p\geq 2$, the result follows from \cite[Proposition 2.1]{OSQ18} and \cite[Remark 2.3]{OSQ18}.
If $1<p<2$, then the weight $w$ is singular at the origin, and an additional argument is required in order to establish parts (c) and (e) (as the others follow from \cite{OSQ18}). Note that part (e) also follows from \cite{OSQ18} after we verify (c), and so it suffices to show the latter.

Let $\psi=|\cdot|^p$.
From \cite[Remark 2.3]{OSQ18}, the following formula holds on $\widetilde{E}_p$,
\begin{multline}\label{eq:ContinuityMatters}
(w\nu_p\ast w\nu_p\ast w\nu_p)(\xi,\tau)\\
=\int_{\mathbb{S}^1}\frac{(|\frac{\xi}3+\alpha(\omega_1+\omega_2)||\frac{\xi}3-\alpha\omega_1||\frac{\xi}3-\alpha\omega_2|)^{\frac{p-2}3}}{\langle \omega_1,\frac{\nabla\psi(\xi/3+\alpha\omega_1+\alpha\omega_2)-\nabla\psi(\xi/3-\alpha\omega_1)}{\alpha}\rangle+\langle\omega_2,\frac{\nabla\psi(\xi/3+\alpha\omega_1+\alpha\omega_2)-\nabla\psi(\xi/3-\alpha\omega_2)}{\alpha}\rangle}\d\mu_{(\omega_1,\omega_2)},
\end{multline}
provided that the function $W$ defined by
\begin{equation}\label{eq:DefinitionW}
W(\xi,\omega_1,\omega_2):=(|{\xi}/3+\alpha(\omega_1+\omega_2)||{\xi}/3-\alpha\omega_1||{\xi}/3-\alpha\omega_2|)^{\frac{p-2}3}
\end{equation}
is continuous in the domain of integration. 
Here $\omega_1^2+\omega_2^2=1$, arc length measure on the unit circle $\mathbb{S}^1$ is denoted by $\mu$, and the function $\alpha=\alpha(\xi,\tau,\omega_1,\omega_2)$ is implicitly defined by
\[ \ab{\xi/3+\alpha(\omega_1+\omega_2)}^p+\ab{\xi/3-\alpha\omega_1}^p+\ab{\xi/3-\alpha\omega_2}^p=\tau, \]
see \cite{OSQ18} for details.
It follows that
\[\ab{\xi/3+\alpha(\omega_1+\omega_2)}^p+\ab{\xi/3-\alpha\omega_1}^p+\ab{\xi/3-\alpha\omega_2}^p<2^{1-p}\ab{\xi}^p, \]
provided $(\xi,\tau)\in\widetilde{E}_p$.
On the other hand, if $\xi/3-\alpha\omega_1=0$, then convexity of $\psi$ implies
\[ \ab{2\xi/3+\alpha\omega_2}^p+\ab{\xi/3-\alpha\omega_2}^p\geq 2^{1-p}\ab{\xi}^p, \]
and similarly if $\xi/3-\alpha\omega_2=0$, while if $\xi/3+\alpha(\omega_1+\omega_2)=0$, then 
\[ \ab{\xi/3-\alpha\omega_1}^p+\ab{\xi/3-\alpha\omega_2}^p\geq 2^{1-p}\ab{2\xi/3-\alpha(\omega_1+\omega_2)}^p=2^{1-p}\ab{\xi}^p. \]
It follows that none of these three terms can vanish in a neighborhood of any point $(\xi,\tau)\in\widetilde{E}_p$, and therefore $W$ is continuous there. 
Thus identity \eqref{eq:ContinuityMatters} holds, and this concludes the verification of part (c). 
\end{proof}

The boundedness of $w\nu_p\ast w\nu_p\ast w\nu_p$ 
provides an alternative way towards estimate
	\eqref{eq:SharpConvolutionFormGenP} via the usual application of the Cauchy--Schwarz inequality, at least in the restricted range $p\geq 2$. 
	Moreover, identity \eqref{eq:BoundaryValues3fold} and the argument in Lemma \ref{lem:p-upper-bound-concentration} together imply that the corresponding optimal constant 
	${\bf C}_p$ satisfies
	\begin{equation*}
	 \mathcal {\bf C}^6_p\geq \frac{2\pi}{\sqrt{3}p(p-1)},
	 \end{equation*}
	 which should be compared to \eqref{eq:IneqCriticalValueCp}.

\subsection{Effective lower bounds for ${\bf C}_p$}
We start by examining a simple lower bound, which is the analogue of \cite[Lemma 6.1]{OSQ16} for 3-fold convolutions in $\R^2$.

\begin{lemma}\label{lem:lowerBoundExp}
	Given a strictly convex function $\Psi:\R\to\R$ and a nonnegative function $w:\R\to[0,\infty)$, consider the measures 
	$\d\nu(y,s)=\ddirac{s-\Psi(y)}\d y\d s$ and $\d\sigma=\sqrt{w}\d\nu$. 
	Let $E$ denote the support of the convolution measure $\nu\ast\nu\ast\nu$.
	Given $\lambda>0, a\in\R$, let 
	$f_{\lambda,a}(y):=e^{-\lambda(\Psi(y)+ay)}\sqrt{w(y)}$. 
	Then
	\begin{equation}
	\label{eq:lowerBoundFunctional}
	\frac{\norma{f_{\la,a}\sigma\ast f_{\la,a}\sigma\ast 
			f_{\la,a}\sigma}_{L^2(\R^2)}^2}{\norma{f_{\la,a}}_{L^2(\R)}^6}
	\geq\frac{\norma{f_{\la,a}}_{L^2(\R)}^6}{\int_E e^{-2\la(\tau+a\xi)} \d\xi \d\tau},
	\end{equation}
	for every $f_{\la,a}\in L^2(\R)$ such that 
	$f_{\la,a}\sigma\ast f_{\la,a}\sigma\ast f_{\la,a}\sigma\in L^2(\R^2)$.
\end{lemma}
The proof is entirely parallel to that of \cite[Lemma 6.1]{OSQ16}. 
	Note that \eqref{eq:lowerBoundFunctional} implies
	\begin{equation}\label{eq:OpNormLowerBound} 
	\sup_{0\neq f\in L^2(\R)}\frac{\norma{f\sigma\ast f\sigma\ast f\sigma}_{L^2(\R^2)}^2}{\norma{f}_{L^2(\R)}^6}	\geq\sup_{\la>0,\,a\in\R}\frac{\norma{f_{\la,a}}_{L^2(\R)}^6}{\int_E e^{-2\la(\tau+a\xi)} \d\xi \d\tau}.
	\end{equation}
		Specializing Lemma \ref{lem:lowerBoundExp} to the case of the measure $\sigma_p$ with the natural choice of trail function $f(y)=e^{-\ab{y}^p}\ab{y}^{(p-2)/6}$, a quick computation yields
	\begin{equation}\label{eq:boundGammaPP}
	\Phi_p(f)
	\geq \frac{4\,\Gamma(\frac{p+1}{3p})^3}{3^{1-\frac{1}{p}}p^2\Gamma\bigl(\frac{1}{p}\bigr)}.
	\end{equation}
This lower bound is good enough to establish the strict inequality \eqref{eq:IneqCriticalValueCp} in a range of $p$ that includes the cubic case $p=3$ but {\it not} the quartic case $p=4$, so we have to refine it.
For the above choice of trial function, the corresponding ratio \eqref{eq:RatioToTest} can be expanded as an infinite series with nonnegative terms, whose coefficients are given {in} terms of the Gamma function and whose first term equals the expression on the right-hand side of \eqref{eq:boundGammaPP}.

\begin{proposition}\label{prop:second-lower-bound-Qp}
	Let $p>1$ and $f(y)=e^{-\ab{y}^p}\ab{y}^{(p-2)/6}\in L^2(\R)$. Then
	\begin{equation}\label{eq:series-expansion-p-power}
	\Phi_p(f)
	=\frac{3^{1-\frac{1}{p}}p^2\,\Gamma\bigl(\frac{1}{p}\bigr)}{2^3\,
		\Gamma\bigl(\frac{p+1}{3p}\bigr)^3}\sum_{n=0}^{\infty}(4n+1)2^{4n-1}
	\biggl(\sum_{k=0}^{n}\binom{2n}{2k}\binom{n+k-\frac12}{2n}
	I_{2k}(p)\biggr)^2,
	\end{equation}
	where the coefficients $\{I_{2k}(p)\}_{k\geq 0}$ are given by expression \eqref{eq:MomentExpression} below.
\end{proposition}
 The proof will make use of the classical Legendre polynomials, denoted $\{P_n\}_{n\geq 0}$, which constitute a family of orthogonal polynomials with respect to the $L^2$-norm on the interval $[-1,1]$. Explicitly, they are given by\footnote{Recall that the binomial coefficient $\binom{\alpha}{n}:=\frac{\alpha(\alpha-1)\ldots(\alpha-n+1)}{n!}$ is also defined when $\alpha\notin\Z$.}
 \begin{equation}\label{eq:DefLegendre}
  P_n(t)=2^n\sum_{k=0}^{n}\binom{n}{k}\binom{\frac{n+k-1}{2}}{n}t^k, \;\;\;-1\leq t\leq 1,
  \end{equation}
  from where one checks that  $\langle P_m,P_n\rangle_{L^2}=\frac2{2n+1}\ddirac{n=m}$, see \cite[Corollary 2.16, Chapter 4]{SW71}.
  See also \cite{COS15, CS12a, Fo15, Go17, Ne18} for earlier appearances of Legendre and other families of orthogonal polynomials in sharp Fourier restriction theory.

\begin{proof}[Proof of Proposition \ref{prop:second-lower-bound-Qp}]
Start by noting that the function $f(y)=e^{-\ab{y}^p}\ab{y}^{(p-2)/6}$ coincides with $e^{-\tau}\sqrt{w(\xi)}$ on the support of $\sigma_p$.
Using this together with parts (b) and (d) of Proposition \ref{prop:convPowers}, {we obtain}
\begin{align}
\norma{f\sigma_p\ast f\sigma_p\ast f\sigma_p}_{L^2}^2
&=\norma{e^{-\tau}(w\nu_p\ast w\nu_p\ast w\nu_p)}_{L^2}^2 \notag\\
&=\int_0^{\infty}\int_{-3^{1-\frac1p}\tau^{\frac1p}}^{3^{1-\frac1p}\tau^{\frac1p}} e^{-2\tau}(w\nu_p\ast w\nu_p\ast w\nu_p)^2(\xi,\tau) \d \xi \d \tau \notag\\
&=\int_0^{\infty}\int_{-3^{1-\frac1p}}^{3^{1-\frac1p}} \tau^{\frac1p}e^{-2\tau}(w\nu_p\ast w\nu_p\ast w\nu_p)^2(\tau^{\frac1p}\la,\tau)  \d \la \d \tau \notag\\
&=\Big(\int_0^{\infty}\tau^{\frac1p}e^{-2\tau} \d \tau\Big)\int_{-3^{1-\frac1p}}^{3^{1-\frac1p}} (w\nu_p\ast w\nu_p\ast w\nu_p)^2(\la,1) \d \la \notag\\
&=\frac{3^{1-\frac1p}\Gamma(\tfrac{1}{p})}{p2^{1+\frac1p}}\int_{-1}^{1} (w\nu_p\ast w\nu_p\ast w\nu_p)^2(3^{1-\frac1p}t,1) \d t.\label{eq:LastLineLongComputation}
\end{align}
On the other hand,
\begin{equation}\label{eq:SecondPartLongComputation}
\norma{f}_{L^2}^2
=\int_{\R}e^{-2\ab{y}^p}\ab{y}^{\frac{p-2}3}\d y
=2\int_0^\infty e^{-2y^p}y^{\frac{p-2}3}\d y
=\frac{2^{\frac{2}{3}-\frac{1}{3p}}}{p}\Gamma\Bigl(\frac{p+1}{3p}\Bigr).
\end{equation}
Given  $t\in[-1,1]$, define $g_p(t):=(w\nu_p\ast w\nu_p\ast w\nu_p)\bigl(3^{1-\frac{1}{p}}t,1\bigr)$. 
Expanding $g_p$ in the basis of Legendre polynomials,
\begin{align*} 
\norma{g_p}_{L^2([-1,1],\d t)}^2&= 
\sum_{n=0}^{\infty}\frac{1}{\norma{P_n}_{L^2}^2}\Bigl(\int_{-1}^{1}g_p(t)P_n(t) \d t\Bigr)^2\\
&=\sum_{n=0}^{\infty}(4n+1)2^{4n-1}\biggl(\sum_{k=0}^{n}\binom{2n}{2k}\binom{n+k-\frac{1}{2}}{2n}
\int_{-1}^{1}g_p(t)t^{2k} \d t\biggr)^2,
 \end{align*}
 where the last identity follows from \eqref{eq:DefLegendre}, the normalization $\norma{P_n}_{L^2}^2=\frac2{2n+1}$, and the fact that $g_p$ is an even function of $t$.
 We proceed to find an explicit expression for the moments 
 $I_{n}(p):=\int_{-1}^{1}g_p(t)\,t^n\d t$.
 Given $b\in\R$, we compute:
 \begin{align}
\int_{\R^2}e^{-(\tau-b\xi)}&(w\nu_p\ast w\nu_p\ast w\nu_p)(\xi,\tau) \d \xi\d \tau\notag\\
&=\int_{0}^{\infty} \int_{-3^{1-\frac1p}}^{3^{1-\frac1p}} \tau^{\frac1p}e^{-\tau}e^{b\tau^{\frac1p}\la}(w\nu_p\ast w\nu_p\ast w\nu_p)( \la,1) \d\la\d \tau\notag\\
&=\sum_{n=0}^{\infty}\frac{3^{(1-\frac1p)(2n+1)}b^{2n}}{(2n)!}\Big(\int_{0}^{\infty} e^{-\tau}\tau^{\frac{2n+1}p}\d \tau\Big)\int_{-1}^{1} t^{2n}(w\nu_p\ast w\nu_p\ast w\nu_p)(3^{1-\frac{1}{p}}t,1) \d t\notag\\
&=\sum_{n=0}^{\infty}\frac{3^{(1-\frac1p)(2n+1)}b^{2n}}{(2n)!}\frac{2n+1}{p}\Gamma\Bigl(\frac{2n+1}{p}\Bigr) I_{2n}(p).\label{eq:LastLineLongComputation2}
\end{align}
This Laplace transform can be alternatively computed as follows: 
\begin{align}
\int_{\R^2}e^{-(\tau-b\xi)}(w\nu_p\ast & w\nu_p\ast w\nu_p)(\xi,\tau) \d \xi\d \tau
=\biggl(\int_{\R} e^{-\ab{y}^p}e^{by}\ab{y}^{\frac{p-2}3}\d y\biggr)^3\notag\\
&=\biggl(\sum_{n=0}^{\infty}\frac{2b^{2n}}{(2n)!}\int_0^\infty 
e^{-y^p}y^{\frac{p-2}{3}+2n}\d y\biggr)^3
=\biggl(\sum_{n=0}^{\infty}\frac{2b^{2n}}{p (2n)!}\Gamma\Bigl(\frac{p+1+6n}{3p}\Bigr)\biggr)^3.\label{eq:SecondPartLongComputation2}
\end{align}
Equating coefficients of the same degree, {we obtain} that
\begin{equation}\label{eq:MomentExpression}
I_{2n}(p)
=\frac{2^3\,(2n)!}{3^{(1-\frac{1}{p})(2n+1)}p^2(2n+1)\Gamma\bigl(\frac{2n+1}{p}\bigr)}\sum_{k=0}^n\sum_{m=0}^{n-k}\frac{\Gamma\bigl(\frac{p+1+6k}{3p}\bigr)
	\Gamma\bigl(\frac{p+1+6m}{3p}\bigr)\Gamma\bigl(\frac{p+1+6(n-k-m)}{3p}\bigr)}{(2k)!(2m)!(2(n-k-m))!}.
\end{equation}
Identity \eqref{eq:series-expansion-p-power} follows at once, and the proof is complete.
\end{proof}

\begin{figure}
\centering
\includegraphics[width=1\linewidth]{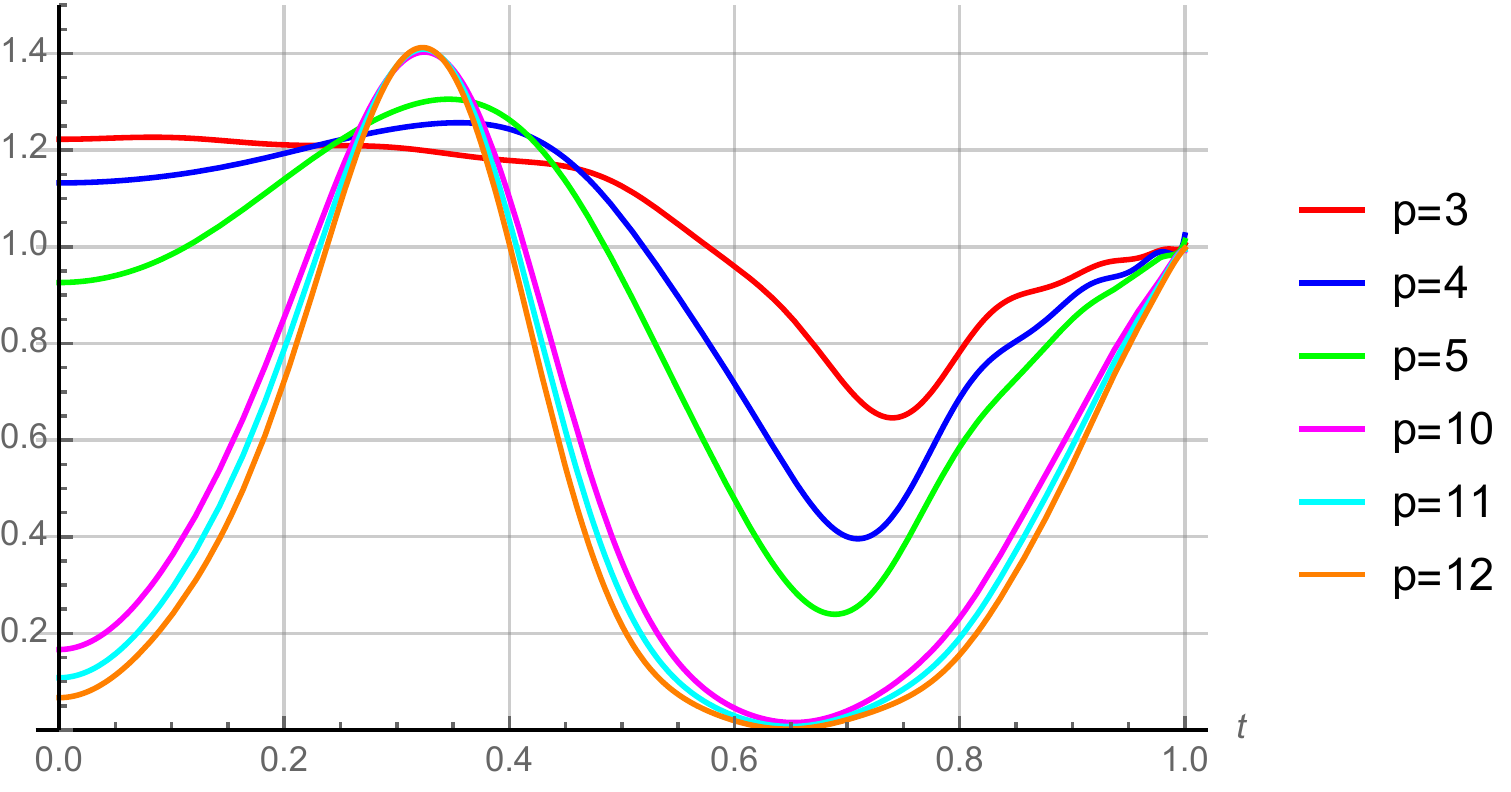}
\caption{Plot of the functions $g_{p,N}(t)$, appropriately normalized  so that they are close to $1$ at $t=1$,  for $p\in\{3,4,5,10,11,12\}$. We used $N=10$ for $p\in\{3,4,5\}$, and $N=15$ for $p\in\{10,11,12\}$.}
\label{fig:graph_triple_convolution_powers}
\vspace{2cm}
\includegraphics[width=1\linewidth]{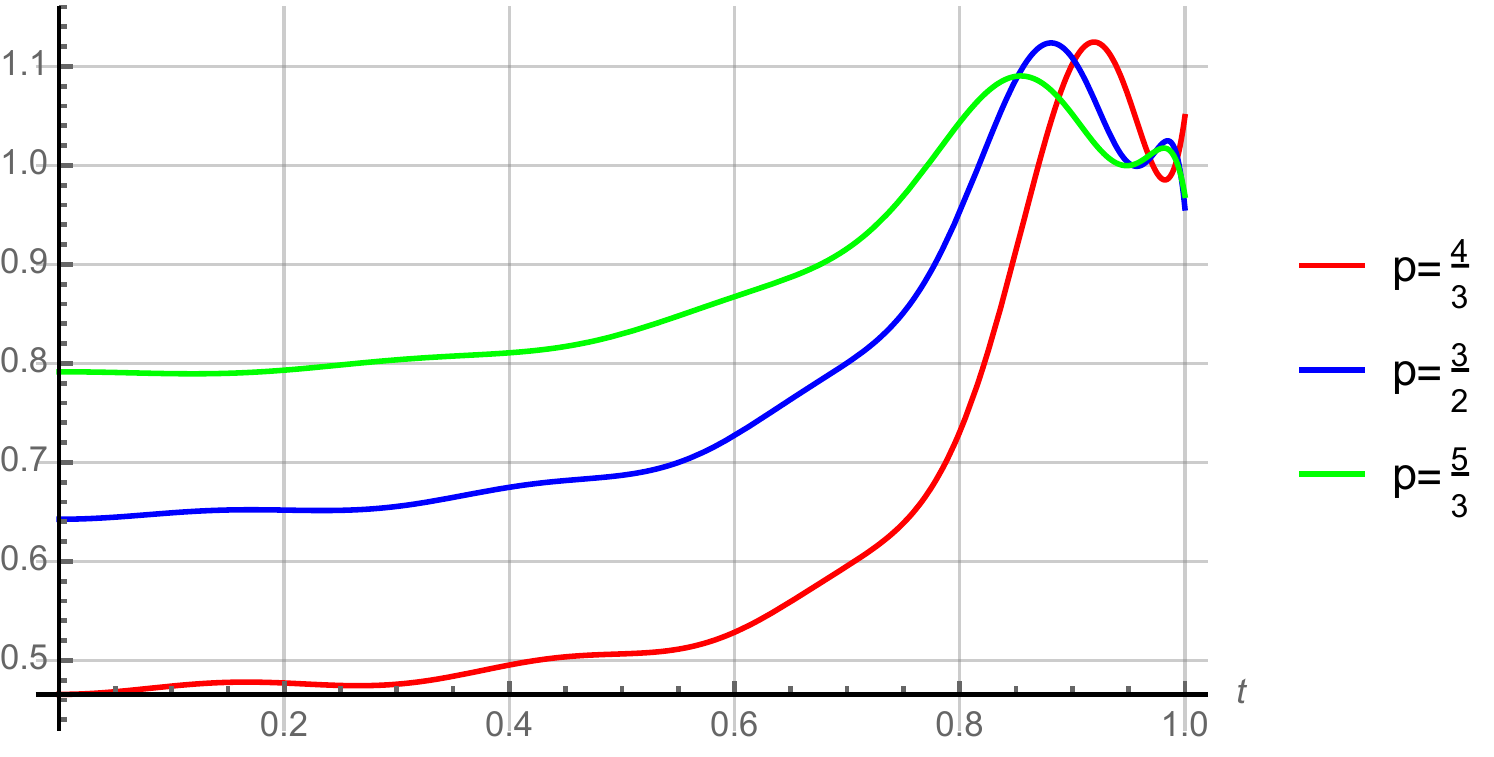}
\caption{Plot of the functions $g_{p,N}(t)$, appropriately normalized  so that they are close to $1$ at $t=1$, for $p\in\{\frac{4}{3},\frac{3}{2},\frac{5}{3}\}$. We used $N=10$.}
\label{fig:graph_triple_convolution_powers_leq2}
\end{figure}

\begin{remark}	
	From the preceding proof, we have the following approximating sequence $\{g_{p,N}\}_{N\geq 0}$ for $g_p$:
	\begin{align*}
	g_{p,N}(t)	:=\sum_{n=0}^{N}(4n+1)2^{2n-1}\biggl(\sum_{k=0}^{n}\binom{2n}{2k}\binom{n+k-\frac{1}{2}}{2n}I_{2k}(p)\biggr)P_{2n}(t),\;\;-1\leq t\leq 1.
	\end{align*}
	This  was used to construct Figures \ref{fig:graph_triple_convolution_powers} and \ref{fig:graph_triple_convolution_powers_leq2}.
	They correspond to approximate  graphs of $w\nu_p\ast w\nu_p\ast w\nu_p$ on the region $\{(\xi,1)\colon 0\leq \xi\leq 3^{1-{1}/{p}} \}$, for different values of $p$. 
	By homogeneity,  the full picture on $\R^2$ can be obtained from these graphs. 
  Figure \ref{fig:graph_triple_convolution_powers} indicates that, for  large $p$, the function $g_p(t)$  becomes small as $t\to 0$. The function $(w\nu_p\ast w\nu_p\ast w\nu_p)(\xi,\tau)$ should then be small near the $\tau$--axis, unlike the case of small values of $p$. This suggests that extremizing sequences may concentrate at the boundary if $p$ is large enough.
\end{remark}

\subsection{Proof of Theorem \ref{thm:Thm3}}
We consider the case $p>2$ first.
From Theorem \ref{thm:Thm2} and Proposition \ref{prop:second-lower-bound-Qp}, it suffices to show that there exists $N\in\N$, such that
\begin{equation}\label{eq:truncated-lower-bound}
\frac{3^{1-\frac{1}{p}}p^2\,\Gamma\bigl(\frac{1}{p}\bigr)}{2^3\,
	\Gamma\bigl(\frac{p+1}{3p}\bigr)^3}\sum_{n=0}^{N}(4n+1)2^{4n-1}
\biggl(\sum_{k=0}^{n}\binom{2n}{2k}\binom{n+k-\frac{1}{2}}{2n}
I_{2k}(p)\biggr)^2>\frac{2\pi}{\sqrt{3}p(p-1)},
\end{equation}
where the coefficients $I_{2k}(p)$ are given by \eqref{eq:MomentExpression}. 
The range of validity of \eqref{eq:truncated-lower-bound} can be estimated by performing an accurate numerical calculation. 
Taking $N=15$, one checks that inequality \eqref{eq:truncated-lower-bound} holds for every $p\in(2,p_0)$, where  $p_0\in[4,5]$ and can be numerically estimated by $p_0\approx 4.803$, with 3 decimal places.
Increasing  the value of $N$ does not seem to substantially increase $p_0$. 

If $1<p<2$,  then inequality \eqref{eq:truncated-lower-bound} fails (for every $N\in\N$).
Incidentally, note that if $p=2$, then the left- and right-hand sides of \eqref{eq:truncated-lower-bound} are equal (for every $N\in\N$) since the 3-fold convolution of projection measure on the parabola is constant inside its support, see \cite[Lemma 4.1]{Fo07}.
We are thus led to a different trial function. For $n\in\N$, define 
\begin{equation}\label{eq:TrialFunction1p2}
f_n(y)=e^{-\frac{n}{2}(\ab{y}^p-py)}\ab{y}^{-\frac{2-p}6}.
\end{equation}
In light of Lemma \ref{lem:p-upper-bound-concentration}, the sequence $\{f_n\norma{f_n}_{L^2}^{-1}\}$ concentrates at $y_0=1$. 
Passing to a continuous parameter $\la>0$, Lemma \ref{lem:lowerBoundExp} yields the lower bound
\[ \Phi_p(f_\la)
\geq \frac{\norma{f_\la}_{L^2(\R)}^6}{\int_{E_p} e^{-\la(\tau-p\xi)}\d\xi \d\tau}=:\phi_p(\la), \]
which we proceed to analyze. Since
\begin{gather*}
\norma{f_\la}_{L^2(\R)}^2=\int_{-\infty}^{\infty} 
e^{-\la(\ab{y}^p-py)}\ab{y}^{-\frac{2-p}3}\d y,\\
\int_{E_p} e^{-\la(\tau-p\xi)}\d \xi\d \tau
=\int_{-\infty}^{\infty}e^{\la p\xi}\Big(\int_{3^{1-p}\ab{\xi}^p}^{\infty} e^{-\la\tau}\d \tau\Big) \d \xi
=\frac{1}{\la}\int_{-\infty}^\infty e^{-\la(3^{1-p}\ab{\xi}^p-p\xi)}\d \xi,
\end{gather*}
we have that
\[ \phi_p(\la)=\la\frac{\Bigl(\int_{-\infty}^{\infty} 
	e^{-\la(\ab{y}^p-py)}\ab{y}^{-\frac{2-p}3}\d y\Bigr)^3}{\int_{-\infty}^\infty 
	e^{-\la(3^{1-p}\ab{\xi}^p-p\xi)}\d \xi}. \]
In view of \eqref{eq:OpNormLowerBound}, we have that ${\bf C}_p^6\geq \phi_p(\la)$, for every $\la>0$. 
Therefore it suffices to show that $\phi_p(\la)>\frac{2\pi}{\sqrt{3}p(p-1)}$, provided $\la$ is large enough. 
This is the content of the following lemma, which we choose to formulate in terms of the function $\varphi_p(\la):=\phi_p(\la^{-1})$.
	
\begin{lemma}\label{lem:perturbative}
	Let $p\in(1,2)$. Then
	\begin{align}
	\label{eq:value0}
	&\lim_{\la\to 0^+}\varphi_p(\la)=\frac{2\pi}{\sqrt{3}p(p-1)},\\
	\label{eq:value-der0}
	&\lim_{\la\to 0^+}\varphi_p'(\la)=\frac{\pi(2-p)(2p-1)}{9\sqrt{3}\,p^2(p-1)^2},
	\end{align}
	In particular, if $\la>0$ is small enough, then $\varphi_p(\la)>\frac{2\pi}{\sqrt{3}p(p-1)}$.
\end{lemma}
	\noindent Note that \eqref{eq:value0} follows from Lemma \ref{lem:p-upper-bound-concentration}, but we choose to present a unified approach that establishes both \eqref{eq:value0} and \eqref{eq:value-der0}.
\begin{proof}[Proof of Lemma \ref{lem:perturbative}]
Rewrite $\phi_p$ in the equivalent form
	\[ \phi_p(\la)=\la\frac{\Bigl(\int_{-\infty}^\infty
		e^{-\la(\ab{y}^p-1-p(y-1))}\ab{y}^{-\frac{2-p}3}\d y\Bigr)^3}{\int_{-\infty}^\infty 
		e^{-\la3^{1-p}(\ab{y}^p-3^p-p3^{p-1}(y-3))}\d y}. \]
	Define real-valued functions $y\mapsto\alpha(y)$ and $y\mapsto\beta(y)$ via\footnote{Note that $\alpha(y)=3^{-2}\beta(3y)$.} 
	\begin{align}
	\ab{y}^p-1-p(y-1)&=\binom{p}{2}\big((y-1)^2+\alpha(y-1)\big),\label{eq:DefAlphaAllR}\\
	\ab{y}^p-3^p-p3^{p-1}(y-3)&=3^{p-2}\binom{p}{2}\big((y-3)^2+\beta(y-3)\big).\notag
	\end{align}
	 By the binomial series expansion, if $\ab{y}<1$, then
	\begin{align}
	\alpha(y)&=\frac{p-2}{3}y^3+\frac{(p-2)(p-3)}{12}y^4+\ldots,\label{eq:defPhiAsympt}\\
	\beta(y)&=\frac{p-2}{3\cdot 3}y^3+\frac{(p-2)(p-3)}{12\cdot 3^2}y^4+\ldots.\label{eq:defPsiAsympt}
	\end{align}
	One easily checks that $|\alpha(y)|\to\infty$, $|\beta(y)|\to\infty$, as $|y|\to\infty$,	
	and 
	\begin{equation}\label{eq:LimitPhiPsiInfinity}
	\lim_{\la\to\infty}\la\alpha(\la^{-\frac12}{y})=\lim_{\la\to\infty}\la\beta(\la^{-\frac12}{y})=0,
	\end{equation}
        for each $y\in\R$. We also have that
 	\begin{align*}
	\int_{\R} \exp\Big(-\la\tfrac{\ab{y}^p-1-p(y-1)}{\binom{p}{2}}\Big)\ab{y}^{-\frac{2-p}3}\d y
	&=\la^{-\frac12}\int_{\R} e^{-y^2}e^{-\la\alpha(\la^{-\frac12}y)}\ab{1+\la^{-\frac12}y}^{-\frac{2-p}3}\d y,
	\end{align*}
	\begin{align*}
	\int_{\R} \exp\Big(-\la{\tfrac{3^{1-p}(\ab{y}^p-3^p-p3^{p-1}(y-3))}{\binom{p}{2}}}\Big)\d y
	&=3^{\frac12}\la^{-\frac12}\int_{\R} e^{-y^2}e^{-\frac \la3\beta((\frac 3\la)^{\frac12}y)}\d y,
	\end{align*}
	and consequently
	\begin{equation*}
	\phi_p(\tfrac{2\la}{p(p-1)})
	=\frac{2}{\sqrt{3}p(p-1)}\frac{\Bigl(\int_{\R} e^{-y^2}e^{-\la\alpha(\la^{-\frac12}y)}\ab{1+\la^{-\frac12}y}^{-\frac{2-p}3}\d y\Bigr)^3}{\int_{\R} e^{-y^2}e^{-\frac \la3\beta((\frac 3\la)^{\frac12}y)}\d y}.
	\end{equation*}
	For bookkeeping purposes, set
	\begin{align*}
	A_p(\la):=\Bigl(\int_{\R} e^{-y^2}e^{-\la\alpha(\la^{-\frac12}y)}\ab{1+\la^{-\frac12}y}^{-\frac{2-p}3}\d y\Bigr)^3,\text{ and }
	B_p(\la):=\int_{\R} e^{-y^2}e^{-\frac \la3\beta((\frac 3\la)^{\frac12}y)}\d y.
	\end{align*}
	We now analyze each expression. Recalling \eqref{eq:defPsiAsympt}, the numerator $A_p(\la)$ is seen to satisfy
	\begin{equation}\label{eq:ApproxA}
	A_p(\la)=\pi^{\frac32}\Bigl(1-\frac{(p-2)(2p-1)}{144\la}+O(\la^{-\frac32})\Bigr)^3, \text{ as } \la\to\infty.
	\end{equation}
	Since binomial series expansions are only valid inside the unit ball, this step requires some care which we now briefly describe.
	Split the integral defining $A_p(\la)$ into three regions,
	$$A_p^{\frac13}(\la)=\Big(\int_{-\infty}^{-\frac{\sqrt{\la}}2}+\int_{-\frac{\sqrt{\la}}2}^{\frac{\sqrt{\la}}2}+\int_{\frac{\sqrt{\la}}2}^{\infty}\Big) e^{-y^2}e^{-\la\alpha(\la^{-\frac12}y)}\ab{1+\la^{-\frac12}y}^{-\frac{2-p}3}\d y=:\text{I}+\text{II}+\text{III},$$
	and estimate each of them separately.
	The main contribution comes from the integral $\text{II}=\text{II}(\la)$.
	Appealing to \eqref{eq:defPhiAsympt} and to the binomial series expansion, we have that
	\begin{align*}
	\exp(-\la\alpha(\la^{-\frac12}y))&=1-\tfrac{p-2}3\la^{-\frac12}{y^3}-\tfrac{(p-2)(p-3)}{12}\la^{-1}{y^4}+\tfrac{(p-2)^2}{18}\la^{-1}{y^6}+O_y(\la^{-\frac32}),\\
	|1+\la^{-\frac12}y|^{-\frac{2-p}{3}}&=1+\tfrac{p-2}{3}\la^{-\frac12}y+\tfrac{(p-2)(p-5)}{18}\la^{-1}y^2+O_y(\la^{-\frac32}),
	\end{align*}
	uniformly in $y\in[-{\sqrt{\la}}/2,{\sqrt{\la}}/2]$. From this one easily checks that 
	$$\text{II}(\la)=\pi^{\frac{1}{2}}+\pi^{\frac{1}{2}}\frac{(p-2)(2p-1)}{144}\la^{-1}+O(\la^{-\frac{3}{2}}).$$
	Matters are thus reduced to verifying that the contributions from $\text{I}$ and $\text{III}$ become negligible, as $\la\to\infty$.
        On the region of integration of $\text{I}=\text{I}(\la)$, the factor $\ab{1+\la^{-1/2}y}^{-\frac{2-p}3}$ has an integrable singularity at $y=-{\lambda}^{1/2}$.
	Recalling the definition \eqref{eq:DefAlphaAllR} of the function $\alpha$, and changing variables $\la^{-1/2}y\rightsquigarrow x$, we have that
	$$\text{I}(\la)=\la^{\frac12}\int_{-\infty}^{-\frac12} e^{-\frac{2\la}{p(p-1)}(|1+x|^p-1-px)}  |1+x|^{-\frac{2-p}3} \d x.$$
	Invoking the elementary inequality $|1+x|^p-1-px\gtrsim_p |x|^p$, which is valid for every $x\leq-\tfrac12$ and $1<p<2$, we may use H\"older's inequality together with the local integrability of $x\mapsto|1+x|^{-\frac{2-p}3}$ in order to bound
         $$\text{I}(\la)=O_p(\la^{\frac12} \exp(-C_p\la)),$$ 
         for some $C_p>0$.
	The contribution of $\text{III}(\la)$ is easier to handle because no singularity occurs on the corresponding region of integration. 
	This concludes the verification of \eqref{eq:ApproxA}, which can then be differentiated term by term because there is sufficient decay. Therefore 
	\[\lim_{\la\to\infty}A_p(\la)= \pi^{\frac32}, \,\text{ and }\,\lim_{\la\to\infty}-\la^2A_p'(\la)= -\frac{3(p-2)(2p-1)\pi^{\frac32}}{144}. \]
	On the other hand, using the binomial series expansion \eqref{eq:defPsiAsympt} we obtain
	\[ \exp\big(-\tfrac \la3\beta((\tfrac 3\la)^{\frac12}y)\big)=1-\tfrac{p-2}{3^{\frac{3}{2}}}\la^{-\frac{1}{2}}y^3-\tfrac{(p-2)(p-3)}{36}\la^{-1}y^4+\tfrac{(p-2)^2}{54}\la^{-1}y^6+O_y(\la^{-\frac{3}{2}}), \]
	uniformly in $y\in[-\frac12(\frac{\la}3)^{\frac{1}{2}},\frac12(\frac{\la}3)^{\frac{1}{2}}]$,
	so that an argument similar to that for $A_p(\la)$ gives
	\[ B_p(\la)=\pi^{\frac{1}{2}}+\frac{(p-2)(2p-1){\pi}^{\frac12}}{144\la}+O(\la^{-\frac32}), \]
	\[\lim_{\la\to\infty}B_p(\la)=\pi^{\frac{1}{2}},\text{ and } \lim_{\la\to\infty}-\la^2B_p'(\la)
	=\frac{(p-2)(2p-1){\pi}^{\frac12}}{144}. \]
	We conclude 
	\[  \lim_{\la\to 0^+}\vphi_p(\la)=\lim_{\la\to \infty}\phi_p(\la)=\lim_{\la\to \infty}\phi_p\Bigl(\frac{2\la}{p(p-1)}\Bigr)=\frac{2\pi}{\sqrt{3}p(p-1)}.\]
	To address \eqref{eq:value-der0}, note that 
	\[\vphi_p'(\la)=-\la^{-2}\phi_p'({\la^{-1}}), 
	\text{ and so }
	\lim_{\la\to 0^+}\vphi_p'(\la)= \lim_{\la\to\infty} -\la^{2}\phi_p'(\la). \]
	Therefore
	\begin{align*}
	\lim_{\la\to \infty}-\la^2\frac{\d}{\d\la}\Big(\phi_p\Bigl(\frac{2\la}{p(p-1)}\Bigr)\Big)&= \frac{2\pi}{\sqrt{3}p(p-1)}\Bigl(-\frac{3(p-2)(2p-1)}{144}-\frac{(p-2)(2p-1)}{144}\Bigr)\\
	&=\frac{\pi(2-p)(2p-1)}{18\sqrt{3}\,p(p-1)},
	\end{align*}
	which readily implies \eqref{eq:value-der0}. This completes the proof of the lemma (and therefore of Theorem \ref{thm:Thm3}).
\end{proof}

\subsection{Improving $p_0$}\label{sec:other-powers-v2}
In view of the results from the last subsection, it is natural to let the functional $\Phi_p$ defined on \eqref{eq:RatioToTest} act on trial functions
$f(y)=e^{-\ab{y}^p}\ab{y}^{(p-2)/6+a}$, for different choices of $a$.\footnote{Note that $L^2$-integrability forces $a>-\frac{p+1}{6}$.}
By doing so, the value $p_0\approx 4.803$ can be improved. We turn to the details.

Set  $\kappa:=\ab{\cdot}^{(p-2)/3+a}$, and note that 
$$(\kappa\nu_p\ast\kappa\nu_p\ast\kappa\nu_p)(\la\xi,\la^{p}\tau)=\la^{3a}(\kappa\nu_p\ast\kappa\nu_p\ast\kappa\nu_p)(\xi,\tau), \text{ for every } \la>0.$$ 
Reasoning as in \eqref{eq:LastLineLongComputation} and \eqref{eq:SecondPartLongComputation}, one checks that

\begin{align*}
\norma{f\sigma_p\ast f\sigma_p\ast f\sigma_p}_{L^2(\R^2)}^2
=\frac{3^{1-\frac1p}\Gamma(\frac{1+6a}{p})}{p 2^{1+\frac{1+6a}{p}}}(1+6a)\int_{-1}^{1} (\kappa\nu_p\ast \kappa\nu_p\ast \kappa\nu_p)^2(3^{1-\frac1p}t,1) \d t,
\end{align*}
\begin{align*}
\norma{f}_{L^2(\R)}^2
=\frac{2^{\frac{2}{3}-\frac{1+6a}{3p}}}{p}\Gamma\Bigl(\frac{p+1+6a}{3p}\Bigr).
\end{align*}
Given $t\in[-1,1]$, define $h_p(t):=(\kappa\nu_p\ast \kappa\nu_p\ast \kappa\nu_p)(3^{1-\frac1p}t,1)$.
Expanding $h_p$ in the basis of Legendre polynomials, 
\[ \norma{h_p}_{L^2([-1,1],\d t)}^2= 
\sum_{n=0}^{\infty}(4n+1)2^{4n-1}\biggl(\sum_{k=0}^{n}\binom{2n}{2k}\binom{n+k-\frac{1}{2}}{2n}
\int_{-1}^{1}h_p(t)t^{2k}\d t\biggr)^2. \]
We proceed to find explicit expressions for the moments  $I_n(p,a):=\int_{-1}^{1} h_p(t)t^{n}\d t$.
Given $b\in\R$, we compute as in \eqref{eq:LastLineLongComputation2} and \eqref{eq:SecondPartLongComputation2}:
\begin{align*}
\int_{\R^2}e^{-(\tau-b\xi)}(\kappa\nu_p\ast \kappa\nu_p\ast \kappa&\nu_p)(\xi,\tau) \d \xi\d \tau\\
&=\sum_{n=0}^{\infty}\frac{3^{(1-\frac1p)(2n+1)}b^{2n}}{(2n)!}\frac{2n+1+3a}{p}\Gamma\Bigl(\frac{2n+1+3a}{p}\Bigr)I_{2n}(p,a)\\
&=\biggl(\sum_{n=0}^{\infty}\frac{2b^{2n}}{p (2n)!}\Gamma\Bigl(\frac{p+1+6n+3a}{3p}\Bigr)\biggr)^3.
\end{align*}
Equating coefficients as before, we find that the moment $I_{2n}(p,a)$ equals
\begin{multline*}
\frac{3^{-(1-\frac1p)(2n+1)}2^3(2n)!}{p^2(2n+1+3a)\Gamma\bigl(\frac{2n+1+3a}{p}\bigr)}
\sum_{k=0}^n\sum_{m=0}^{n-k}\frac{\Gamma\bigl(\frac{p+1+6k+3a}{3p}\bigr)
	\Gamma\bigl(\frac{p+1+6m+3a}{3p}\bigr)\Gamma\bigl(\frac{p+1+6(n-k-m)+3a}{3p}\bigr)}{(2k)!(2m)!(2(n-k-m))!}.
\end{multline*}
This implies
\begin{equation*}
\Phi_p(f)=\frac{3^{1-\frac{1}{p}}p^2\Gamma\bigl(\frac{1+6a}{p}\bigr)}{2^3 
	\Gamma\bigl(\frac{p+1+6a}{3p}\bigr)^3}(1+6a)\sum_{n=0}^{\infty}(4n+1)2^{4n-1}
\biggl(\sum_{k=0}^{n}\binom{2n}{2k}\binom{n+k-\frac{1}{2}}{2n}
I_{2k}(p,a)\biggr)^2,
\end{equation*}
and consequently the following lower bound holds, for every $N\geq 0$:
\begin{equation*}
\Phi_p(f)\geq\frac{3^{1-\frac{1}{p}}p^2\Gamma\bigl(\frac{1+6a}{p}\bigr)}{2^3 
	\Gamma\bigl(\frac{p+1+6a}{3p}\bigr)^3}(1+6a)\sum_{n=0}^{N}(4n+1)2^{4n-1}
\biggl(\sum_{k=0}^{n}\binom{2n}{2k}\binom{n+k-\frac{1}{2}}{2n}
I_{2k}(p,a)\biggr)^2.
\end{equation*}
By numerically evaluating this  sum with $N=15$ and $a=\frac7{15}$, one can establish a lower bound that beats the critical threshold $\frac{2\pi}{\sqrt{3}p(p-1)}$, for every $p\in(2,p_1)$, where $p_1\approx 5.485$ with 3 decimal places.
One further observes that the lower bound for small values of $a>0$ is larger than that for $a=0$, strongly suggesting that the original trial function $y\mapsto e^{-\ab{y}^p}\ab{y}^{(p-2)/6}$ might {\it not} be an extremizer in that range of exponents.

\section{Superexponential $L^2$-decay}\label{sec:Smoothness}
This section is devoted to the proof of Theorem \ref{thm:Thm4}.
We follow the outline of \cite{EHL11, HS12}, and shall sometimes be brief.
The Euler--Lagrange equation associated to \eqref{eq:SharpExtensionFormGenP} is 
\begin{equation}\label{eq:EulerLagrange}
\mathcal{E}_p^\ast \Big( \mathcal{E}_p(f)(\cdot,t) |\mathcal{E}_p(f)(\cdot,t)|^4 \Big)=\lambda f,
\end{equation}
see \cite[Proposition 2.4]{CQ14} for the variational derivation in a related context.
The following 6-linear form will play a prominent role in the analysis:
$$Q(f_1,f_2,f_3,f_4,f_5,f_6):=\int_{\R^2} \prod_{j=1}^3\mathcal{E}_p(f_j)(x,t) \overline{\mathcal{E}_p(f_{j+3})(x,t)} \d x \d t.$$
An immediate consequence of \eqref{eq:SharpExtensionFormGenP} is the following basic estimate:
\begin{equation}\label{eq:BasicEstimateQ}
|Q(f_1,f_2,f_3,f_4,f_5,f_6)|\lesssim\prod_{j=1}^6 \|f_j\|_{L^2(\R)}.
\end{equation}
The form $Q$ can be rewritten as follows:
\begin{equation*}
Q(f_1,f_2,f_3,f_4,f_5,f_6)
=\int_{\R^6} \prod_{j=1}^3 {f_j}(y_j)|y_j|^{\frac{p-2}6}\overline{{f}_{j+3}(y_{j+3})}|y_{j+3}|^{\frac{p-2}6}\ddirac{\alpha({\bf y})}\ddirac{\beta({\bf y})}\d {\bf y},
\end{equation*}
where ${\bf y}=(y_1,\ldots,y_6)\in\R^6$, $\alpha({\bf y}):= |y_1|^p+|y_2|^p+|y_3|^p-|y_4|^p-|y_5|^p-|y_6|^p$, and $\beta({\bf y}):= y_1+y_2+y_3-y_4-y_5-y_6$.
We will also consider the associated form 
$$K(f_1,f_2,f_3,f_4,f_5,f_6):=Q(|f_1|,|f_2|,|f_3|,|f_4|,|f_5|,|f_6|),$$
which is sublinear in each entry. Clearly, 
\begin{gather}
\label{eq:Immediate1}
|Q(f_1,f_2,f_3,f_4,f_5,f_6)|\leq K({f}_1,{f}_2,{f}_3,{f}_4,{f}_5,{f}_6),\\
\label{eq:NonImprovedBoundK}
K(f_1,f_2,f_3,f_4,f_5,f_6)\lesssim \prod_{j=1}^6 \norma{f_j}_{L^2(\R)}.
\end{gather}
Let us now introduce a parameter $s\geq 1$, which will typically be large.
If there exist $j\neq k$ such that $f_j$ is supported on $[-s,s]$ and $f_k$ is supported outside of $[-Cs,Cs]$, for some $C>1$, then  estimate \eqref{eq:NonImprovedBoundK} can be improved to
\begin{equation}\label{eq:ImprovedBoundK}
K(f_1,f_2,f_3,f_4,f_5,f_6)\lesssim C^{-\frac{p-1}6} \prod_{j=1}^6 \norma{f_j}_{L^2(\R)},
\end{equation}
in accordance to the bilinear estimates of Corollary \ref{cor:bilinear-separated}. 
Introducing the weighted variant
$$K_G(f_1,f_2,f_3,f_4,f_5,f_6):=\int_{\R^6} e^{G(y_1)-\sum_{j=2}^6 G(y_j)} \prod_{j=1}^6 |f_j(y_j)| |y_j|^{\frac{p-2}6}\ddirac{\alpha({\bf y})}\ddirac{\beta({\bf y})}\d {\bf y},$$
one checks at once that
\begin{equation}\label{eq:RelKWeightedK}
K(e^G f_1,e^{-G} f_2,e^{-G} f_3,e^{-G} f_4,e^{-G} f_5,e^{-G} f_6)=K_G(f_1,f_2,f_3,f_4,f_5,f_6).
\end{equation}
Given $\mu,\eps\geq 0$, define the function
\begin{equation}\label{eq:DefF}
G_{\mu,\eps}(y):=\frac{\mu|y|^p}{1+\eps|y|^p}.
\end{equation}
The same proof as \cite[Proposition 4.5]{HS12} yields
\begin{equation}\label{eq:BoundWeightedK}
K_{G_{\mu,\eps}}(f_1,f_2,f_3,f_4,f_5,f_6)\leq K(f_1,f_2,f_3,f_4,f_5,f_6),
\end{equation}
see \cite[Remark 4.6]{HS12}.
Split $f=f_<+f_>$ with ${f}_{>}:={f} \one_{[-s^2,s^2]^\complement}$, and define 
$$\|{f}\|_{\mu,s,\eps}:=\|e^{G_{\mu,\eps}}{f}_{>}\|_{L^2}.$$ 
\begin{definition}
A  function $f\in L^2(\R)$ is said to be a {\it weak solution} of  \eqref{eq:EulerLagrange} if there exists $\lambda>0$, such that
\begin{equation}\label{eq:WeakSolution}
Q(g,f,f,f,f,f)=\lambda\langle g,f\rangle_{L^2}, \text{ for every }g\in L^2(\R).
\end{equation}
\end{definition}
\noindent Note that if $f$ extemizes \eqref{eq:SharpExtensionFormGenP}, then $f$ satisfies \eqref{eq:WeakSolution} with $\lambda={\bf E}_p^6\|f\|_{L^2}^4$.
The following key step shows that for some positive $\mu$, the quantity $\|{f}\|_{\mu,s,\eps}$ is bounded in $\eps>0$.

\begin{proposition}\label{prop:KeySmoothnessStep}
Given $p>1$, let $f$ be a weak solution of the Euler--Lagrange equation  \eqref{eq:EulerLagrange} with $\|f\|_{L^2}=1$. If $s\geq 1$ is sufficiently large, then there exists $C<\infty$ such that
\begin{equation}\label{eq:Bootstrap}
\lambda\|{f}\|_{s^{-2p},s,\eps}\leq o_1(1)\|{f}\|_{s^{-2p},s,\eps}+C\sum_{\ell=2}^5\|{f}\|_{s^{-2p},s,\eps}^\ell+o_2(1),
\end{equation}
where  for $j\in\{1,2\}$ the quantity $o_j(1)\to 0$, as $s\to\infty$, uniformly in $\eps$.
Moreover the constant $C$ is independent of $s$ and $\eps$.
\end{proposition}

\begin{proof}
We start by introducing some notation.
Let $G:=G_{\mu,\eps}$ be as in \eqref{eq:DefF}.
Let $h:=e^G{f}$, $h_{>}:=e^G{f}_{>}$ and $h_{<}:=h-h_{>}$.
Further split $f_{<}=f_{\ll}+f_\sim$ and $h_{<}=h_{\ll}+h_\sim$, where ${f}_{\ll}:={f} \one_{[-s,s]}$ and $h_{\ll}:=e^G{f}_{\ll}$.
Since $f$ satisfies \eqref{eq:WeakSolution}, we have that
\begin{multline*}
\lambda\|e^G{f}_{>}\|_{L^2}^2
=\lambda\langle e^{2G}{f}_{>},{f}_{>}\rangle_{L^2}
=\lambda\langle e^{2G}{f}_{>},{f}\rangle_{L^2}
=Q(e^{2G}f_>,f,f,f,f,f)
\\=Q(e^{G}h_>,f,f,f,f,f)
=Q(e^G h_>,e^{-G} h,e^{-G} h,e^{-G} h,e^{-G} h,e^{-G} h)=:Q_G.$$
\end{multline*}
It follows from \eqref{eq:Immediate1}, \eqref{eq:RelKWeightedK} and \eqref{eq:BoundWeightedK} that
$|Q_G|\lesssim K(h_>,h,h,h,h,h).$
Writing $h=h_<+h_>$, the sublinearity of $K$ implies 
$$|Q_G|\lesssim K(h_>,h_<,h_<,h_<,h_<,h_<)+\Big(\sideset{}{'}\sum+\sideset{}{''}\sum\Big)K(h_>,h_{j_2},h_{j_3},h_{j_4},h_{j_5},h_{j_6}),$$
where the first sum, denoted $B_1$, is taken over indices $j_2,\ldots,j_6\in\{>,<\}$ with exactly one of the $j_k$ equal to $>$, and the second sum, denoted $B_2$, is taken over indices $j_2,\ldots,j_6\in\{>,<\}$ with two or more of the $j_k$ equal to $>$.
We estimate the three terms separately. 
For the first one,
\begin{align*}
A:=K(h_>,h_<,h_<,h_<,h_<,h_<)
&\leq K(h_>,h_\ll,h_<,h_<,h_<,h_<)+ K(h_>,h_\sim,h_<,h_<,h_<,h_<)\\
&\lesssim \|h_>\|_{L^2}\big(s^{-\frac{p-1}6}\|h_\ll\|_{L^2}+\|h_\sim\|_{L^2}\big) \|h_<\|_{L^2}^4,
\end{align*}
where we  made use of the support separation of $h_>$ and $h_\ll$ via \eqref{eq:ImprovedBoundK}.
Since $\|f\|_{L^2}=1$, the following estimates hold
\begin{align*}
\|h_<\|_{L^2}\lesssim e^{\mu s^{2p}},\;
\|h_\ll\|_{L^2}\lesssim e^{\mu s^{p}},\;\text{ and }
\|h_\sim\|_{L^2}\lesssim e^{\mu s^{2p}} \|f_\sim\|_{L^2},
\end{align*}
and therefore
$$A\lesssim \|h_>\|_{L^2}\big(s^{-\frac{p-1}6}e^{\mu(s^p-s^{2p})}+\|f_\sim\|_{L^2}\big) e^{5\mu s^{2p}}.$$
The terms $B_1, B_2$ can be estimated in a similar way. One obtains:
\begin{align*}
B_1\lesssim \|h_>\|^2_{L^2}\big(s^{-\frac{p-1}6}e^{\mu(s^p-s^{2p})}+\|f_\sim\|_{L^2}\big) e^{4\mu s^{2p}},\text{ and }
B_2\lesssim \|h_>\|_{L^2} \Big(\sum_{\ell=2}^5 \|h_>\|_{L^2}^\ell\Big)e^{3\mu s^{2p}}.
\end{align*}
The result follows by choosing $\mu=s^{-2p}$ and noting that $\|f_\sim\|_{L^2}\to 0$, as $s\to\infty$.
\end{proof}

We are finally ready to prove that extremizers decay super-exponentially fast.

\begin{proof}[Proof of Theorem \ref{thm:Thm4}]
Let $f\in L^2$ be an extremizer of \eqref{eq:SharpExtensionFormGenP}, normalized so that $\|f\|_{L^2}=1$. 
Then $f$ satisfies \eqref{eq:WeakSolution} with $\lambda={\bf E}_p^6$. 
Note that the function $(s,\eps)\mapsto\norma{ f}_{s^{-2p},s,\eps}$ is continuous in $(s,\eps)\in(0,\infty)^2$ and, for each fixed $\eps>0$,
\begin{equation}\label{eq:norm-s-to-zero}
\norma{f}_{s^{-2p},s,\eps}=\norma{e^{G_{s^{-2p},\eps}}{f}\one_{[-s^2,s^2]^\complement}}_{L^2}\to 0,\quad\text{as }s\to\infty.
\end{equation}
Let $C$ be the constant promised by Proposition \ref{prop:KeySmoothnessStep}, and consider the function
$$H(v):=\frac{\lambda}2v-C(v^2+v^3+v^4+v^5).$$
In \eqref{eq:Bootstrap} choose $s$ sufficiently large so that $o_1(1)\leq\frac{\lambda}2$, for every $\eps>0$.
This is possible since $o_1(1)\to 0$, as $s\to\infty$, uniformly in $\eps>0$.
Consequently,
$$H(\|{f}\|_{s^{-2p},s,\eps})\leq o_2(1),\quad\text{ for every }\eps>0.$$
In view of \eqref{eq:norm-s-to-zero}, and the facts that $H(0)=0$, $H'(0)>0$, and $H$ is concave on $[0,\infty)$, we may choose $s$ sufficiently large so that $\sup_{\eps>0}o_2(1)<H(v_0)$ and $\|{f}\|_{s^{-2p},s,1}\leq v_0$,
where $0<v_0<v_1$ are the two unique positive solutions of the equation
$$H(v_j)=\frac12\max\{H(v):v\geq 0\}.$$
By continuity, 
$\|{f}\|_{s^{-2p},s,\eps}\leq v_0$, for every $\eps>0$.
The Monotone Convergence Theorem then implies
$\|{f}\|_{s^{-2p},s,0}\leq v_0<\infty$, which translates into
$$e^{s^{-2p}|\cdot|^p}{f}\in L^2(\R).$$
Letting $\mu_0:=s^{-2p}$, where $s$ is large enough so that all of the above steps hold, we have thus proved the first part. For the second part, note that, for every $\mu\in\R$, the function
$$e^{\mu|x|}{f}(x)=e^{\mu|x|-\mu_0|x|^p}\cdot e^{\mu_0|x|^p}{f}(x)$$
belongs to $L^2(\R)$, since the first factor is bounded (here we use $p>1$) and the second factor is, as we have just seen, square integrable. The result then follows from the Paley--Wiener theorem as in \cite[Theorem IX.13]{RS75}.
\end{proof}
We finish with two concluding remarks. 
Firstly, the argument  can be adapted to the case of extremizers for odd curves treated in the next section.
Secondly, an interesting problem is whether extremizers are smooth (and not only their Fourier transforms).
This question has been addressed in the context of the Fourier extension operator on low dimensional spheres in \cite{CS12b, Sh16b}, but we have not investigated the extent to which their analysis can be adapted to the present case.

\section{The case of odd curves}\label{sec:Odd}
In this section we discuss the necessary {modifications} to establish analogues of Theorems \ref{thm:Thm2} and \ref{thm:Thm3} for odd curves. 
In general terms, the analysis is similar, but the existence of parallel tangents requires an extra symmetrization step. 
Estimate \eqref{eq:OddSharpConvolutionFormGenP} can  be \mbox{rewritten as}
\begin{equation}\label{eq:OddSharpExtensionFormGenP}
\norma{\mathcal{S}_p (f)}_{L^6(\R^2)}\leq{\bf O}_p \norma{f}_{L^2(\R)},
\end{equation}
where the  Fourier extension operator on the curve $s=y\ab{y}^{p-1}$ is given by 
\begin{equation}\label{eq:odd-fourier-extension-operator}
 \mathcal{S}_p (f)(x,t)=\int_\R e^{ixy}e^{ity\ab{y}^{p-1}}\ab{y}^{\frac{p-2}6}f(y)\d y.
\end{equation}
Given a real-valued function $f\in L^2(\R)$, denote the reflexion of $f$ with respect to the origin by $\tilde{f}:=f(-\cdot)$. 
One easily checks that
\[ \mathcal{S}_p (\tilde{f})(x,t)=\mathcal{S}_p (f)(-x,-t)=\overline{\mathcal{S}_p (f)(x,t)}, \]
where the bar denotes complex conjugation. In particular, 
\[ \norma{\mathcal{S}_p(f)\mathcal{S}_p(g)}_{L^3}=\norma{\mathcal{S}_p(f)\mathcal{S}_p(\tilde{g})}_{L^3}, \]
and so functions $f,g$ supported on intervals $I$ and $-I$, respectively, are seen to interact in the same way as if they were both  supported on $I$, 
unlike the case of even curves. In this way, one is led to symmetrize with respect to reflexion. This has already been observed in the case of the spheres $\mathbb{S}^1$ \cite{Sh16} and $\mathbb{S}^2$ \cite{CS12a}.
Symmetrization on $\mathbb{S}^2$ has been efficiently handled via $\delta$--calculus in \cite{Fo15}. 
The same method can be applied to the present case, but we choose to present a different argument which does not rely on the underlying convolution structure.
\begin{lemma}\label{lem:odd-symmetrization}
	Let $p>1$ and $f\in L^2(\R)$. Then
	\begin{equation}\label{eq:upper-bound-by-even}
	\frac{\norma{\mathcal{S}_p (f)}_{L^6(\R^2)}}{\norma{f}_{L^2(\R)}}\leq \sup_{\substack{0\neq g\in L^2(\R)\\ g\textrm{ even}}}\frac{\norma{\mathcal{S}_p (g)}_{L^6(\R^2)}}{\norma{g}_{L^2(\R)}}.
	\end{equation}
	If equality holds in \eqref{eq:upper-bound-by-even}, then $f$ is necessarily an even function.
\end{lemma}
\begin{proof}
Given $f\in L^2(\R)$, $f\neq 0$, decompose $f=f_e+f_o$, where $f_e$ is an even function, $f_e=\tilde{f}_e$ a.e. in $\R$, and $f_o$ is odd, $f_o=-\tilde{f}_o$ a.e. in $\R$. Then $\norma{f}_{L^2}^2=\norma{f_e}_{L^2}^2+\norma{f_o}_{L^2}^2$, and $\mathcal{S}_p(f_e)$ is real-valued  while $\mathcal{S}_p(f_o)$ is purely imaginary. Thus
\begin{equation}\label{eq:pointwise-l2}
 \ab{\mathcal{S}_p(f)(x,t)}^2=\ab{\mathcal{S}_p(f_e)(x,t)}^2+\ab{\mathcal{S}_p(f_o)(x,t)}^2,\quad\text{ for almost every }(x,t)\in\R^2, 
 \end{equation}
and so, by the triangle inequality for the $L^3$-norm,
$\norma{\mathcal{S}_p(f)}^2_{L^6}
\leq \norma{\mathcal{S}_p(f_e)}_{L^6}^2+\norma{\mathcal{S}_p(f_o)}_{L^6}^2.$
It follows that
\[ \frac{\norma{\mathcal{S}_p(f)}_{L^6}^2}{\norma{f}_{L^2}^2}\leq \frac{\norma{\mathcal{S}_p(f_e)}_{L^6}^2+\norma{\mathcal{S}_p(f_o)}_{L^6}^2}{\norma{f_e}_{L^2}^2+\norma{f_o}_{L^2}^2}\leq \max\biggl\{\frac{\norma{\mathcal{S}_p(f_e)}_{L^6}^2}{\norma{f_e}_{L^2}^2},\frac{\norma{\mathcal{S}_p(f_o)}_{L^6}^2}{\norma{f_o}_{L^2}^2}\biggr\}, \]
where we set either ratio on the right-hand side of this chain of inequalities to zero whenever the corresponding function $f_e$ or $f_o$ happens to vanish identically.
Therefore we may restrict attention to functions which are either even or odd. 
On the other hand, the equivalent convolution form \eqref{eq:OddSharpConvolutionFormGenP} of the inequality implies $\norma{\mathcal{S}_p(g)}_{L^6}\leq \norma{\mathcal{S}_p(\ab{g})}_{L^6}$, with equality if and only if $g=\ab{g}$ a.e. in $\R$. Thus
\begin{equation}\label{eq:inequality-even-and-odd}
\frac{\norma{\mathcal{S}_p(f)}_{L^6}^2}{\norma{f}_{L^2}^2}\leq  \max\biggl\{\frac{\norma{\mathcal{S}_p(f_e)}_{L^6}^2}{\norma{f_e}_{L^2}^2},\frac{\norma{\mathcal{S}_p(\ab{f_o})}_{L^6}^2}{\norma{f_o}_{L^2}^2}\biggr\}\leq \sup_{\substack{0\neq g\in L^2\\ g\textrm{ even}}}\frac{\norma{\mathcal{S}_p(g)}_{L^6}}{\norma{g}_{L^2}},
\end{equation}
where we used that both $f_e$ and $\ab{f_o}$ are even functions. In order for equality to hold in \eqref{eq:upper-bound-by-even}, both inequalities in \eqref{eq:inequality-even-and-odd} must be equalities. Inspection of the chain of inequalities leading to \eqref{eq:inequality-even-and-odd} shows that, if there is equality in the first inequality, then necessarily one of the following alternatives must hold:
\begin{itemize}
	\item $\norma{f_o}_{L^2}=0$, in which case $f=f_e$, and so $f$ is even; or
	\item $\norma{f_e}_{L^2}=0$ and $f_o=\ab{f_o}$ a.e. in $\R$, which implies that $f_o\equiv 0$, and so $f\equiv 0$ which does not hold by assumption; or
	\item $\norma{f_e}_{L^2}\norma{f_o}_{L^2}\neq 0$  and ${\norma{\mathcal{S}_p(f_e)}_{L^6}}{\norma{f_e}_{L^2}^{-1}}={\norma{\mathcal{S}_p(f_o)}_{L^6}}{\norma{f_o}_{L^2}^{-1}}={\norma{\mathcal{S}_p(\ab{f_o})}_{L^6}}{\norma{f_o}_{L^2}^{-1}}$, which again forces $f_o=\ab{f_o}$ a.e. in $\R$, so that $f_o=0$ which is absurd.
\end{itemize}
Therefore equality in \eqref{eq:upper-bound-by-even} forces $f$ to be an even function, as desired.
\end{proof}

For the remainder of this section, we restrict attention to nonnegative, even functions $f$. To prove the analogue of Proposition \ref{prop:existence-vs-concentration}, we need bilinear estimates as in Propositions \ref{prop:bilinear-p-power} and \ref{prop:unc-bilinear-p-power}, and an $L^1$ cap bound as in Proposition \ref{prop:l1-cap-bound-p-power}. These can be obtained in exactly the same way as for the case of even curves, since the Jacobian factor corresponding to \eqref{eq:jacobian} is now equal to
$p\ab{\ab{y'}^{p-1}-\ab{y}^{p-1}}$,
which amounts to the bound we used before.
We also need an analogue of Proposition \ref{prop:metric-concentration-compactness-lemma} with two points removed, i.e.\@ consider $X_{\bar{x},\bar{y}}:=X\setminus\{\bar{x},\bar{y}\}$ equipped with a pseudometric $\vrho:X_{\bar{x},\bar{y}}\times X_{\bar{x},\bar{y}}\to [0,\infty)$. The statement is analogous and we omit the obvious writing. Next,  defining the dyadic pseudometric centered at zero as in \eqref{eq:pseudometric-center-zero} and invoking the appropriate bilinear estimates, we obtain an analogue of Proposition \ref{prop:dyadic-localization}, the statement again being identical (omitted).
The analogue of Proposition \ref{prop:small-cap-implies-concentration} requires the pseudometric
\[ \vrho\colon\R\setminus\{-1,1\}\times \R\setminus\{-1,1\}\to [0,\infty),\quad \vrho(x,y):=\ab{k-k'}, \]
where $k,k'\in \Z$ are such that $\ab{\ab{x}-1}\in [2^k,2^{k+1})$ and $\ab{\ab{y}-1}\in [2^{k'},2^{k'+1})$.
It handles concentration at a pair of opposite points, which we now define.
\begin{definition}
Let $y_0\in\R$.
A sequence of even functions $\{f_n\}\subset L^2(\R)$ {\it concentrates at the pair} $\{-y_0,y_0\}$  if, for every $\eps,\rho>0$, there exists $N\in\N$ such that, for every $n\geq N$,
\[ \int_{\substack{\ab{y+y_0}\geq \rho,\\ \ab{y-y_0}\geq \rho}}\ab{f_n(y)}^2\d y<\eps\norma{f_n}_{L^2(\R)}^2. \]
\end{definition}
\noindent The following analogue of Proposition \ref{prop:small-cap-implies-concentration}  holds for odd curves.
\begin{proposition}\label{prop:small-cap-implies-concentration-odd}
	Let $\{f_n\}\subset L^2(\R)$ be an $L^2$-normalized extremizing sequence of even functions for \eqref{eq:OddSharpExtensionFormGenP}. Let $\{r_n\}$ be a sequence of nonnegative numbers, satisfying $r_n\to0$, as $n\to\infty$, and
	\[ \inf_{n\in\N}\int_{1-r_n}^{1+r_n}\ab{f_n(y)}^2\d y> 0. \]
	Then the sequence $\{f_n\}$ concentrates at the pair $\{-1,1\}$.
\end{proposition}
\noindent As in the case of even curves, this can be used to prove the analogue of {Proposition  \ref{prop:existence-vs-concentration}}.
\begin{proposition}\label{prop:existence-vs-concentration-odd}
	Let $\{f_n\}\subset L^2(\R)$ be an $L^2$-normalized extremizing sequence of nonnegative, even functions for  \eqref{eq:OddSharpExtensionFormGenP}. 
	Then there exists a subsequence $\{f_{n_k}\}$, and a sequence $\{a_k\}\subset \R\setminus\{0\}$,  such that the rescaled sequence $\{g_k\}$, $g_k:=\ab{a_k}^{1/2}f_{n_k}(a_k \cdot)$, satisfies one of the 
	following conditions:
	\begin{enumerate}
		\item[(i)] There exists $g\in L^2(\R)$ such that $g_{k}\to g$ in $L^2(\R)$, as $k\to\infty$, or
		\item[(ii)] $\{g_k\}$  concentrates at the pair $\{-1,1\}$.
	\end{enumerate}
\end{proposition}

Let $\{f_n\}\subset L^2(\R)$ be an $L^2$-normalized sequence of nonnegative, even functions concentrating at the pair $\{-1,1\}$. 
Write $f_n=g_n+\tilde{g}_n$, where $g_n:=f_n\one_{[0,\infty)}$. In particular,  $\norma{g_n}_{L^2}={2}^{-1/2}$, and the sequence $\{g_n\}$ concentrates at $y_0=1$. 
The left-hand side of \eqref{eq:OddSharpConvolutionFormGenP} can be expanded into
\begin{align}
\norma{f_n\mu_p\ast f_n\mu_p\ast f_n\mu_p}_{L^2}^2&= \norma{g_n\mu_p\ast 
	g_n\mu_p\ast g_n\mu_p}_{L^2}^2+\norma{\tilde{g}_n\mu_p\ast \tilde{g}_n\mu_p\ast 
	\tilde{g}_n\mu_p}_{L^2}^2\\
&\quad+9\norma{g_n\mu_p\ast g_n\mu_p\ast \tilde{g}_n\mu_p}_{L^2}^2+9\norma{g_n\mu_p\ast 
	\tilde{g}_n\mu_p\ast \tilde{g}_n\mu_p}_{L^2}^2\notag\\
&\quad+6\langle g_n\mu_p\ast g_n\mu_p\ast g_n\mu_p\;,\;g_n\mu_p\ast g_n\mu_p\ast 
\tilde{g}_n\mu_p\rangle_{L^2}\notag\\
&\quad+6\langle g_n\mu_p\ast \tilde{g}_n\mu_p\ast \tilde{g}_n\mu_p\;,\;\tilde{g}_n\mu_p\ast \tilde{g}_n\mu_p\ast 
\tilde{g}_n\mu_p\rangle_{L^2}\notag\\
&\quad+18\langle g_n\mu_p\ast \tilde{g}_n\mu_p\ast \tilde{g}_n\mu_p\;,\;g_n\mu_p\ast g_n\mu_p\ast 
\tilde{g}_n\mu_p\rangle_{L^2}\notag\\
&\quad+6\langle g_n\mu_p\ast g_n\mu_p\ast g_n\mu_p\;,\;g_n\mu_p\ast 
\tilde{g}_n\mu_p\ast \tilde{g}_n\mu_p\rangle_{L^2}\notag\\
&\quad+6\langle g_n\mu_p\ast g_n\mu_p\ast \tilde{g}_n\mu_p\;,\;\tilde{g}_n\mu_p\ast 
\tilde{g}_n\mu_p\ast \tilde{g}_n\mu_p\rangle_{L^2}\notag\\
&\quad+2\langle g_n\mu_p\ast g_n\mu_p\ast g_n\mu_p\;,\;\tilde{g}_n\mu_p\ast 
\tilde{g}_n\mu_p\ast \tilde{g}_n\mu_p\rangle_{L^2}\notag.
\end{align}
The last three summands vanish since the corresponding supports intersect on a Lebesgue null set. 
The symmetry of the inner products then implies
\begin{multline*}
\norma{f_n\mu_p\ast f_n\mu_p\ast f_n\mu_p}_{L^2}^2\\
= 20\norma{g_n\mu_p\ast g_n\mu_p\ast g_n\mu_p}_{L^2}^2+30\langle g_n\mu_p\ast g_n\mu_p\ast g_n\mu_p\ast g_n\mu_p\;,\,  g_n\mu_p\ast g_n\mu_p\rangle_{L^2}.
\end{multline*}
Note that $\mu_p=\sigma_p$ on the support of $g_n$, where $\sigma_p$ was defined in \eqref{eq:DefSigmaMeasure}. It follows that
\begin{multline}\label{eq:even-odd-relationship-2}
\frac{\norma{f_n\mu_p\ast f_n\mu_p\ast f_n\mu_p}_{L^2}^2}{\norma{f_n}_{L^2}^6}\\
= \frac{5}{2}\frac{\norma{g_n\sigma_p\ast g_n\sigma_p\ast g_n\sigma_p}_{L^2}^2}{\norma{g_n}_{L^2}^6}
+\frac{15}{4}\frac{\langle g_n\sigma_p\ast g_n\sigma_p\ast g_n\sigma_p\ast g_n\sigma_p\;,\,  g_n\sigma_p\ast g_n\sigma_p\rangle_{L^2}}{\norma{g_n}_{L^2}^6}.
\end{multline}
Since the sequence $\{g_n\}$ concentrates at $y_0=1$, we have that
\[ \lim_{n\to\infty}\langle g_n\sigma_p\ast g_n\sigma_p\ast g_n\sigma_p\ast g_n\sigma_p\;,\,  g_n\sigma_p\ast g_n\sigma_p\rangle_{L^2}=0. \]
Heuristically, $g_n\sigma_p\ast g_n\sigma_p$ is supported near the point $(2,2)$, while 
$(g_n\sigma_p)^{\ast(4)}$ is supported near the point $(4,4)$, and so in the limit there is no contribution of the inner product. More precisely, given $\eps>0$, write $g_n=h_n+\kappa_n$, where $h_n:=g_n\one_{[1-\eps,1+\eps]}$ 
and $\norma{\kappa_n}_{L^2}^2\to 0$, as $n\to\infty$. 
If $\eps$ is  small enough, then support considerations force
\[\langle h_n\sigma_p\ast h_n\sigma_p\ast h_n\sigma_p\ast h_n\sigma_p\;,\; h_n\sigma_p\ast 
h_n\sigma_p\rangle_{L^2}=0,\text{ for every }n,\] 
whereas the cross terms involve $\kappa_n$, whose $L^2$-norm tends to zero, as $n\to\infty$.
 We conclude
\begin{equation}\label{eq:even-odd-concentration-relation}
\limsup_{n\to\infty}\frac{\norma{f_n\mu_p\ast f_n\mu_p\ast f_n\mu_p}_{L^2}^2}{\norma{f_n}_{L^2}^6}
=\frac{5}{2}\limsup_{n\to\infty}\frac{\norma{g_n\sigma_p\ast g_n\sigma_p\ast g_n\sigma_p}_{L^2}^2}{\norma{g_n}_{L^2}^6},
\end{equation}
and similarly for the limit inferior. 
Lemma \ref{lem:p-upper-bound-concentration} applied to the sequence $\{g_n\}$ implies
\[ \limsup_{n\to\infty}\frac{\norma{f_n\mu_p\ast f_n\mu_p\ast f_n\mu_p}_{L^2(\R^2)}^2}{\norma{f_n}_{L^2}^6}
\leq \frac{5\pi}{\sqrt{3}p(p-1)}. \]
Moreover, equality holds if we take $f_n=g_n+\tilde{g}_n$,  with $g_n:=2^{-1/2}h_n\norma{h_n}_{L^2}^{-1}$, and
\[ h_n(y):=e^{-n(\ab{y}^p-1-p(y-1))}\ab{y}^{\frac{p-2}6}\one_{[0,\infty)}(y). \]
Theorem \ref{thm:Thm5} is now proved.

\begin{remark}
The invariant form of condition \eqref{eq:IneqCriticalValueOp} in Theorem \ref{thm:Thm5} is
\begin{equation}\label{eq:InvariantThreshold}
\biggl(\frac{{\bf Q}_p}{{\bf C}_2}\biggr)^6>\frac{5}{p(p-1)}, 
\end{equation}
where ${\bf C}_2^6=\pi/\sqrt{3}$ is the best constant for the parabola in convolution form.
In the case $p=3$, a similar condition  appears in previous work of Shao \cite{Sh09} on the  Airy--Strichartz inequality which translates into $(\frac{{\bf Q}_3}{{\bf C}_2})^6>\frac{1}{3}$.  
This is of course incompatible with \eqref{eq:InvariantThreshold} but, as was recently pointed out in \cite[Remark 2.7]{FS17},  there is a problem in \cite[Lemma 6.1]{Sh09} in the passage from Eq. (89) to Eq. (90), as the argument presented there disregards the effect of symmetrization.
{On the other hand, the case $p=3$ of \eqref{eq:InvariantThreshold} agrees with \cite[Case $p=q=6$ of Theorem 1]{FS17}, once the proper normalization is considered.}
\end{remark}
We now come to the question of whether extremizers for \eqref{eq:OddSharpConvolutionFormGenP} actually exist, and discuss the case $1<p<2$ first.
Just as in \eqref{eq:TrialFunction1p2}, set $g_n(y):=e^{-\frac n2(|y|^p-py)}|y|^{-\frac{2-p}6}$.
	Its even extension, 
	\[f_n:=\frac{g_n\one_{[0,\infty)}+\tilde{g}_n\one_{(-\infty,0]}}{2^{\frac12}\norma{g_n}_{L^2(0,\infty)}},\]
	can be used to establish the strict inequality in \eqref{eq:IneqCriticalValueOp}. One simply uses \eqref{eq:even-odd-concentration-relation} together with the fact that the sequence $\{g_n\|g_n\|_{L^2}^{-1}\}_{s>0}$ concentrates at $y_0=1$, {so that an argument  similar to Lemma \ref{lem:perturbative} can be applied to the present case}.
	Therefore, extremizers for \eqref{eq:OddSharpConvolutionFormGenP} exist if $1<p<2$, and Theorem \ref{thm:Thm5.5} is now proved.

The case  $p\geq 2$ seems harder.
In view of \eqref{eq:even-odd-concentration-relation}, it is natural to use the methods of \S \ref{sec:Existence} in order to find the series expansion for the trial functions $f=2^{-1/2}(g+\tilde{g})$, where $g(y)=e^{-\ab{y}^p}\ab{y}^{(p-2)/6+a}\one_{[0,\infty)}(y)$, for different choices of $a$. By doing so, we find that we cannot reach the critical threshold $\frac{5\pi}{\sqrt{3}p(p-1)}$, but that we can approach it from below by varying the value of $a$.
We are led to the following conjecture.
\begin{conjecture}\label{cjct:nonexistence}
	For every $p\geq 2$,
	\[ \biggl(\frac{{\bf Q}_p}{{\bf C}_2}\biggr)^6=\frac{5}{p(p-1)}. \]
	Moreover, extremizers for \eqref{eq:OddSharpConvolutionFormGenP} do not exist.
\end{conjecture}

\subsection{On symmetric complex- and real-valued extremizers}\label{sec:symmetry}
	The proof of Lemma \ref{lem:odd-symmetrization} merits some further remarks which we attempt to insert within a broader context. 
	
	First of all,  identity \eqref{eq:pointwise-l2} holds thanks to the symmetry with respect to the origin of both the curve $s=y\ab{y}^{p-1}$ and the measure $\d\mu_p=\ddirac{t-y\ab{y}^{p-1}}\ab{y}^{(p-2)/6}\d y\d s$.
	In fact, the proof of Lemma \ref{lem:odd-symmetrization} immediately generalizes to the Fourier extension operator associated to any {\it antipodally symmetric pair} $(\Sigma,\mu)$. By this we mean a set $\Sigma\subseteq \R^d$ (usually a smooth submanifold) together with a Borel measure $\mu$ supported on $\Sigma$,  both symmetric with respect to the origin in the sense that $T(\Sigma)=\Sigma$ and  $T^\ast\mu=\mu$, where $T$ denotes the antipodal map $T(y)=-y$ and $T^* \mu$ denotes the  pushforward measure. 
	
	Secondly, the Lebesgue exponent $6$ can be replaced with any finite exponent $r\geq 2$. More precisely, in the general context of an antipodally symmetric pair $(\Sigma,\mu)$, if an estimate 
	\begin{equation}\label{eq:generalExtension}
	\norma{\widehat{f\mu}}_{L^r(\R^d)}\lesssim\norma{f}_{L^2(\Sigma,\mu)}
	\end{equation}
	does hold for some $r\in [2,\infty)$, then necessarily\footnote{Here, a real-valued function $g: \Sigma\to\R$ is naturally defined to be  {\it even} (resp. {\it odd}) if $g(y)=g(-y)$ (resp. $g(y)=-g(-y)$), for $\mu$-almost every point $y\in\Sigma$.}
	\[ \sup_{\substack{0\neq f\in L^2(\Sigma,\mu)\\f\;\R \text{-valued}}}\frac{\norma{\widehat{f\mu}}_{L^r(\R^d)}}{\norma{f}_{L^2(\Sigma,\mu)}}= \sup_{\substack{0\neq g\in L^2(\Sigma,\mu)\\g\; \R\text{-valued, } g\text{ even or }g \text{ odd}}}\frac{\norma{\widehat{g\mu}}_{L^r(\R^d)}}{\norma{g}_{L^2(\Sigma,\mu)}}. \]

	Thirdly,  the discussion extends to the more general situation of complex-valued functions.
	 For concreteness, let us specialize to the case of the unit sphere $\Sigma=\mathbb S^{d-1}\subseteq \R^d$, $d\geq 2$, equipped with its natural surface measure $\mu$. Given an exponent $p\geq p_d:=\frac{2(d+1)}{d-1}$,  the  Tomas--Stein inequality states that
	\begin{equation}\label{eq:Thomas-Stein}
	\norma{\widehat{(u\mu)}}_{L^p(\R^d)}\lesssim_{p,d}\norma{u}_{L^2(\mathbb S^{d-1})},
	\end{equation}
	for every complex-valued function $u\in L^2(\mathbb S^{d-1})$.
It is known  \cite{FVV11,FLS16} that complex-valued extremizers for \eqref{eq:Thomas-Stein}  exist in the full range $p\geq p_d$, the endpoint existence in dimensions $d\geq 4$ being conditional on a celebrated conjecture concerning \eqref{eq:BestConstantStrSchr}. 
	Moreover, if $p\geq p_d$ is an even integer, then real-valued, even, nonnegative extremizers for \eqref{eq:Thomas-Stein} exist, by virtue of the equivalent convolution form, see \cite{CS12a,Fo15,Sh16}. 
	Finally, if $p=\infty$, then one easily checks that the unique extremizers for \eqref{eq:Thomas-Stein} are the constant functions. For general $p\geq p_d$, $p\neq\infty$, we argue that the search for extremizers of \eqref{eq:Thomas-Stein} can be restricted to the class of complex-valued, symmetric functions. Indeed, write $u=f+ig$, with $f=\Re u,\, g=\Im u$. By reorganizing the summands, we may write $u=F+iG$, where $F=f_e+ig_o$ and $G=g_e-if_o$. The functions $F,G$ are  complex-valued and symmetric, in the sense that $F(y)=\overline{F(-y)}$ and $G(y)=\overline{G(-y)}$, for every $y\in\mathbb S^{d-1}$. Moreover, one easily checks that $F(y)=\frac{1}{2}(u(y)+\overline{u(-y)}),\, G(y)=\frac{1}{2i}(u(y)-\overline{u(-y)})$,  $\norma{u}_{L^2}^2=\norma{F}_{L^2}^2+\norma{G}_{L^2}^2$ and that, in view of the antipodal symmetry of the pair $(\mathbb S^{d-1},\mu)$, the functions $\widehat{F\mu},\, \widehat{G\mu}$ are real-valued. Following the proof of Lemma \ref{lem:odd-symmetrization}, we are thus led to the following result.
	
	\begin{proposition}\label{prop:symmetrized-extremizers-sphere}
		Let $d\geq 2$ and $ \frac{2(d+1)}{d-1}\leq p\leq\infty$. Then for every complex-valued $u\in L^2(\mathbb S^{d-1}),\, u\neq 0$, the following inequality holds:
		\begin{equation}\label{eq:symmetrization-sphere}
		\frac{\norma{\widehat{(u\mu)}}_{L^p(\R^d)}}{\norma{u}_{L^2(\mathbb S^{d-1})}}\leq \sup_{0\neq F\in L_{\emph{sym}}^2(\mathbb S^{d-1})} \frac{\norma{\widehat{F\mu}}_{L^p(\R^d)}}{\norma{F}_{L^2(\mathbb S^{d-1})}},
		\end{equation}
		where $L^2_{\emph{sym}}(\mathbb S^{d-1}):=\{F\in L^2(\mathbb S^{d-1})\colon F(y)=\overline{F(-y)},\; \text{ for }\mu\text{-a.e. }y\in\mathbb S^{d-1} \}$. Moreover, if $u$ realizes equality in \eqref{eq:symmetrization-sphere}, then there exist $F\in L^2_{\emph{sym}}(\mathbb S^{d-1})$ and a constant $\kappa\in\mathbb C$ such that $u=\kappa F$, $\mu$-a.e.
	\end{proposition}
	\begin{proof}
		In light of the previous discussion, we can assume $p<\infty$, and only the last statement merits further justification. Suppose that $u$ realizes equality in \eqref{eq:symmetrization-sphere}. 
		In particular, $u$ is a complex-valued extremizer for \eqref{eq:Thomas-Stein}. 
		Decompose $u=F+iG$ as before, with $F(y)=\frac{1}{2}(u(y)+\overline{u(-y)})$, $G=\frac{1}{2i}(u(y)-\overline{u(-y)})$, so that $F,G\in L^2_{\text{sym}}(\mathbb S^{d-1})$. If either $F\equiv 0$ or $G\equiv 0$, then there is nothing to prove, and so in what follows we assume $F,G$ not to be identically zero. Following the proof of Lemma \ref{lem:odd-symmetrization}, we note that equality occurs in the application of the triangle inequality with respect to the $L^{p/2}(\R^d)$-norm (recall that $p/2> 1$ is finite) only if there exists $\la>0$, such that\footnote{As Fourier transforms of compactly supported distributions, both sides of \eqref{eq:pointwise-equality-fourier} coincide with the absolute value of real-valued, {\it smooth} functions, so that the pointwise equality occurs at every point, and not just almost everywhere.}
		\begin{equation}\label{eq:pointwise-equality-fourier}
		\ab{\widehat{F\mu}(\xi)}=\la\ab{\widehat{G\mu}(\xi)}, \text{ for every }\xi\in\R^d.
		\end{equation}
Subsequent cases of equality further imply
		\[ \frac{\norma{\widehat{(u\mu)}}_{L^p(\R^d)}}{\norma{u}_{L^2(\mathbb S^{d-1})}}= \frac{\norma{\widehat{F\mu}}_{L^p(\R^d)}}{\norma{F}_{L^2(\mathbb S^{d-1})}}=\frac{\norma{\widehat{G\mu}}_{L^p(\R^d)}}{\norma{G}_{L^2(\mathbb S^{d-1})}},  \]
		and so the functions $F, G$ are also extremizers for \eqref{eq:Thomas-Stein}. It suffices to show that $F=\kappa G$, where $\kappa\in\{-\la,\la\}$. Recall that $\widehat{F\mu},\widehat{G\mu}$ are real-valued functions, since $F,G\in L^2_{\text{sym}}(\mathbb S^{d-1})$. Let $\xi_0\in\R^d$ be such that $\ab{\widehat{F\mu}(\xi_0)}\neq 0$.
		We lose no generality in assuming that $\widehat{F\mu}(\xi_0)>0$ and $\widehat{G\mu}(\xi_0)>0$, for otherwise we could replace $F$ by $-F$, or $G$ by $-G$. By continuity, there exists $r_0>0$, such that 
		\begin{equation}\label{eq:pointwise-equality-fourier2}
		\widehat{F\mu}(\xi+\xi_0)=\la\widehat{G\mu}(\xi+\xi_0), \text{ for every }\ab{\xi}<r_0.
		\end{equation}
		On the other hand, $\widehat{F\mu}(\xi+\xi_0)=(e^{-iy\cdot \xi_0}F\mu)\widehat{\;}(\xi)$ and $\widehat{G\mu}(\xi+\xi_0)=(e^{-iy\cdot \xi_0}G\mu)\widehat{\;}(\xi)$. The functions $e^{-iy\cdot \xi_0}F$ and $e^{-iy\cdot \xi_0}G$ belong to $L^2_{\text{sym}}(\mathbb S^{d-1})$, and may be expanded in the basis of spherical harmonics,
		\begin{equation}\label{eq:SphHarmoDecomposition}
		e^{-iy\cdot \xi_0}F=\sum_{n=0}^{\infty}\sum_{k=1}^{\gamma(d,n)}a_{n,k}Y_{n,k},\text{ and }  e^{-iy\cdot \xi_0}G=\sum_{n=0}^{\infty}\sum_{k=1}^{\gamma(d,n)}b_{n,k}Y_{n,k}.
		\end{equation}
		Here, $\{Y_{n,k}\}_{k=1}^{\gamma(d,n)}$ denotes a basis for the space of spherical harmonics of degree $n$ in the sphere $\mathbb{S}^{d-1}$, which has dimension  $\gamma(d,n):=\binom{d+n-1}{n}-\binom{d+n-3}{n-2}$, see \cite[Chapter IV]{SW71}. The coefficients $a_{n,k},b_{n,k}$ are complex numbers. Applying the Fourier transform to \eqref{eq:SphHarmoDecomposition}, we find that
		\begin{equation}\label{eq:FormulaFourierTransY}
		\begin{split}
		\widehat{F\mu}(\xi+\xi_0)&=(2\pi)^{\frac d2}\sum_{n=0}^{\infty}\sum_{k=1}^{\gamma(d,n)}a_{n,k}i^{-n}\ab{\xi}^{-\frac d2+1}J_{\frac d2-1+n}(\ab{\xi})Y_{n,k}\Bigl(\frac{\xi}{\ab{\xi}}\Bigr),\\
		\widehat{G\mu}(\xi+\xi_0)&=(2\pi)^{\frac d2}\sum_{n=0}^{\infty}\sum_{k=1}^{\gamma(d,n)}b_{n,k}i^{-n}\ab{\xi}^{-\frac d2+1}J_{\frac d2-1+n}(\ab{\xi})Y_{n,k}\Bigl(\frac{\xi}{\ab{\xi}}\Bigr).
		\end{split}
		\end{equation}
		Using \eqref{eq:pointwise-equality-fourier2} and \eqref{eq:FormulaFourierTransY} together with the orthogonality of the functions $\{Y_{n,k}\}$ in $L^2(\mathbb S^{d-1})$,  we obtain  
		$$a_{n,k}r^{-\frac d2+1}J_{\frac d2-1+n}(r)=\la b_{n,k}r^{-\frac d2+1}J_{\frac d2-1+n}(r), \text{ for every } r\in (0,r_0).$$ 
		In particular,  $a_{n,k}=\la b_{n,k}$. This and \eqref{eq:SphHarmoDecomposition} together imply $F=\la G$.
		\end{proof}
		A similar result to Proposition \ref{prop:symmetrized-extremizers-sphere} holds for a broader class of antipodally symmetric pairs $(\Sigma,\mu)$. 
			Indeed, let $r\in[2,\infty)$ be such that the extension estimate \eqref{eq:generalExtension} holds. Then
	\begin{equation}\label{eq:symetrizationGeneralPairs}
	\sup_{0\neq u\in L^2(\Sigma,\mu)}\frac{\norma{\widehat{(u\mu)}}_{L^r(\R^d)}}{\norma{u}_{L^2(\Sigma,\mu)}}= \sup_{0\neq F\in L_{\text{sym}}^2(\Sigma,\mu)} \frac{\norma{\widehat{F\mu}}_{L^r(\R^d)}}{\norma{F}_{L^2(\Sigma,\mu)}},
	\end{equation}
	with the obvious definition of $L_{\text{sym}}^2(\Sigma,\mu)$. Moreover, if $\mu$ is compactly supported and finite, then any complex extremizer $u$ for \eqref{eq:generalExtension} necessarily coincides with a multiple of a symmetric extremizer $F\in L^2_{\text{sym}}(\Sigma,\mu)$. 
	Regarding the second part of Proposition \ref{prop:symmetrized-extremizers-sphere}, the previous proof used the particular geometry of the sphere, but it can be modified to handle this more general situation. The crux of the matter is the fact that  the Fourier transform of a compactly supported finite measure is real analytic. Indeed, if $\mu$ is a positive, compactly supported finite measure, and $F\in L^2(\Sigma,\mu)$, then, for every $\xi_0\in\R^d$,
	\begin{equation}\label{eq:analyticExp}
	\begin{split}
	\widehat{F\mu}(\xi)&=\int_{\Sigma}e^{-i\xi\cdot y}F(y)\d \mu(y)=\int_{\Sigma}e^{-i(\xi-\xi_0)\cdot y}e^{-i\xi_0\cdot y}F(y)\d \mu(y)\\
	&=\sum_{k=0}^{\infty} \frac{(-i)^k}{k!}\int_{\Sigma}\big((\xi-\xi_0)\cdot y\big)^k e^{-i\xi_0\cdot y}F(y)\d \mu(y),
	\end{split}
	\end{equation}
	where the convergence is locally uniform. To see this, note  the following tail estimate,
	\[\NOrma{\sum_{k=K}^{\infty}\frac{(-i)^k}{k!}\int_{\Sigma}((\xi-\xi_0)\cdot y)^k e^{-i\xi_0\cdot y}F(y)\d \mu(y)}_{L_\xi^\infty(\Omega)}\leq \mu(\Sigma)^{\frac12}\norma{F}_{L^2(\Sigma,\mu)}\sum_{k=K}^{\infty}\frac{s^k}{k!},\]
	  which holds for every compact subset $\Omega\subseteq\R^d$ and every $K\in\N$.
	 Here, $s=\sup_{\xi\in\Omega,\,y\in\Sigma}\ab{\xi-\xi_0}\ab{y}<\infty$. Therefore, the analogue of \eqref{eq:pointwise-equality-fourier} in this setting leads to the corresponding \eqref{eq:pointwise-equality-fourier2}, which by analyticity of \eqref{eq:analyticExp} implies $\widehat{F\mu}=\la\widehat{G\mu}$, and therefore $F=\la G$.

These observations can be of interest when combined with the main result of \cite{FVV11}, which states that complex-valued extremizers exist in the non-endpoint setting, provided $\mu$ is a positive, compactly supported finite measure.
	Important cases of antipodally symmetric pairs $(\Sigma,\mu)$ which have attracted recent attention include the aforementioned case of spheres, together with ellipsoids equipped with surface measure, and the double cone, the one- and the two-sheeted hyperboloids equipped with their natural Lorentz invariant measures, 
	see \cite{FOS17} and the references therein.\\

We end this section with a final remark on the multiplier form of inequality \eqref{eq:OddSharpExtensionFormGenP}.
Consider the Cauchy problem
\begin{equation}\label{eq:pthOrderOddSchroedinger}
\begin{cases}
\partial_t u-\ab{\partial_x}^{p-1}\partial_x u=0,\quad (x,t)\in\R\times\R,\\
u(\cdot,0)=f\in L_x^2(\R),
\end{cases} 
\end{equation}
whose solution can be written in terms of the propagator
\begin{equation}\label{eq:evolutionOdd} u(x,t)=e^{t\ab{\partial_x}^{p-1}\partial_x}f(x)=\frac{1}{2\pi}\int_{\R}e^{ix\xi}e^{it\xi\ab{\xi}^{p-1}}\widehat{f}(\xi)\d\xi.
\end{equation}
In view of \eqref{eq:OddSharpExtensionFormGenP}, and more generally of \cite[Theorem 2.1]{KPV91}, this satisfies the mixed norm estimate
\[ \norma{|D|^{\frac{p-2}r}e^{t\ab{\partial_x}^{p-1}\partial_x}f}_{L^r_tL^s_{x}(\R^{1+1})}\lesssim_{r,s}\norma{f}_{L^2(\R)}, \]
whenever the Lebesgue exponents $r,s$ are such that $\frac{2}{r}+\frac{1}{s}=\frac{1}{2}$. 

In this context, as noted in \cite{FS17,Sh09} for the case $p=3$, it makes sense to distinguish between real-valued and general complex-valued $L^2$ initial data. This is because the evolution $e^{t\ab{\partial_x}^{p-1}\partial_x}$ preserves real-valuedness.
In other words, if $f$ is real valued, then so is $e^{t\ab{\partial_x}^{p-1}\partial_x}f$, for every $t\in\R$. In fact, if $f$ is real-valued, then $\overline{\widehat{f}(-\xi)}=\widehat{f}(\xi)$, and so taking the complex conjugate of \eqref{eq:evolutionOdd} reveals that $\overline{u(x,t)}=u(x,t)$. The operator $|D|^{\frac{p-2}{r}}e^{t\ab{\partial_x}^{p-1}\partial_x}$  is seen to preserve real-valuedness in a similar way.

It is then natural to consider the following family of sharp inequalities, for real- and complex-valued initial data and admissible Lebesgue exponents $r,s$:
\begin{align}
\label{eq:complex-case}
\norma{|D|^{\frac{p-2}r}e^{t\ab{\partial_x}^{p-1}\partial_x}u}_{L^r_tL^s_x(\R^{1+1})}\leq \mathbf{M}_{p,r,s}(\mathbb C)\norma{u}_{L^2(\R)},\\
\label{eq:real-case}
\norma{|D|^{\frac{p-2}r}e^{t\ab{\partial_x}^{p-1}\partial_x}f}_{L^r_{t}L^s_x(\R^{1+1})}\leq \mathbf{M}_{p,r,s}(\R)\norma{f}_{L^2(\R)},
\end{align}
where $u:\R\to\Co$ is complex-valued and $f:\R\to\R$ is real-valued. The study of extremizers for \eqref{eq:complex-case}--\eqref{eq:real-case} in the Airy--Strichartz case $p=3$ has been considered in \cite{FV18, FS17, HS12, Sh09}.
It would be interesting to determine whether the methods developed in the present paper can be adapted to the study of extremizers for \eqref{eq:complex-case}--\eqref{eq:real-case} in the mixed norm case $r\neq s$, so as to obtain an alternative approach to  profile decomposition  or the missing mass method. We do not pursue these matters here. However, we would still like to point out two interesting features of this problem which are easily derived from our previous analysis, and are the content of the following result.
\begin{proposition}\label{prop:two-comments}
	Let $p>1$, and $r,s\in (2,\infty)$ be such that $\mathbf{M}_{p,r,s}(\mathbb C)$ and $\mathbf{M}_{p,r,s}(\R)$ are finite. Then $\mathbf{M}_{p,r,s}(\mathbb C)=\mathbf{M}_{p,r,s}(\R)$. Moreover, if a complex-valued extremizer $u$ for $\mathbf{M}_{p,r,s}(\mathbb C)$ exists, then there exist $\kappa\in\mathbb C$ and a real-valued extremizer $f$ for $\mathbf{M}_{p,r,s}(\mathbb R)$, such that $u=\kappa f$.
\end{proposition}
\noindent The problem of the relationship between arbitrary complex-valued
extremizers and real-valued extremizers has been considered in the literature, see e.g. \cite{CS12b} for the case of the Tomas--Stein inequality on the sphere $\mathbb{S}^2$.
Note the duality with the second statement of Proposition \ref{prop:symmetrized-extremizers-sphere} above. 
\begin{proof}[Proof of Proposition \ref{prop:two-comments}]
The equality $\mathbf{M}_{p,r,s}(\mathbb C)=\mathbf{M}_{p,r,s}(\R)$ follows the same lines of the proof of Lemma \ref{lem:odd-symmetrization}.
To see why this is the case, let $u\in L^2(\R)$ and write $u=f+ig$, where $f$ and $g$ are the real and imaginary parts of $u$, and hence real-valued. Therefore
\begin{gather}
\label{eq:l2-norm-mixed}
\norma{u}_{L^2}^2=\norma{f}_{L^2}^2+\norma{g}_{L^2}^2,\\
\label{eq:pointwise-mixed}
\ab{|D|^{\frac{p-2}r}e^{t\ab{\partial_x}^{p-1}\partial_x}u(x)}^2=\ab{|D|^{\frac{p-2}r}e^{t\ab{\partial_x}^{p-1}\partial_x}f(x)}^2+\ab{|D|^{\frac{p-2}r}e^{t\ab{\partial_x}^{p-1}\partial_x}g(x)}^2, 
\end{gather}
for every $(x,t)\in\R^2$.
If $r,s\geq 2$, then we can use the triangle inequality for the $L^{{s}/{2}}_x$- and the $L^{{r}/{2}}_t$-norms applied to \eqref{eq:pointwise-mixed}, and obtain
\begin{equation}\label{eq:inequality-norm-mixed}
\norma{|D|^{\frac{p-2}r}e^{t\ab{\partial_x}^{p-1}\partial_x}u}_{L^r_tL^s_x}^2\leq \norma{|D|^{\frac{p-2}r}e^{t\ab{\partial_x}^{p-1}\partial_x}f}_{L^r_tL^s_x}^2+\norma{|D|^{\frac{p-2}r}e^{t\ab{\partial_x}^{p-1}\partial_x}g}_{L^r_tL^s_x}^2. 
\end{equation}
Without loss of generality, assume that $f,g$ are not identically zero. Reasoning as in the proof of Lemma \ref{lem:odd-symmetrization} yields
\begin{equation}\label{eq:complex-real-inequality}
\frac{\norma{|D|^{\frac{p-2}r}e^{t\ab{\partial_x}^{p-1}\partial_x}u}_{L^r_tL^s_x}^2}{\norma{u}_{L^2}^2}
\leq \max\Big\{\frac{\norma{|D|^{\frac{p-2}r}e^{t\ab{\partial_x}^{p-1}\partial_x}f}_{L^r_tL^s_x}^2}{\norma{f}_{L^2}^2},\frac{\norma{|D|^{\frac{p-2}r}e^{t\ab{\partial_x}^{p-1}\partial_x}g}_{L^r_tL^s_x}^2}{\norma{g}_{L^2}^2}\Big\},
\end{equation}
and therefore $\mathbf{M}_{p,r,s}(\mathbb C)\leq \mathbf{M}_{p,r,s}(\R)$. The reverse inequality is immediate. 
We gratefully acknowledge recent personal communication with  R. Frank and J. Sabin, who independently arrived at a similar conclusion in the context of \cite{FS17}.

We proceed to show that an arbitrary complex-valued extremizer for $\mathbf{M}_{p,r,s}(\mathbb C)$ necessarily coincides with a constant multiple of a real-valued extremizer for $\mathbf{M}_{p,r,s}(\R)$. 
Let $r,s\in(2,\infty)$, and suppose that $u$ is a complex-valued extremizer for $\mathbf{M}_{p,r,s}(\mathbb C)$, which we express as the sum of its real and imaginary parts, $u=f+ig$.  An inspection of the chain of inequalities leading to \eqref{eq:complex-real-inequality} shows that one of the following alternatives must hold:
\begin{itemize}
	\item $g=0$ and $u=f$ is a real-valued extremizer, or
	\item $f=0$, $u=ig$, and $g$ is a real-valued extremizer, or
	\item $f,g$ are both not identically zero, and 
	\begin{equation}\label{eq:real-imaginary-part-maximize}
	\frac{\norma{|D|^{\frac{p-2}r}e^{t\ab{\partial_x}^{p-1}\partial_x}f}_{L^r_t L^s_x}^2}{\norma{f}_{L^2}^2}=\frac{\norma{|D|^{\frac{p-2}r}e^{t\ab{\partial_x}^{p-1}\partial_x} g }_{L^r_t L^s_x}^2}{\norma{g}_{L^2}^2}=\mathbf{M}_{p,r,s}(\R),
	\end{equation}
	so that $f,g$ are real-valued extremizers.
\end{itemize}
It suffices to analyze the latter case. An inspection of the chain of inequalities leading to \eqref{eq:inequality-norm-mixed} shows that equality must hold in both applications of the triangle inequality. Since $r,s\in(2,\infty)$, this implies the existence of $\la>0$, such that
\begin{equation}\label{eq1}
\ab{|D|^{\frac{p-2}r}e^{t\ab{\partial_x}^{p-1}\partial_x}f(x)} = \la \ab{|D|^{\frac{p-2}r}e^{t\ab{\partial_x}^{p-1}\partial_x} g(x)}, \text{ for almost every } (x,t)\in\R^2. 
\end{equation}
Equality in \eqref{eq:real-imaginary-part-maximize} then implies $\norma{f}_{L^2}=\la\norma{g}_{L^2}$. 
By squaring \eqref{eq1}, and applying the Fourier transform, the equality of the resulting convolutions can be recast as follows:
\begin{multline}\label{eq:deltaEquality}
 \int_{\R^2}\widehat{f}(y_1)\widehat{f}(y_2)\ddirac{t-\psi(y_1)-\psi(y_2)}\ddirac{x-y_1-y_2}\ab{y_1 y_2}^{\frac{p-2}r}\d y_1\d y_2\\
 =\la^2\int_{\R^2}\widehat{g}(y_1)\widehat{g}(y_2)\ddirac{t-\psi(y_1)-\psi(y_2)}\ddirac{x-y_1-y_2}\ab{y_1y_2}^{\frac{p-2}r}\d y_1\d y_2,
\end{multline}
where $(x,t)\in\R^2$ and $\psi(y):=y\ab{y}^{p-1}$. Considering points $(x,t)$ in the interior of the support of the convolution measure $\mu_p\ast \mu_p$, i.e. satisfying $t>2\psi(\tfrac{x}{2})$ for $x>0$, and $t<2\psi(\tfrac{x}{2})$ for $x<0$, we see that there exists a unique positive solution 
$\alpha=\alpha(x,t)>0$ of
\begin{equation}\label{eq:equation-alpha}
t=\psi(\tfrac{x}{2}-\alpha(x,t))+\psi(\tfrac{x}{2}+\alpha(x,t)),
\end{equation}
and hence that the system of equations $t=\psi(y_1)+\psi(y_2),\, x=y_1+y_2$ has unique solutions 
$$(y_1,y_2)\in\{(\tfrac{x}{2}-\alpha(x,t),\tfrac{x}{2}+\alpha(x,t)),(\tfrac{x}{2}+\alpha(x,t),\tfrac{x}{2}-\alpha(x,t))\}.$$ 
From \eqref{eq:deltaEquality} and a similar reasoning to that of \cite[Proposition 2.1 and Remark 2.3]{OSQ18}, it then follows that
\[ \widehat{f}(\tfrac{x}{2}-\alpha(x,t))\widehat{f}(\tfrac{x}{2}+\alpha(x,t))=\la^2\widehat{g}(\tfrac{x}{2}-\alpha(x,t))\widehat{g}(\tfrac{x}{2}+\alpha(x,t)), \]
for almost every $(x,t)\in\supp(\mu_p\ast\mu_p)$. 
Alternatively, the latter identity follows by considering the analogue of formula \eqref{eq:altFormulaConv} obtained in the case of even curves, which by the previous discussion  applies to the present scenario as well.
This yields
\begin{equation}\label{eq:pointwiseEqualityProd}
\widehat{f}(x)\widehat{f}(x')=\la^2\widehat{g}(x)\widehat{g}(x'),
\end{equation}
for almost every $(x,x')\in\R^2$. 
As $\widehat{f},\widehat{g}$ belong to $L^2(\R)$, we may integrate over any compact subset $I\subset\R$ in both variables $x,x'$, and obtain
\begin{equation}\label{eq:IntergralOverI}
\Bigl(\int_I \widehat{f}(x)\d x\Bigr)^2=\la^2\Bigl(\int_I \widehat{g}(x)\d x\Bigr)^2.
\end{equation}
Choose a compact subset $J\subset\R$ for which $\int_{J} \widehat{g}(x)\d x\neq 0$.
From \eqref{eq:IntergralOverI}, we have that
	\begin{equation}\label{eq:IntergralOverJ}
	 \int_J \widehat{f}(x)\d x=\la \int_J \widehat{g}(x)\d x, \text{ or } \int_J \widehat{f}(x)\d x=-\la \int_J \widehat{g}(x)\d x.
	 \end{equation}
 Integrating both sides of \eqref{eq:pointwiseEqualityProd} over $x'\in J$, one infers from \eqref{eq:IntergralOverJ} that either $\widehat{f}=\la \widehat{g}$ or $\widehat{f}=-\la \widehat{g}$, and therefore that either $f=\la g$ or $f=-\la g$. 
 The conclusion is that there exists $\la>0$ such that either $u=(\la+i)g$ or $u=(-\la+i)g$, and so  $u$ is a constant multiple of a real-valued extremizer, as desired.
 \end{proof}

\section*{Acknowledgements}
We are grateful to Rupert Frank and Julien Sabin for discussions surrounding symmetric complex- and real-valued extremizers,
and to Rupert Frank for clarifying some historical remarks.

\appendix
\section{Concentration-compactness}
This appendix consists of a useful observation regarding Lions' concentration-compactness lemma \cite{Li84}.
Let us start with some general considerations.
Let $(X,\mathcal B,\mu)$ be a measure space with a {\it distinguished} point $\bar{x}\in X$, such that $\{\bar{x}\}\in \mathcal B$ and 
$\mu(\{\bar{x}\})=0$. Set $X_{\bar{x}}:=X\setminus\{\bar{x}\}$. 
Let $\vrho:X_{\bar{x}}\times X_{\bar{x}}\to[0,\infty)$ be a  pseudometric on $X_{\bar{x}}$, i.e.
a measurable 
function on $X_{\bar{x}}\times X_{\bar{x}}$ satisfying $\vrho(x,x)=0$, $\vrho(x,y)=\vrho(y,x)$, and 
$\vrho(x,y)\leq 
\vrho(x,z)+\vrho(z,y)$, for every $x,y,z\in X_{\bar{x}}$. 
Define the ball of center $x\in X_{\bar{x}}$ and radius $r\geq 0$, $B(x,r):=\{y\in X_{\bar{x}}\colon \vrho(x,y)\leq r\}$, and its complement 
$B(x,r)^\complement:=X\setminus B(x,r)$. It is clear that  $X_{\bar{x}}=\bigcup_{r\geq 0}B(x,r)$, for every $x\neq \bar{x}$. 
We have the following concentration-compactness result, which should be compared to  \cite[Lemma I.1]{Li84}.
\begin{proposition}\label{prop:metric-concentration-compactness-lemma}
	Let $(X,\mathcal B, \mu),\bar{x}\in X,\vrho:X_{\bar{x}}\times X_{\bar{x}}\to[0,\infty)$ be as above. 
	Let $\{\rho_n\}$ be a sequence in $L^1(X,\mu)$ 
	satisfying: 
	\[\rho_n\geq 0\text{ in }X,\qquad\int_{X}\rho_n\d\mu=\la,\]
	where $\la>0$ is fixed.
        Then there exists a subsequence $\{\rho_{n_k}\}$ satisfying 
	one of the following three possibilities:
	\begin{enumerate}
		\item [(i)] (Compactness) There exists $\{x_k\}\subset X_{\bar{x}}$ such that $\rho_{n_k}(\cdot+x_k)$ is tight i.e.:
				\[\forall \eps>0,\, \exists R<\infty:\;\int_{B(x_k,R)}\rho_{n_k} \d\mu\geq \la-\eps;\]
		\item [(ii)] (Vanishing) 
		$\ds\lim\limits_{k\to\infty}\sup_{y\in X_{\bar{x}}}\int_{B(y,R)}\rho_{n_k}\d\mu=0$,
		for all $R<\infty$;
		\item [(iii)] (Dichotomy) There exists $\alpha\in(0,\la)$ with the following property.
		For every $\eps>0$, there exist $R\in [0,\infty)$, $k_0\geq 1$, and
		nonnegative functions $\rho_{k,1},\rho_{k,2}\in L^1(X,\mu)$ such that, for every $k\geq k_0$,
		\begin{gather*}
		\norma{\rho_{n_k}-(\rho_{k,1}+\rho_{k,2})}_{L^1(X)}\leq \eps,\;\;\;
		\abs{\int_{X}\rho_{k,1}\d\mu-\alpha}\leq
		\eps,\;\;\,\abs{\int_{X}\rho_{k,2}\d\mu-(\la-\alpha)}\leq\eps,\\
		\supp(\rho_{k,1})\subseteq B(x_k,R), \text{ and }
		\supp(\rho_{k,2})\subseteq B(x_k,R_k)^\complement,
		\end{gather*}
		for certain sequences $\{x_k\}\subset X_{\bar{x}}$, $\{R_k\}\subset [0,\infty)$ with
		$R_k\to\infty$, as $k\to\infty$.
	\end{enumerate}
\end{proposition}

\noindent	The proof of Proposition \ref{prop:metric-concentration-compactness-lemma} parallels that of \cite[Lemma I.1]{Li84}, and proceeds via analysis of the sequence of 
{\it concentration functions} 
	\[Q_n\colon[0,\infty)\to\R,\quad Q_n(t):=\sup\limits_{x\in 
	X_{\bar{x}}}\int_{B(x,t)}\rho_n \d \mu.\]
The sequence $\{Q_n\}$ consists of nondecreasing, nonnegative, uniformly bounded functions on $[0,\infty)$, which satisfy $Q_n(t)\to\la$, as $t\to\infty$, since $\mu(\{\bar{x}\})=0$. Very briefly, the argument goes as follows. By the Helly Selection Principle, there exists a subsequence $\{n_k\}\subset \N$ and a nondecreasing, nonnegative function $Q\colon[0,\infty)\to\R$, such that $Q_{n_k}(t)\to Q(t)$, as $k\to\infty$, \textit{for every} $t\geq 0$. Set $\alpha:=\lim_{t\to\infty}Q(t)\in[0,\la]$, and note that:
\begin{itemize}
	\item If $\alpha=0$, then $Q\equiv0$. This translates into the vanishing condition at once.
	\item If $\alpha=\la$, then compactness occurs.
	\item If $0<\alpha<\la$, then dichotomy occurs. 
	In this case, the functions $\rho_{k,1},\rho_{k,2}$ are given by $\rho_{k,1}=\rho_{n_k}\one_{B(x_k,R)}$ and $\rho_{k,2}=\rho_{n_k}\one_{B(x_k,R_k)^\complement}$.
\end{itemize}
We omit further details and refer the interested reader to \cite{Li84}.

\smallskip

When applying Proposition \ref{prop:metric-concentration-compactness-lemma} to the study of extremizing sequences for \eqref{eq:SharpExtensionFormGenP}, the desirable outcome (with a view towards obtaining concentration at a point under the hypotheses of  Proposition \ref{prop:small-cap-implies-concentration}) is \textit{compactness} or \textit{vanishing}. Therefore the possibility of \textit{dichotomy} needs to be discarded. 
To this end, Lions proposes the {\it strict superadditivity condition} \cite[Section I.2]{Li84}, which in the present setting can be recast as follows. Define
\begin{equation}\label{eq:maximumLions}
I_\la:=\sup\{ \norma{\mathcal{E}_p(f)}_{L^6(\R^2)}^6\colon \norma{f}_{L^2(\R)}^2=\la \}.
\end{equation}
The quantity $I_\la$ is said to satisfy the strict superadditivity condition if, for every $\la>0$, 
\begin{equation}\label{eq:strongSupCond}
I_{\la}>I_{\alpha}+I_{\la-\alpha},\,\text{ for every }\alpha\in (0,\la).
\end{equation}
In our case, $\mathcal E_p$ is a linear operator, and so $I_{\la}=\la^3 I_1=\la^3\mathbf{E}_p^6$. Thus \eqref{eq:strongSupCond} translates into the elementary numerical inequality $\la^3>\alpha^3+(\la-\alpha)^3$, which holds for every $\la>0$ and $\alpha\in (0,\la)$.
As seen in the proof of Proposition \ref{prop:small-cap-implies-concentration}, it is condition \eqref{eq:strongSupCond} (applied with $\la=1$) which ensures that dichotomy does not occur.
A similar condition in a more general context is used by Lieb \cite[Lemma 2.7]{Li83}.

\section{Revisiting Br\'ezis--Lieb} 
In this appendix, we prove a useful variant of \cite[Proposition 1.1]{FVV11}, which in turn relies on the Br\'ezis--Lieb lemma \cite{BL83}.
 \cite[Proposition 1.1]{FVV11} states that, in the compact setting, the only obstruction to the strong
convergence of an extremizing sequence is weak convergence to zero.
In the non-compact setting, it is in general non-trivial to verify condition (iv) of \cite[Proposition 1.1]{FVV11}. 
To overcome this difficulty, various arguments using Sobolev embeddings and the  Rellich--Kondrachov compactness theorem  have been employed in  \cite{COS17,FVV12,Qu13,Qu17}. 
In our case,  it is not clear how such an argument would go.
 Instead we take a different route, and argue that condition (iv)  from \cite[Proposition 1.1]{FVV11} can be replaced by uniform decay of the $L^2$-norm, in a sense compactifying the space in question. The following is a precise formulation of this idea. 
\begin{proposition}\label{prop:new-fvv}
Given $p>1$, consider the Fourier extension operator 
	 $\mathcal{E}_p\colon L^2(\R)\to L^6(\R^2)$ defined in \eqref{eq:DefFSigmaHat}.
	Let $\{f_n\}\subset L^2(\R)$, and let $\Theta:[1,\infty)\to(0,\infty)$ with $\Theta(R)\to 0$, as 
	$R\to\infty$, be such that:
	\begin{enumerate}
		\item[(i)] $\norma{f_n}_{L^2(\R)}=1$, for every $n\in\N$;
		\item[(ii)] $\lim_{n\to\infty}\norma{\mathcal{E}_p(f_n)}_{L^6(\R^2)}={\bf E}_p$; 
		\item[(iii)] $f_n\weak f\neq 0$, as $n\to\infty$;
		\item[(iv)] $\norma{f_n}_{L^2([-R,R]^\complement)}\leq \Theta(R)$, for every $n\in\N$ and $R\geq 1$.
	\end{enumerate}
	Then $f_n\to f$ in $L^2(\R)$, as $n\to\infty$. In particular, $\norma{f}_{L^2(\R)}=1$ and 
	$\norma{\mathcal{E}_p(f)}_{L^6(\R^2)}={\bf E}_p$, and so $f$ is an extremizer of \eqref{eq:SharpExtensionFormGenP}.
\end{proposition}

\noindent This variant was already observed in \cite[Proposition 2.31]{Qu12} for the case of the cone, and the proof follows similar lines to that of \cite[Proposition 1.1]{FVV11}. 
Note that the function $\Theta$ may depend on the sequence $\{f_n\}$, but \textit{not} on $n$. 
The following proof is inspired by \cite[Proposition 2.2]{FLS16}.

\begin{proof}[Proof of Proposition \ref{prop:new-fvv}]
	Denote 
	$r_n:=f_n-f$. Then 
	$r_n\weak 0$, as 
	$n\to\infty$, and thus $m:=\lim_{n\to\infty}\norma{r_n}_{L^2}^2$ exists and satisfies 
	$1=\norma{f}_{L^2}^2+m$. 
	Given $R>0$, decompose 
	$$r_n=r_n\one_{[-R,R]}+r_n\one_{[-R,R]^\complement}=:r_{n,1}+r_{n,2}.$$
	Since the support of 
	$r_{n,1}$ is compact and $r_{n,1}\weak 0$, as $n\to\infty$, then $\mathcal{E}_p(r_{n,1})\to 0$ pointwise a.e. in $\R^2$, as 
	$n\to\infty$. 
	On the other hand, from condition (iv) we have that 
	\begin{equation}\label{eq:BoundsTrn2}
	\norma{\mathcal{E}_p(r_{n,2})}_{L^6}\leq {\bf E}_p(\Theta(R)+\norma{f}_{L^2([-R,R]^\complement)}),
	\end{equation}
	for every $R\geq 1$.
	This upper bound is independent of $n$, and tends to $0$ as $R\to\infty$.
	We have $\mathcal{E}_p(f_n-r_{n,2})=\mathcal{E}_p(f)+\mathcal{E}_p(r_{n,1})$, and 
	$\norma{\mathcal{E}_p(f_n-r_{n,2})}_{L^6}\leq 
	{\bf E}_p(1+\Theta(R)+\norma{f}_{L^2([-R,R]^\complement)})$ is uniformly 
	bounded in $n$. Since $\mathcal{E}_p(f_n-r_{n,2})\to \mathcal{E}_p(f)$ pointwise a.e. in $\R^2$, as $n\to\infty$, we can invoke the 
	Br\'ezis--Lieb lemma \cite{BL83} and obtain
	\[ 
	\norma{\mathcal{E}_p(f_n-r_{n,2})}_{L^6}^6=\norma{\mathcal{E}_p(f)}_{L^6}^6+\norma{\mathcal{E}_p(r_{n,1})}_{L^6}^6+o(1),\text{
	 as 
	}n\to\infty. \]
	It follows that $\mu:=\limsup_{n\to\infty}\norma{\mathcal{E}_p(r_{n,1})}_{L^6}^6$ and $\la:= 
	\limsup_{n\to\infty}\norma{\mathcal{E}_p(f_n-r_{n,2})}_{L^6}^6$ 
	satisfy 
	\[ \la= \norma{\mathcal{E}_p(f)}_{L^6}^6 +\mu. \]
	Since $\norma{\mathcal{E}_p(r_{n,1})}_{L^6}^6\leq {\bf E}_p^6\norma{r_{n,1}}_{L^2}^6\leq 
	{\bf E}_p^6\norma{r_{n}}_{L^2}^6$, we have $\mu\leq{\bf E}_p^6m^3$. Therefore
	\[ \la= \norma{\mathcal{E}_p(f)}_{L^6}^6+\mu\leq  
	\norma{\mathcal{E}_p(f)}_{L^6}^6+{\bf E}_p^6(1-\norma{f}_{L^2}^2)^3.\]
	Thus, 
	replacing the definition of $\la$, 
	we have proved
	\begin{equation}\label{eq:has-been-proved}
	\limsup_{n\to\infty}\norma{\mathcal{E}_p(f_n-r_{n,2})}_{L^6}^6\leq 
	\norma{\mathcal{E}_p(f)}_{L^6}^6+{\bf E}_p^6(1-\norma{f}_{L^2}^2)^3,
	\end{equation}
	for every $R\geq 1$. 
	Now, 
	$\norma{\mathcal{E}_p(f_n-r_{n,2})}_{L^6}\geq \norma{\mathcal{E}_p(f_n)}_{L^6}-\norma{\mathcal{E}_p(r_{n,2})}_{L^6}$ and
	$\norma{\mathcal{E}_p(r_{n,2})}_{L^6}$ is bounded above as quantified by \eqref{eq:BoundsTrn2}. Thus
	\[ \limsup_{n\to\infty}\norma{\mathcal{E}_p(f_n-r_{n,2})}_{L^6}\geq 
	{\bf E}_p-{\bf E}_p(\Theta(R)+\norma{f}_{L^2([-R,R]^\complement)}), \]
	for every $R\geq 1$. Using this together with \eqref{eq:has-been-proved}, and letting 
	$R\to\infty$, yields 
	\[{\bf E}_p^6\leq 
	\norma{\mathcal{E}_p(f)}_{L^6}^6+{\bf E}_p^6(1-\norma{f}_{L^2}^2)^3. \]
	By the elementary inequality $(1-t)^3\leq 1-t^3$, valid for every $t\in[0,1]$, we then have
	\[ {\bf E}_p^6\leq 
	\norma{\mathcal{E}_p(f)}_{L^6}^6+{\bf E}_p^6(1-\norma{f}_{L^2}^6). \]
	Since the reverse inequality holds by definition, we conclude that $f$ is an extremizer. 
	Moreover, since $f\neq 0$ and the elementary inequality is strict unless 
	$t\in\{0,1\}$, we conclude that $\norma{f}_{L^2}=1$.
	This completes the proof of the proposition.
\end{proof}

\end{document}